\documentclass[12pt]{amsart}

\usepackage{amsfonts, amsmath, amscd, amssymb, amsbsy, amsthm,  amstext, amsopn, verbatim,multicol, hyperref}
\usepackage[enableskew]{youngtab}
\usepackage{fullpage}
\usepackage[mathscr, mathcal]{eucal}
\usepackage[all]{xy}
\usepackage{mathrsfs}
\usepackage{latexsym}
\usepackage{colortbl}
\usepackage{enumerate}
\usepackage{longtable}

\DeclareMathOperator{\Spec}{\mathrm{Spec}}

\DeclareMathOperator{\gl}{\mathrm{GL}}

\DeclareMathOperator{\SL}{\mathrm{SL}}
\DeclareMathOperator{\g}{\mathfrak{g}}

\DeclareMathOperator{\edo}{End}

\DeclareMathOperator{\rank}{rank}

\DeclareMathOperator{\id}{Id}

\DeclareMathOperator{\ext}{Ext}

\DeclareMathOperator{\md}{-mod}
\DeclareMathOperator{\prj}{-proj}

\DeclareMathOperator{\Hom}{\mathrm{Hom}}
\DeclareMathOperator{\End}{\mathrm{End}}

\DeclareMathOperator{\rk}{\mathrm{rk}}
\DeclareMathOperator{\gr}{\mathrm{gr}}
\DeclareMathOperator{\Ann}{\mathrm{Ann}}

\DeclareMathOperator{\ad}{\mathrm{ad}}

\DeclareMathOperator{\quot}{\mathrm{Quot}}

\newcommand{\la}{\lambda}
\newcommand{\prtn}[1]{\mathscr{P}(#1)}
\newcommand{\prt}[2]{\mathscr{P}_{#1}(#2)}
\newcommand{\col}[1]{{\sf Col}(#1)}
\newcommand{\Opar}[1]{\mathcal{O}_{#1}^{\mbox{\tiny{\sf par}}}}
\newcommand{\tOpar}[1]{\tilde{\mathcal{O}}_{#1}^{\mbox{\tiny{\sf par}}}}

\newcommand{\C}{\mathbb{C}}

\newcommand{\lan}{\langle}
\newcommand{\ran}{\rangle}
\newcommand{\OO}{\mathcal{O}}

\newcommand{\irr}[1]{\textsf{Irrep}(#1)}

\newcommand{\Lmod}[1]{#1\text{-}{\mathsf{mod}}}

\newcommand{\erem}{\hfill$\lozenge$\end{remark}}
\newcommand{\beq}{\begin{equation}\label}
\newcommand{\eeq}{\end{equation}}

\newcommand{\f}[1]{\mathfrak{#1}}
\newcommand{\scr}[1]{\mathscr{#1}}

\newcommand{\GL}{GL}

\renewcommand{\mid}{\enspace\big|\enspace}

\newcommand{\M}{{\mathcal{M}}}

\renewcommand{\part}{{\f P}}

\newcommand{\eps}{\epsilon}

\newcommand{\h}{{\mathfrak{h}}}

\newcommand{\p}{{\f p}}

\newcommand{\Z}{{\mathbb{Z}}}
\newcommand{\La}{\Lambda}

\newcommand{\Coh}{{\operatorname{Coh}}}

\newcommand{\BC}{{\mathbb C}}

\newcommand{\BZ}{{\mathbb Z}}

\newcommand{\res}{{\operatorname{res}}}

\newcommand{\Refl}{\mathcal{A}}
\newcommand{\param}{\mathfrak{c}}
\newcommand{\Res}{\operatorname{Res}}
\newcommand{\Ind}{\operatorname{Ind}}
\newcommand{\Hecke}{\mathcal{H}}
\newcommand{\OCat}{\mathcal{O}}
\newcommand{\Fun}{\operatorname{Fun}}
\newcommand{\Funct}{\mathcal{F}}
\newcommand{\fd}{\operatorname{fd}}
\newcommand{\Sp}{\operatorname{Sp}}
\newcommand{\Gr}{{G}}

\newcommand{\mods}{\operatorname{mod}}
\newcommand{\Str}{\mathcal{O}}
\newcommand{\Pro}{\mathcal{P}}
\newcommand{\Dcal}{\mathcal{D}}
\newcommand{\W}{\mathbb{A}}
\newcommand{\z}{\mathfrak{z}}

\def\euu{{\mathsf{eu}}}

\def\C{{\mathbb{C}}}

\def\gl{{\mathfrak{g}\mathfrak{l}}}

\def\glinfty{{\mathfrak{g}\mathfrak{l}}_\infty}

\makeatletter
\renewcommand{\subsection}{\@startsection{subsection}{2}{0pt}{-3ex plus -1ex minus -0.2ex}{-2mm
plus -0pt minus -2pt}{\normalfont\bfseries}} \makeatother

\numberwithin{equation}{subsection}
\newtheorem*{theorem}{Theorem}
\newtheorem*{proposition}{Proposition}
\newtheorem*{lemma}{Lemma}

\newtheorem*{corollary}{Corollary}

\newtheorem*{theoremA}{Theorem A}
\newtheorem*{theoremB}{Theorem B}
\newtheorem*{theoremC}{Theorem C}

\newtheorem{hypothesis}{Hypothesis}

\newtheorem{conjecture}{Conjecture}

\theoremstyle{remark}
 \newtheorem*{remark}{Remark}

\makeatletter
\newcommand{\cmpl}[1]{%
    \sbox\z@{$#1$}%
    \dimen@=\wd\z@
    \advance \dimen@ -\strip@pt\fontdimen\@ne\textfont\@ne \ht\z@
    \setbox\tw@=\hb@xt@\dimen@{}%
    \ht\tw@=\ht\z@ \dp\tw@=\dp\z@
    \box\z@
    \llap{$\overline{\box\tw@}$}%
}
\makeatother

\begin{document}

\title{On category $\mathcal{O}$ for cyclotomic rational Cherednik
algebras}
\author{Iain Gordon}
\author{Ivan Losev}
\address{School of Mathematics and Maxwell Institute of Mathematics, University of Edinburgh, Edinburgh, EH9 3JZ, U.K. (I.G.); Department
of Mathematics, Northeastern University, Boston MA 02115 USA (I.L.)} \email{igordon@ed.ac.uk, i.loseu@neu.edu}
\begin{abstract} We study equivalences for category $\mathcal{O}_p$ of the rational Cherednik algebras ${\bf H}_p$ of type $G_{\ell}(n) = (\mu_{\ell})^n\rtimes \mathfrak{S}_n$: a highest weight
equivalence between $\mathcal{O}_p$ and $\mathcal{O}_{\sigma(p)}$ for $\sigma\in \mathfrak{S}_{\ell}$ and an action of $\mathfrak{S}_{\ell}$ on an explicit non-empty Zariski open set of parameters $p$; a derived equivalence between $\mathcal{O}_p$ and $\mathcal{O}_{p'}$ whenever $p$ and $p'$ have integral difference; a highest weight equivalence between $\mathcal{O}_p$ and a parabolic category $\mathcal{O}$ for the general linear group, under a non-rationality assumption on the parameter $p$. As a consequence, we confirm special cases of conjectures of Etingof and of Rouquier.
\end{abstract}

\thanks{We thank R. Bezrukavnikov, J. Brundan, P. Etingof, A. Kleshchev,
A. Kuznetsov and C. Stroppel for helpful discussions, the Research in Pairs programme at Oberwolfach where our work together began, and the Hausdorff Institute for Mathematics for the special semester on Representation Theory. The first author is grateful for the full financial support of EPSRC grant EP/G007632. The second author is supported by the NSF grant DMS-0900907.}
\maketitle
\tableofcontents

\section{Introduction}
Category $\mathcal{O}$ for the rational Cherednik algebra of type $G_{\ell}(n) = (\mu_{\ell})^n\rtimes \mathfrak{S}_n$ (where $\mu_{\ell}$ is the group of $\ell^{\text{th}}$ roots of unity in $\C^\times$) has a particularly rich structure. It is a highest weight category, constructed similarly to the Bernstein-Gelfand-Gelfand category $\mathcal{O}$ for a complex reductive Lie algebra, and as such it is a highest weight covering of the Hecke algebra of type $G_{\ell}(n)$ via the ${\sf KZ}$ functor.  Furthermore it can be realised as a certain category of $\mathcal{D}$-modules on a variety of representations of an (extension of an) affine quiver of type $A$, and because of this it is connected with the combinatorics of Hilbert schemes and the $n!$ theorem. Finally, taking the sum of all category $\mathcal{O}$'s for $G_{\ell}(n)$ as $n$ runs through $\Z_{\geq 0}$ categorifies higher level Fock spaces. 

We are first going to exploit this microlocal angle, particularly the  connection between the rational Cherednik algebras of type $G_{\ell}(n)$ and the deformation quantizations of a symplectic resolution of the orbit space $(\h\times \h^*)/G_{\ell}(n)$ where $\h$ is the reflection representation and $G_{\ell}(n)$ acts on $\h\times \h^*$ diagonally. \begin{enumerate}
\item[(a)] The symplectic resolution has a Weyl group action which lifts to its quantizations and in turn passes to non-trivial symmetries of the spherical subalgebras of the Cherednik algebras. Combined with an intrinsic construction of $\mathcal{O}$ for the spherical subalgebra, which is motivated by a similar construction for finite $W$-algebras, these symmetries produce the highest weight equivalences between $\mathcal{O}_p$ and $\mathcal{O}_{\sigma(p)}$.
\item[(b)] There is a tilting bundle $\mathcal{P}$ on the symplectic resolution which induces a derived equivalence between coherent sheaves on the resolution and finitely generated $G_{\ell}(n)$-equivariant modules over $\C[\h\times \h^*]$. We show that this lifts to the noncommutative setting, producing a derived equivalence between a category of torus equivariant sheaves for the deformation quantizations and categories of torus equivariant representations for the Cherednik algebras. Tensoring by appropriate line bundles induces equivalences between the categories of sheaves for different deformation quantizations, and hence derived equivalences for the Cherednik algebra at different parameters. Via a sequence of reductions it is possible to recover the derived category of $\mathcal{O}$ from these category of equivariant representations and so we deduce the derived equivalences between $\mathcal{O}_p$ and $\mathcal{O}_{p'}$.
\end{enumerate}
When we enforce some non-rationality on the parameter $p$ the Hecke algebra of $G_{\ell}(n)$ becomes isomorphic to its own degeneration. As higher level Schur-Weyl duality shows that parabolic category $\mathcal{O}$ for the general linear group produces a highest weight cover of the degenerate Hecke algebra of type $G_{\ell}(n)$, we are then able to compare $\mathcal{O}_p$ and the parabolic category $\mathcal{O}$, showing they are equivalent. The equivalence is explicit enough to give a combinatorial characterization of the support of any irreducible representation in $\mathcal{O}_p$, and in particular a precise classification of all the finite dimensional representations for the Cherednik algebra (with the non-rationality parameter restriction).
\medskip

We now give a more detailed summary of our results.
\subsection{Rational Cherednik algebras}
In \cite{EG} Etingof and Ginzburg introduced symplectic reflection algebras. These algebras arise in the framework of the deformation theory as follows.
Let ${V}$ be a symplectic vector space over $\BC$ with symplectic form $\omega$ and let $\Gr$ be a finite subgroup of $\Sp({V})$. Let $S$ denote the set of symplectic reflections in $\Gr$, that is
the set of all $\gamma\in \Gr$ such that $\rk_V(\gamma-\id)=2$. Decompose $S$ into the union
$\bigsqcup_{i=0}^r S_i$ of $\Gr$-conjugacy classes.
Pick complex numbers $(c_0,c_1,\ldots,c_r)$, one for each conjugacy class in $S$, and
another complex number $h$. Set $p:=(h,c_0,\ldots,c_r)$ and define the algebra
${\bf H}_p={\bf H}_p(\Gr,{V})$ as the quotient of the smash-product $T({V})\# \Gr$ of the tensor algebra on $V$ and the group $\Gr$ by the relations
\begin{equation}\label{eq:SRA_relation}
[u,v]=h \omega(u,v)-2\sum_{i=1}^r c_i\sum_{\gamma\in S_i}\omega_\gamma(u,v)\gamma,
\end{equation}
for $u,v\in V$,
where $\omega_\gamma$ denotes the skew-symmetric form that coincides with $\omega$
on the two-dimensional space $\operatorname{im}(\gamma-\operatorname{id})$ and whose kernel
is $\ker(\gamma-\operatorname{id})$. Under a mild restriction on $\Gr$, namely the absence
of proper symplectic $\Gr$-stable subspaces in ${V}$, the algebras ${\bf H}_p$
exhaust {\it filtered} deformations of the graded algebra $S({V})\#\Gr$.

Let $\h$ be a complex vector space and $W\subset \GL(\h)$ be a finite linear group
generated by complex reflections -- the elements $\gamma\in W$ with $\operatorname{rk}_{\h}(\gamma -\operatorname{id})=1$.
The  space ${V}=\h\times \h^*$ carries a natural symplectic form, and for $\Gr$
we take the image of $W$ in $\Sp({V})$ under the diagonal action. The corresponding symplectic reflection algebra
is usually referred to as a {\it rational Cherednik algebra}. When $h\neq 0$ the structure and representation theory of the rational Cherednik algebra are reminiscent
to those of the universal enveloping algebra of a reductive Lie algebra. We have a triangular decomposition ${\bf H}_p=S(\h^*)\otimes \BC W\otimes S(\h)$
and so can consider the category $\OCat_p$ of all left ${\bf H}_p$-modules that are finitely generated
over $S(\h^*)$ and locally nilpotent for the action of $\h\subset S(\h)$.  This category has a highest weight
structure where the standard modules $\Delta(\lambda)$ and simple modules $L(\lambda)$ are parameterized by irreducible $W$-modules $\lambda$, \cite{GGOR}.

Now let $\Gamma\subset \SL_2(\BC)$ be a finite subgroup, and form the wreath product $\Gr=\Gamma^n\rtimes \mathfrak{S}_n$ for some positive integer $n$, where the symmetric group $\mathfrak{S}_n$ acts on $\Gamma^n$ by permutations. This group naturally embeds into
$\Sp({V})$ where ${V}=(\BC^2)^{\oplus n}$. In these cases the symplectic reflection algebras are related to differential operators on the quiver representation spaces and to symplectic resolutions of quotient singularities, see \cite{quant} for instance.

The intersection of these two examples are the groups $G(\ell ,1,n) = (\mu_\ell)^n\rtimes \mathfrak{S}_n$ acting on $(\C^n)^{\oplus 2} = (\C^2)^{\oplus n}$. This is the case that we are mostly interested in here and we are going to study the corresponding categories $\mathcal{O}_p$ using structure from differential operators and symplectic resolutions.

\subsection{Main results}
The behaviour of category $\OCat_p$ crucially depends on the parameter $p$: it is, for example, semisimple for generic $p$, \cite{GGOR}.
So $\OCat_p$ becomes interesting only under some restrictions on $p$.
In order to state our results most efficiently we use an alternative parametrization of the rational Cherednik algebras using parameters that we call $\kappa, {\bf s}=(s_1,\ldots,s_{\ell})$ -- see (\ref{eq:c_formula}) for the expressions of the $c_i$'s in terms of these new parameters. Inside this set of parameters there is a non-empty Zariski open subset of {\it spherical} parameters -- see \ref{sphdefn} for
the definition -- which was described explicitly in \cite{DG}.

Let $\mathfrak{S}_{\ell}$ be the symmetric group on $\ell$ symbols
viewed as the group of bijections of $\mathbb{Z}/\ell\mathbb{Z}$. Then $\mathfrak{S}_{\ell}$  acts
on the set of parameters $(\kappa, s_1,\ldots,s_{\ell})$ by  permuting the $s$'s. The set of spherical parameters is stable under the action of $\mathfrak{S}_\ell$. If we identify $\irr{G_{\ell}(n)}$  with the set of $\ell$-multipartions $\lambda=(\lambda^{(1)},\ldots,\lambda^{(\ell)})$
of $n$ -- see  \ref{multiprtns} for this -- then $\mathfrak{S}_{\ell}$ also acts on the irreducible representations by permuting the components of $\lambda$.

Our first theorem shows that the action of $\mathfrak{S}_{\ell}$ produces highest weight equivalences of $\OCat$.

\begin{theoremA}\label{Thm:sym_group_equiv}
Let $p=(\kappa,{\bf s})$ be a spherical parameter and $\kappa\neq 0,1$.
Then for any $\sigma\in \mathfrak{S}_{\ell}$ there is an equivalence $\Xi_{\sigma,p}:\OCat_p\rightarrow
\OCat_{\sigma p}$ which for each $\lambda\in\irr{G_{\ell}(n)}$ sends $\Delta_p(\lambda)$ to $\Delta_{\sigma p}(\sigma \lambda)$.
\end{theoremA}

For a somewhat more extensive statement, see Theorem \ref{SSS_sph_sym_main}.
It is the sphericity condition which is both crucial
and restrictive in this theorem;  the restriction $\kappa\neq 0,1$ is irritating. \bigskip

Our second theorem concerns certain derived equivalences and confirms \cite[Conjecture 5.6]{rouqqsch} for the group $G(\ell, 1, n)$. Parameters $p=(\kappa, {\bf s}),p'=(\kappa',{\bf s}')$ have {\it integral difference}
if $\kappa'-\kappa\in \mathbb{Z}$ and there exists $a\in \C$ such that $\kappa's_i'-\kappa s_i\in a+ \mathbb{Z}$ for each $i=1,\ldots,\ell$.

\begin{theoremB}\label{Thm:derived}
Suppose that $p,p'$ have integral difference.
Then there is an equivalence of triangulated categories $D^b(\OCat_p)\rightarrow D^b(\OCat_{p'})$.
\end{theoremB}

Again, a slightly more detailed statement holds, see Theorem \ref{der_main_thm}. We do not, however, say anything here about where this equivalence sends the important objects of $\OCat_p$, even at the level of Grothendieck groups. Our methods are complementary to the derived equivalences for ${\bf H}_p\md$ of \cite{mn}: the latter are constructed for a general class of hamiltonian reductions but would, in the instance of this paper, require the parameters to be spherical.
\bigskip

For our third theorem we put relatively strong restrictions on the parameter $p = (\kappa, {\bf s})$: we assume $\kappa\notin \mathbb{Q}$ and that $s_{\ell}>s_1>s_2>\cdots > s_{\ell - 1}$. (But see Proposition \ref{reductioncase} where we establish a reduction for $\kappa\notin \mathbb{Q}$ using the $\mathfrak{S}_\ell$-action of Theorem A.) To ${\bf s}$ we associate first a higher level Fock space $F= V(\Lambda_{-s_{\ell -1}})\otimes \cdots \otimes V(\Lambda_{-s_1}) \otimes V(\Lambda_{-s_{\ell}})$ for $\mathfrak{gl}_{\infty}$  and its combinatorial crystal, see \ref{SUBSECTION_Fock}. Now let $N =  \ell (n+s_{\ell}) - (s_1 + \cdots + s_{\ell})$ and let $\mathfrak{p}$ be the standard parabolic subalgebra of $\gl_{N}(\C)$ with Levi the block diagonal matrices $\mathfrak{gl}_{s_{\ell}-s_{\ell-1}+n}\times \mathfrak{gl}_{s_{\ell}-s_{\ell-2}+n} \cdots \times \mathfrak{gl}_{s_{\ell}-s_{1}+n}\times \mathfrak{gl}_{n}.$    Following \cite{BKschurweyl} we consider a category $\mathcal{O}^{\sf par}_{\bf s}(n)$ which is a sum of blocks of the parabolic category $\mathcal{O}$ attached to $\mathfrak{p}$.

\begin{theoremC}
Suppose $p= (\kappa, {\bf s})$ where $\kappa \notin \mathbb{Q}$ and  ${\bf s} = (s_1, \ldots , s_{\ell})\in \Z^{\ell}$ satisfies $s_\ell > s_1 > \cdots > s_{\ell-1}$.
\begin{enumerate}
\item There is an equivalence of highest weight categories $\mathcal{O}_p \stackrel{\sim}{\longrightarrow} \Opar{{\bf s}}(n)$.
\item The set $\{ L_p(\lambda): \lambda^{\star} :=((\lambda^{(\ell-1)})^t, (\lambda^{(\ell -2)})^t, \ldots , (\lambda^{(1)})^t, (\lambda^{(\ell)})^t) \text{ highest weight in the crystal of $F$}\}$ is a complete set of finite dimensional irreducible representations for ${\bf H}_p$.
\end{enumerate}
\end{theoremC}

For more details, see Theorem \ref{parabolic} and its corollaries. Part (1) of this result was conjectured in \cite[Remark 8.10(b)]{VV} as a degenerate analogue of the Main Conjecture 8.8 in loc.cit.. The results of \cite{Shan} provide a count of the finite dimensional irreducible ${\bf H}_p$ representations, see too \cite{shanvasserot}, and some families of finite dimensional irreducible representations have been classified in \cite{EM}, \cite{etsup} and \cite{gri}. As a corollary of this theorem we are also able to confirm the non-degeneracy of a homological inner product defined on the finite dimensional representations for the parameters above, \cite[Conjecture 4.7]{etingofaffine}.
\subsection{Structure of the paper}
The paper is split into five further sections. We commence in Section \ref{SECTION_RCA} with preliminaries:
the definition of RCA in the general setting
including two different choices of parameters; the definition of the category $\OCat$ and of the $c$-function which is used
to define the highest weight structure on $\OCat$; various specific combinatorial points about the cyclotomic ($W=G_{\ell}(n)$) setting; and finally the induction
and restriction functors of \cite{BE}.

In Section \ref{SECTION_sphericalO} we introduce and study the so called
spherical categories $\OCat$. These analogues of $\OCat$ for the spherical subalgebras
of the rational Cherednik algebras are used in our proof of Theorem A, but may be of interest in their own right as they exist for all complex reflection groups. We work at this full level of generality in Section \ref{SECTION_sphericalO}.
We begin with the definition of the spherical $\OCat$ and then establish some of its simple
properties. We then discuss a relationship
between the usual and the spherical categories $\OCat$ and state our main result on this
relationship, Theorem \ref{sph_gen_main}, whose proof takes a little while to complete. Finally, we introduce another tool that will be used in the proof of Theorem A, the spherical restriction functor.

Theorem A is proved in Section \ref{SECTION_sym_equiv}. We start by constructing equivalences and then stating the more precise version of Theorem A that we actually prove,
Theorem \ref{SSS_sph_sym_main}. The proof of the theorem occupies the rest of this section.

In Section \ref{SECTION_derived} we prove Theorem B and in Section \ref{parabolicequiv} we prove Theorem C. Our routes to these results are very indirect, rather like flying between Minsk and Edinburgh via Moscow then Oregon, so we leave the detailed description of their contents to \ref{SSS_derived_content} and \ref{para_content}.

\subsection*{Notation}

Here is a list of standard notation to be frequently used in the paper.

\begin{longtable}{p{2.5cm} p{13.5cm}}
$[\mathcal{C}]$& the complexified Grothendieck group of an abelian category $\mathcal{C}$.\\
$(V)$& the two-sided ideal in an associative algebra, generated by a subset $V$.\\
$A\# G$& the smash-product of an algebra $A$ and a group $G$, where $G$
acts on $A$.\\
$D^b(\mathcal{C})$& the bounded derived category of an abelian category $\mathcal{C}$.\\
$\mu_\ell$&$:=\{\zeta\in \C| \zeta^\ell=1\}$.\\
$G_{\ell}(n)$&$:=(\mu_\ell)^n\rtimes S_n$.\\
$G_x$& the stabiliser of $x$ in a group $G$.\\
$\gr A$& the associated graded vector space of a filtered
vector space $A$.\\
$\irr{A}$& the set of (isomorphism classes of) an algebra (or a group) $A$. \\
$\mathscr{P},\mathscr{P}_\ell$& the sets of partitions and of $\ell$-multipartitions, respectively. \\
$RG$& the group algebra of a group $G$ with coefficients in a ring $R$.\\
$X^G$& the $G$-invariants in $X$.\\
$X/G$& the orbit set for an action of a group $G$ on $X$.\\
$Z(A)$& the centre of an algebra $A$. 
\end{longtable}

\section{Reminder on cyclotomic rational Cherednik algebras}
\label{SECTION_RCA}
\subsection{Rational Cherednik algebras: general setting}\label{SUBSECTION_RCA}
\subsubsection{Notation}\label{SSS:init_notation}
 Let $\h$ be a finite dimensional vector space over $\BC$, and $W\subset \GL(\h)$ be a finite group generated by its subset
$S$ of complex reflections.   Let $S_0,\ldots,S_r$ be the $W$-conjugacy classes in $S$.
Let $\Refl$ be the set of reflecting hyperplanes in $\mathfrak{h}$.
For $H\in \Refl$ let $W_H$ be the pointwise stabiliser of $H$ in $W$, set $e_H = |W_H|$ and
let $U = \bigcup_{H\in \Refl/W}\irr{W_H}$.
Since $W_H$ is a finite cyclic group, we may identify elements of $U$ with pairs $(H,j)$ where $0\leq j < e_H$ and the irreducible representation of $W_H$ is given by $\det^j|_{W_H}$. Here $\det$ stands for the representation of $W$ in $\bigwedge^{top}\h$.
Given $H\in \Refl$, choose $\alpha_H\in \mathfrak{h}^*$ with $\ker \alpha_H = H$ and let $\alpha^\vee_H\in \mathfrak{h}$ be such that $\C \alpha^\vee_H$ is a $W_H$-stable complement to $H$ and $\langle \alpha_H,\alpha^\vee_H\rangle = 2$. If $w\in W_H\setminus\{1\}$ we set
$\alpha_w :=\alpha_H$ and $\alpha^\vee_w:=\alpha^\vee_H$.

\subsubsection{Definition}
Pick  independent variables $c_0,c_1,\ldots,c_r$, one for each conjugacy class in $S$, and
also an independent variable $h$. Let $\param$ denote the vector space
dual to the span of $h,c_0,\ldots,c_r$. By the (universal) rational Cherednik
algebra  we mean the quotient ${\bf H}$ of $\BC[\param]\otimes T(\h\oplus\h^*)\#W$ by the relations
\begin{equation}\label{eq:RCA_relation}
[x,x']=[y,y']=0, [y,x]=h \langle y,x\rangle -\sum_{i=0}^r c_i\sum_{w\in S_i}\langle y,\alpha_w\rangle\langle \alpha_w^\vee, x\rangle w,
\end{equation}
where $x,x'\in \h^*, y,y'\in \h$.  We remark that (\ref{eq:RCA_relation}) is a special case of (\ref{eq:SRA_relation}), with the difference
that now $c_i$'s are independent variables and not complex numbers.

For $p\in \param$ let ${\bf H}_p$ denote the specialization of ${\bf H}$ at $p$. More generally given
a $\C$-algebra homomorphism $\psi$ from $\BC[\param]$ to some $\BC$-algebra $R$ we set ${\bf H}_\psi:={\bf H}\otimes_{\BC[\param]}R$.

We have a natural homomorphism $\BC[\param][\h]\otimes \BC W\otimes S(\h)\rightarrow {\bf H}$. According
to Etingof and Ginzburg, \cite[Theorem 1.3]{EG}, this homomorphism is a vector space isomorphism.
Observe that the algebra ${\bf H}$ is graded: $W$ has degree 0, $\h\oplus \h^*$ degree 1,
while $\param^*\subset {\bf H}$ is of degree 2.

For future use, we let $\param_1$ denote the affine subset of $\param$ where $h=1$.
\subsubsection{Alternative parametrization}\label{SSS_alt_pres}
Often it is more convenient to work with another parametrization of ${\bf H}$. We follow
\cite[\S 5]{rouqqsch}. Let $\{{\bf h}_u\}$ be a set of indeterminates with $u\in U$, and set ${\bf R}:= \C [\{{\bf h}_u\}]$.
Define a homomorphism $\iota: \BC[\param]\rightarrow {\bf R}$ by
\begin{equation}\label{eq:Rparam1}
h\mapsto 1, c_i\mapsto \sum_{j=1}^{e_H-1} \frac{1-\det(w)^{-j}}{2}({\bf h}_{H,j} - {\bf h}_{H,j-1}),
\end{equation}
where  $H\in \Refl$ and $w\in S_i\cap W_H$. Let us remark that we have $S_i\cap W_H\neq \varnothing$
for exactly one equivalence class $H$ in $\mathcal{A}$. We write ${\bf H_R}$ for ${\bf H}_\iota$. The commutation relation between $x$'s and $y$'s in ${\bf H_R}$ is
$$[y,x] = \langle y, x\rangle + \sum_{H\in\Refl} \frac{\langle y, \alpha_H\rangle\langle \alpha^\vee_H, x\rangle}{\lan \alpha_H^{\vee}, \alpha_H\ran}\gamma_H,$$ where $$\gamma_H := \sum_{w\in W_H\setminus\{1\}} \left(\sum_{j=0}^{e_H-1} \det(w)^{-j} ({\bf h}_{H,j} - {\bf h}_{H,j-1})\right) w.$$
In the above equality we assume ${\bf h}_{H,-1}:={\bf h}_{H, e_H-1}$.

Given a homomorphism
$\Psi:{\bf R}\rightarrow R$ we will write ${\bf H}_\Psi$ for ${\bf H}_{\Psi\circ\iota}$. Almost all the specializations
of ${\bf H}$ we consider in this paper factor through ${\bf H_R}$.

\subsection{Category $\OCat$}\label{SUBSECTION_OCat}
\subsubsection{c-function}
\label{Odefn}
Let $\Psi : {\bf R}\rightarrow R$ be an algebra homomorphism, where $R$ is a local commutative noetherian algebra with residue field $K$, and let $\psi: {\bf R} \rightarrow K$ be the composition of $\Psi$ and the projection $R\twoheadrightarrow K$. We denote the induced homomorphism ${\bf H}\rightarrow
{\bf H}_\psi$ also by $\psi$.

Given an irreducible $W$-module $E$ over $\BC$
set ${c}_E(\Psi)\in K$ to be the scalar by which the element $$\frac{\dim \h}{2} - \sum_{i=0}^r \psi(c_i)\sum_{w\in S} \frac{2}{1-\lambda_w}w\in Z(K W)$$ acts on $E\otimes_{\C} K$, where $\lambda_w$ is the non-trivial eigenvalue of $w$ in its action on $\h^*$.

\subsubsection{Definition of $\OCat$} Set $\OCat_\Psi(=\mathcal{O}_{\Psi}(\h,W))$ to be the category of  ${\bf H}_{\Psi}$-modules
that are locally nilpotent for the action of $\mathfrak{h}\subset R[\mathfrak{h}^*]$ and finitely generated over $R[\h]$
(equivalently, under the local nilpotency condition, over ${\bf H}_\Psi$). This is a highest weight category, \cite{GGOR} and \cite[\S 5.1]{rouqqsch}. Its standard objects are $\Delta_{\Psi}(E) = {\bf H}_{\Psi}\otimes_{R[\mathfrak{h}^*]\# W} (R\otimes E)$ where $E\in \irr{W}$. Here we consider $R\otimes E$
as an $R[\mathfrak{h}^*]\# W$-module by making $\h$ act by $0$. An  ordering on $\OCat_\Psi$ is defined by $$\Delta_{\Psi}(E) < \Delta_{\Psi}(F) \text{ if and only if } {c}_F(\Psi) - {c}_E(\Psi) \in \mathbb{Z}_{<0}$$ for $E,F\in\irr{W}$.  Henceforth we will write $E<_{\Psi} F$ if $c_F(\Psi) - c_E(\Psi) \in \mathbb{Z}_{<0}$.

We remark that as an $R[\h]\# W$-module $\Delta_\Psi(E)$ is
isomorphic to $R[\h]\otimes_{\BC} E$.

\subsubsection{Euler element} \label{SUBSECTION_OCath}Let ${\euu} \in {\bf H}$ denote the deformed Euler element, defined by $${\euu} = \sum_{i=1}^{\dim \h} x_iy_i + \frac{\dim \h}{2} - \sum_{i=0}^ r c_i \sum_{w\in S_i} \frac{2}{1-\lambda_w}w.$$ By definition, for $E\in \irr{W}$ the element $\psi(\euu) \in {\bf H}_{\psi}$ acts on $1\otimes (E\otimes_{\C} K) \subseteq \Delta_{\psi}(E)$ by multiplication by the scalar $c_E(\Psi)$ defined in \ref{Odefn}.

\subsection{Rational Cherednik algebras: cyclotomic setting}
\label{sec:cyclosetting}
\subsubsection{Notation}
\label{cyclogener}
For any positive integer $e$ we will write $\zeta_e$ for $\exp(2\pi \sqrt{-1}/e)\in \C$.

Let $\ell$ be a positive integer which we will assume to be fixed from now on. For a positive integer $n$ we set $$W  = G_{\ell}(n): = \mu_{\ell}^n \rtimes \mathfrak{S}_n,$$ where $\mu_{\ell}$ denotes the group of $\ell^{\text{th}}$ roots of unity in $\C$. The group $G_\ell(n)$ is a reflection group with reflection representation $\h = \C^n$. (The case $\ell=1$ is degenerate and {\it all} constructions that follow have to be modified.) There are two types of reflecting hyperplanes in $\Refl$: \begin{align*} H_{i,j}^{k} &= \ker (x_i - \zeta_\ell^{k}x_j) &&\text{ for $1\leq i < j \leq n$ and $1\leq k \leq \ell$}; \\ H_i &= \ker x_i &&\text{ for $1\leq i \leq n$.}\end{align*} These form two distinct $W$-orbits and we have $W_{H_{i,j}^k} \cong \Z_2$ and $W_{H_i} \cong \mu_{\ell}$.

It is also convenient to introduce some new elements $\epsilon_i, i=0,\ldots,\ell-1$ of $\param^*$.
We set
\begin{align}\label{eq:epsilon}
\epsilon_i&:=\frac{1}{\ell}(h-2\sum_{j=1}^{\ell-1}c_j \zeta_{\ell}^{ji}), \quad i\neq 0, \\\label{eq:epsilon_0}
\epsilon_0&:=\frac{1}{\ell}(h-2\sum_{j=1}^{\ell-1}c_j)+c_0-h/2.
\end{align}
In particular, we have $\iota(\epsilon_i)=\frac{1}{\ell}+{\bf h}_{H_\bullet, i}-{\bf h}_{H_\bullet, i-1}$ for $1\leq i \leq \ell -1$ and $\iota(\epsilon_0) = \frac{1}{\ell} + {\bf h}_{H_{\bullet},0} - {\bf h}_{H_{\bullet}, \ell -1} - 1/2 + {\bf h}_{H_{\bullet,\bullet}^{\bullet},1} - {\bf h}_{H_{\bullet,\bullet}^{\bullet},0}$.

\subsubsection{Parametrization}\label{SSS_cyclot_param}
We will be most interested in the following specialization $\Psi: {\bf R} \rightarrow \C$: let ${\bf s} = (s_1, \ldots , s_{\ell})\in \C^{\ell}$ and $\kappa\in \C$, and set \begin{equation} \label{parameterchoice} \Psi( {\bf h}_{H_{i,j}^k, 0}) = \kappa, \Psi( {\bf h}_{H_{i,j}^k, 1}) = 0, \text{ and } \Psi({\bf h}_{H_i, j}) = \kappa s_{j} - \frac{j}{\ell} \quad \text{for }0\leq j \leq \ell - 1, \end{equation} where $s_0 = s_{\ell}$.
We refer to the corresponding algebra as $H_{\kappa,{\bf s}}(G_{\ell}(n))$, dropping various parts of the $G_{\ell}(n)$ if it will lead to no ambiguity.

Let us provide the corresponding formulae for the $c$-parameters. Let $S_0$ be the class of reflections about the planes
$H_{i,j}^k$, and $S_e$, for $e=1,\ldots,\ell-1$, the class of reflections
corresponding to $\zeta_\ell^e\in \mu_{\ell}\cong W_{H_i}$.  We get
\begin{equation}\label{eq:c_formula}\Psi(h)=1, \Psi(c_0)=-\kappa, \Psi(c_i)=-(1+\kappa\sum_{j=1}^{\ell-1}(\zeta_{\ell}^{-ij}-1)(s_j-s_{j-1}))/2, i=1,\ldots,\ell-1.
\end{equation}
Moreover,
\begin{equation}\label{eq:eps_formula}
\Psi(\epsilon_i)=\kappa (s_i-s_{i-1}), i=1,\ldots,\ell-1.
\end{equation}

We remark that a parameter $p=(1,c_0,\ldots,c_{\ell-1})$
can be expressed via $\kappa$ and ${\bf s}$
as above provided $c_0\neq 0$. In this case the parameters ${\bf s}$ are defined uniquely up to
a common summand.

\subsubsection{Multipartitions}\label{multiprtns}
A partition is a non-increasing sequence $\la = (\la_1,\la_2,\dots)$
of non-negative integers; the length of $\la$ is
$|\la| = \la_1+\la_2+\cdots$.
An $\ell$-multipartition
is an ordered $\ell$-tuple of partitions
$\la = (\la^{(1)} , \dots,\la^{(\ell)})$;
its length is $|\la| = |\la^{(1)}|+\cdots+|\la^{(\ell)}|$.
We let $\mathscr{P}$, respectively $\mathscr{P}_{\ell}$,
denote the set of all partitions, respectively $\ell$-multipartitions; the set of those of length $n$ are denoted $\prtn{n}$ and $\prt{\ell}{n}$ respectively. Given $\lambda \in \mathscr{P}_{\ell}$, define its transpose $\lambda^{\star}$ to be  $((\lambda^{(\ell-1)})^t, (\lambda^{(\ell -2)})^t, \ldots , (\lambda^{(1)})^t, (\lambda^{(\ell)})^t)$ where $(\lambda^{(r)})^t$ is the usual transpose of the partition $\lambda^{(r)}$.

The Young diagram of the multipartition $\la = (\la^{(1)},
\dots, \la^{(\ell)})\in \mathscr{P}_{\ell}$ is
$$
\{(a,b,m)\in\Z_{>0}\times\Z_{>0}\times \{1,\dots,\ell\}\mid 1\leq b\leq
\la_a^{(m)}\}.
$$
The elements of this set
are called the boxes of $\la$. We identify the multipartition $\la$ with its
Young diagram, considered as a row vector of Young diagrams. For example, $((3,1), (4,2))$ is the Young diagram
\begin{equation}\label{youngdiag}
 \yng(1,3) \,\,\, ,\,\, \yng(2,4)
\end{equation}
Given an $\ell$-tuple ${\bf s} = (s_1, \ldots , s_{\ell})\in \C^{\ell}$ the ${\bf s}$-shifted content $\res^{\bf s}A$ of
the box $A = (a,b,m)$ is the integer
\begin{equation}\label{resdef}
\res^{\bf s} A = s_m+b-a \in \Z.
\end{equation}
If ${\bf s} = {\bf 0}$ we denote this by $\res\,A$.
If $\res^{\bf s}A = i$ then we say that $A$ is an $i$-box.

The dominance ordering on $\mathscr{P}_{\ell}$ is given by $\lambda \unrhd \mu $ if $\sum_{i=1}^{r-1} |\lambda^{(i)}| + \sum_{j=1}^t \lambda_{j}^{(r)} \geq \sum_{i=1}^{r-1} |\mu^{(i)}| + \sum_{j=1}^t \mu_{j}^{(r)}$ for all $1\leq r \leq \ell$ and $t\geq 1$, and with equality for $r = \ell$ and $t\gg 1$.

\subsubsection{Irreducible representations of $G_{\ell}(n)$}
We label the irreducible representations of $G_{\ell}(n)$ by $\prt{\ell}{n}$ as follows. Given $\lambda = (\lambda^{(1)}, \ldots , \lambda^{(\ell)})\in \prt{\ell}{n}$
denote by $l_r$ the largest integer such that $\lambda_{l_r}^{(r)} \neq 0$ (i.e., $l_r$
is the number of rows in $\lambda^{(r)}$). We put
$$I_{\lambda}(r) = \left\{ \sum_{i=1}^{r-1}|\lambda^{(i)}|+1,  \sum_{i=1}^{r-1}|\lambda^{(i)}|+2, \cdots ,  \sum_{i=1}^{r}|\lambda^{(i)}|\right\}.$$
We set $\mathfrak{S}_\lambda = \mathfrak{S}_{I_\lambda(1)}\times \cdots \times \mathfrak{S}_{I_\lambda(\ell)}$ and $G_\ell(\lambda) = \mu_{\ell}^{\{1,\ldots , n\}}\rtimes \mathfrak{S}_{\lambda}.$ Here, for $I\subset \{1,\ldots,n\}$, $\mathfrak{S}_I$ stands
for the subgroup of $\mathfrak{S}_n$ consisting of all permutations of $I$, leaving the other elements invariant.

Now a partition $\alpha\in\mathscr{P}(n)$ corresponds to an irreducible representation of $\mathfrak{S}_n$, where $(n)$ labels the trivial representation of $\mathfrak{S}_n$.  Then given $\lambda\in\prt{\ell}{n}$ we have a corresponding irreducible representation of $G_{\ell}(n)$ which is constructed as $$\text{Ind}_{G_\ell(\lambda)}^{G_{\ell}(n)} (\phi^{(1)}\cdot{\lambda^{(1)}} \otimes \cdots \otimes \phi^{(\ell)}\cdot{\lambda^{(\ell)}}),$$ where $\phi^{(r)}$ is the one dimensional character of $\mu_{\ell}^{I_{\lambda}(r)}\rtimes \mathfrak{S}_{I_{\lambda}(r)}$ whose restriction to $\mu_{\ell}^{I_{\lambda}(r)}$ is $\det^r$ and whose restriction to $\mathfrak{S}_{I_{\lambda}(r)}$ is trivial.

This labeling agrees with \cite{DJM}. For later use we remark that it does, however, differ from that of \cite[\S 6]{rouqqsch} by a shift as that paper takes $\text{Ind}_{G_\ell(\lambda)}^{G_{\ell}(n)} (\phi^{(0)}\cdot{\lambda^{(1)}} \otimes \cdots \otimes \phi^{(\ell-1)}\cdot{\lambda^{(\ell)}})$; for us the trivial representation is labeled by $(\emptyset, \ldots , \emptyset, (n))$ and $\wedge^n \h$ by $( (1^n), \emptyset, \ldots , \emptyset)$; and it also differs from \cite[\S 5]{Stembridge} as Stembridge uses $(\lambda^{(0)}, \ldots , \lambda^{(\ell-1)})$ where $\lambda^{(0)} = \lambda^{(\ell)}$.

\subsubsection{c-function}\label{cfnsec} Let $\euu \in H_{\kappa,{\bf s}}$ denote the deformed Euler element, the image of ${\euu}\in {\bf H}$ under the natural epimorphism ${\bf H}\twoheadrightarrow H_{\kappa,{\bf s}}$; by \eqref{eq:Rparam1} it equals \begin{eqnarray*} \euu &=& \sum_{i=1}^n x_iy_i + \frac{n}{2} + \sum_{H\in \Refl}\sum_{w\in W_H\setminus \{1\}} \left(\sum_{j=0}^{e_H-1} \det(w)^{-j}\Psi({\bf h}_{H,j})\right) w \\ & = & \sum_{i=1}^n x_iy_i + \frac{n}{2} + \sum_{H\in \Refl}\sum_{w\in W_H} \left(\sum_{j=0}^{e_H-1} \det(w)^{-j}\Psi({\bf h}_{H,j})\right) w - \sum_{H\in \Refl}\sum_{j=0}^{e_{H}-1} \Psi({\bf h}_{H,j}).\end{eqnarray*}
The action of $\sum_{H\in \Refl}\sum_{w\in W_H} \left(\sum_{j=0}^{e_H-1} \det(w)^{-j}\Psi({\bf h}_{H,j})\right) w$ on $\lambda = (\lambda^{(1)}, \ldots , \lambda^{(\ell)}) \in \irr{G_{\ell}(n)}$ has been calculated in \cite[Section 3.2.1 and Proposition 6.2]{rouqqsch} and equals $$\left(\sum_{r=1}^{\ell-1}|\lambda^{(r)}|(\kappa \ell s_{r} - \kappa \ell s_{\ell}  - r) - \frac{\kappa \ell n(n-1)}{2} + \sum_{A\in \lambda} \kappa\ell \res(A)\right) +  \left(\kappa \ell n(n-1) +  \ell n \kappa s_\ell \right).$$ Thus we find \begin{equation} \label{cfunvalue} c_{\lambda}(\kappa, {\bf s}) := c_{\lambda}(\Psi) = - \sum_{r=1}^{\ell-1}|\lambda^{(r)}| r  + \kappa\ell\sum_{A\in \lambda} \res^{\bf s}(A) - \kappa n s +n - \frac{n\ell}{2},\end{equation}
where $s= \sum_{j=1}^{\ell} s_j$.

\subsubsection{Fake degree polynomial and modified c-function}\label{SSS_fake} We will be interested in the graded $G_{\ell}(n)$-module
structure of the coinvariant ring $\C[\h]^{\text{co}G_{\ell}(n)} : = \C[\h]/\lan \C[\h]_+^{G_{\ell}(n)}\ran$. Recall first that $\C[\h]^{\text{co}{G_{\ell}(n)}}$ is a graded space which transforms as the regular representation under the action of $G_{\ell}(n)$. For $\tau = (\tau^{(1)}, \ldots , \tau^{(\ell)}) \in \irr{G_{\ell}(n)}$ the polynomial $f_{\tau}(q) := \sum_{i \geq 0 } [\C[\h]^{\text{co}{G_{\ell}(n)}}_i : \tau] q^i$ is called the fake degree polynomial of  $\tau$. It has been calculated in \cite[Theorem 5.3]{Stembridge} (where $S(\h)$ is used rather than $\C[\h]$): \begin{equation}\label{eq:fake_polyn} f_{\tau^*}(q) =  (q^{\ell})_n q^{w(\tau)}\prod_{r=1}^{\ell} \frac{q^{n(\tau^{(r)})\ell}}{H_{\tau^{(r)}}(q^{\ell})}\end{equation} where $\tau^*$ is the dual of $\tau$, $w(\tau) = \sum_{r=1}^{\ell-1} r |\tau^{(r)}|$, $(q^{\ell})_n = (1-q^{\ell})(1-q^{2\ell})\cdots (1-q^{n\ell})$, $n(\tau^{(r)}) = \sum_{j\geq 1} (j-1)\tau^{(r)}_j$, and $H_{\tau^{(r)}}(q^{\ell}) = \prod_{A\in \tau^{(r)}} (1- q^{h(A)\ell})$ with $h(A)$ the hook length of the box $A$.

Let ${\sf fd}(\tau^*)$  denote the least degree in which a copy of $\tau^*$ appears in $\C[\h]^{\text{co}G_{\ell}(n)}$.
Since $(0)_n =1$ and $H_{\tau^{(r)}} (0) = 1$, we deduce that $${\sf fd}(\tau^*) = w(\tau) + \deg(\prod_{r=1}^{\ell} q^{n(\tau^{(r)})\ell})= \sum_{r=1}^{\ell-1} r|\tau^{(r)}| + \sum_{r=1}^{\ell}n(\tau^{(r)})\ell.$$
Define a  function $\hat{c}_\lambda(\kappa,{\bf s})$ by $$\hat{c}_\lambda(\kappa,{\bf s}) := {c}_\lambda(\kappa,{\bf s}) + {\sf fd}(\lambda^*)$$
We have
\begin{equation}\label{fakecequation}\hat{c}_\lambda(\kappa,{\bf s}) =\ell \kappa  \sum_{A\in \lambda}\res^{\bf s}(A) + \ell \sum_{r=1}^{\ell} n(\lambda^{(r)})  - \kappa ns  +n- \frac{n\ell}{2}.\end{equation}

\subsection{Induction and restriction functors}\label{SUBSECTION_IndRes}
\subsubsection{Bezrukavnikov-Etingof functors}\label{SSS_BE_functors}
Let $W$ be an arbitrary complex reflection group, $\h$ its reflection representation,
and let $\underline{W}\subset W$ be a parabolic subgroup, that is the stabiliser of some point
in $\h$. Let $\h^+$ denote a unique $\underline{W}$-stable complement to $\h^{\underline{W}}$
in $\h$. Restrict $p\in \param_1$ to the set of complex reflections in $\underline{W}$
and form the corresponding RCA and its category $\OCat_p(\underline{W},\h^+)$.

In \cite{BE} Bezrukavnikov and Etingof established exact functors $\operatorname{Res}^W_{\underline{W}}:\OCat_p(W,\h)
\rightarrow \OCat_p(\underline{W},\h^+)$ (restriction) and $\operatorname{Ind}_{\underline{W}}^W:
\OCat_p(\underline{W},\h^+)\rightarrow \OCat_p(W,\h)$ (induction).
In what follows we will need a construction of the  restriction functors, so we will
recall them in detail.

\subsubsection{Isomorphism of completions}
The starting point for the construction in \cite{BE} is a certain isomorphism of algebras ${\bf H}^{\wedge_b}$ and $Z(W,\underline{W}, \underline{\bf H}^{\wedge_b})$. Here we pick a point $b\in\h$ with $W_b=\underline{W}$ and set ${\bf H}^{\wedge_b}:=\C[\h/W]^{\wedge_b}\otimes_{\C[\h/W]}{\bf H}$. This space has a natural algebra structure because the adjoint action of $\C[\h/W]$ on ${\bf H}$ is by locally nilpotent endomorphisms.

Consider the algebra $\underline{\bf H}$ that is the quotient of $\C[\param]\otimes T(\h\oplus\h^*)\#\underline{W}$
by the relations
$$[x,x']=[y,y']=0, [y,x]=h \langle y,x\rangle -\sum_{i=0}^r c_i\sum_{w\in S_i\cap\underline{W}}\langle y,\alpha_w\rangle\langle \alpha_w^\vee, x\rangle w.$$
Set $\underline{\bf H}^{\wedge_b}:=\C[\h/\underline{W}]^{\wedge_b}\otimes_{\C[\h/\underline{W}]}\underline{\bf H}$.
We remark that Bezrukavnikov and
Etingof considered the completion  $\underline{\bf H}^{\wedge_0}$ but the latter is naturally
isomorphic to $\underline{\bf H}^{\wedge_b}$, the isomorphism is provided by the shift
by $b$.

Next, let $Z(W,\underline{W},\bullet)$ be the centralizer algebra introduced in \cite[\S 3.2]{BE}.
More precisely, let $H\subset G$ be finite groups and $A$ be an associative algebra
equipped with an embedding  $\C H\hookrightarrow A$. Then one can form the space $\Fun_H(G,A)$
of all $H$-equivariant maps $G\rightarrow A$, i.e., maps satisfying $f(hg)=h f(g)$.
This space has a natural structure of a right $A$-module. By $Z(G,H,A)$ one denotes the endomorphism
algebra of this right module. Since $\Fun_H(G,A)$ a free right $A$-module of rank $|G/H|$, we see that
$Z(G,H,A)$ may be identified with $\operatorname{Mat}_{|G/H|}(A)$ but this identification is
not canonical.

Now for $\alpha\in \h^*,a\in \h$
we write $x_\alpha,y_a,\underline{x}_\alpha,\underline{y}_a$ for the corresponding elements in ${\bf H}, \underline{\bf H}$, respectively. Then, by \cite[Theorem 3.2]{BE}
the assignment $\vartheta$ defined by
\begin{equation}\label{def_theta}
\begin{split}
&[\vartheta(w')f](w)=f(ww'),\\
&[\vartheta(x_\alpha)f](w)=\underline{x}_{w\alpha}f(w),\\
&[\vartheta(y_a)f](w)=\underline{y}_{wa}f(w)+\sum_i\sum_{u\in S_i\setminus \underline{W}}\frac{2c_i}{1-\lambda_u}\frac{\langle\alpha_u,wa\rangle}{\underline{x}_{\alpha_u}}(f(uw)-f(w)),\\
& w,w'\in W, \alpha\in \h^*, a\in \h, f\in \Fun_{\underline{W}}(W,\underline{H}),
\end{split}
\end{equation}
extends to a homomorphism $\vartheta:{\bf H}\rightarrow Z(W,\underline{W}, \underline{\bf H}^{\wedge_b})$.
This homomorphism further extends by continuity to ${\bf H}^{\wedge_b}$ and the map
${\bf H}^{\wedge_b}\rightarrow Z(W,\underline{W}, \underline{\bf H}^{\wedge_b})$
is an isomorphism.

\subsubsection{Construction of the restriction functor}\label{SSS_BE_functors1}
We write $\OCat$ for $\OCat_p(W,\h)$, $\underline{\OCat}$ for $\OCat_p(\underline{W},\h)$
and $\underline{\OCat}^+$ for $\OCat_p(\underline{W},\h^+)$.
Consider the category $\OCat^{\wedge_b}$ that consists of all ${\bf H}_p^{\wedge_b}$-modules that are finitely
generated over $\C[\h/W]^{\wedge_b}$. There is an equivalence between $\OCat^{\wedge_b}$ and $\underline{\OCat}^+$
constructed as follows.

The categories of $Z(W,\underline{W}, A)$-modules and
of $A$-modules are naturally equivalent for any algebra $A$ equipped with
an embedding of $\C \underline{W}$. Indeed to an $A$-module $M$
we assign the $Z(W,\underline{W}, A)$-module $\Fun_{\underline{W}}(W, M)$, of all
$\underline{W}$-equivariant functions $f:W\rightarrow M$. Conversely,
let $M'$ be a $Z(W,\underline{W},A)$-module. Define an element $e(\underline{W})\in Z(W,\underline{W},
A)$ by
\begin{equation}
e(\underline{W}) f(w)= \begin{cases} f(w) \qquad & \text{if }w\in \underline{W},\\
 0 &\text{ else},\end{cases}
\end{equation}
where $f\in \Fun_{\underline{W}}(W, A).$ It is easy to see that $A$ is naturally identified with $e(\underline{W})Z(W,\underline{W},A)e(\underline{W})$.
So to $M'$ we can assign the $A$-module $e(\underline{W})M'$.
Then the functors $M\mapsto \Fun_{\underline{W}}(W,M),
M'\mapsto e(\underline{W})M'$ are quasi-inverse equivalences.

The composition $e(\underline{W})\circ\vartheta_*$ is an equivalence  $\OCat^{\wedge_b}\xrightarrow{\sim} \underline{\OCat}^{\wedge_b}$,
where the latter category is defined analogously to $\OCat^{\wedge_b}$ but for the pair $(\underline{W},\h)$.
Then there is an equivalence $\underline{\OCat}^{\wedge_b}\rightarrow \underline{\OCat}^+$ that maps
an object $N'\in \underline{\OCat}^{\wedge_b}$ to the space of all elements in $N'$ that are
eigenvectors for $\underline{y}_a,a\in \h^{\underline{W}},$ and generalized eigenvectors  for $\underline{y}_a, a\in \h^+,$
with (both) eigenvalues $0$. The reader is referred to \cite[Theorem 2.3 and \S 3.5]{BE} for details.
We denote the resulting equivalence $\OCat^{\wedge_b}\xrightarrow{\sim}\OCat(\underline{W})$ by $\Funct_b$.

For $M\in \OCat$  set  $\Res_b(M):=\mathcal{F}_b(M^{\wedge_b})$, where
$M^{\wedge_b}:=\C[\h/W]^{\wedge_b}\otimes_{\C[\h/W]}M$ is the completion of $M$
at $b$. As proved in \cite{BE}, up to isomorphism the functor $\Res_b$
does not depend on the choice of $b$ with $W_b=\underline{W}$.
So we write $\Res^W_{\underline{W}}$ for $\Res_b$.

\subsubsection{Res vs the support filtration}\label{SSS_res_support}
An important feature of the restriction
functors is that they can be used to identify finite dimensional modules: a module in $\OCat_p(\h,W)$
is finite dimensional if and only if it is annihilated by $\operatorname{Res}^W_{\underline{W}}$
for any proper parabolic subgroup $\underline{W}\subset W$.  More generally, for a parabolic subgroup
$W_1\subset W$ consider the subvariety $X_{W_1}\subset \h/W$
that is the image of $\h^{W_1}\subset \h$ in $\h/W$ under the quotient morphism. It is known due to \cite[Theorem 6.8]{ginprim}
that the support of any simple module in $\OCat_p(\h,W)$ viewed as a $\C[\h]^W$-module coincides
with some $X_{W_1}$. It is easy to prove that the simples whose support is $X_{W_1}$ are precisely
those annihilated by $\Res^W_{\underline{W}}$ with $\underline{W}$ not $W$-conjugate to  a subgroup of $W_1$.

By the support filtration on $\OCat$ we mean the filtration of $\OCat$ by the full subcategories
$\OCat_{\underline{W}}$, one for each parabolic subgroup $\underline{W}\subset W$, consisting of
all modules supported on $X_{\underline{W}}$. The previous paragraph shows that the support filtration
can be recovered from the restriction functors.

\section{Spherical category $\OCat$}\label{SECTION_sphericalO}
Throughout this section $W$ is an arbitrary complex reflection group with reflection representation $\h$.
\subsection{Category $\OCat^{sph}$}\label{SUBSECTION_sphO_def}
\subsubsection{Spherical subalgebra} \label{sphdefn}
Let $e:=|W|^{-1}\sum_{w\in W}w$ be the trivial idempotent in $\C W$. The spherical subalgebra
in ${\bf H}$ is ${\bf U}:=e{\bf H}e$. A parameter $p\in \param$ is called {\it spherical} if ${\bf H}_p={\bf H}_p e {\bf H}_p$. We write $\param_1^{sph}$ for the set of spherical values in $\param_1$, the affine subspace in $\param$ defined by the equation $h=1$.

If $p\in \param$ is spherical then  there is an equivalence of abelian categories $\Lmod{{\bf H}_p}\stackrel{\sim}{\rightarrow} \Lmod{{\bf U}_p}$ given by $M\mapsto e M$ with quasi-inverse $N\mapsto {\bf H}_p e\otimes_{{\bf U}_p}N$.

\subsubsection{Euler grading on ${\bf U}$}
We set $\euu^{sph}:=e \euu$, where $\euu$ is the deformed Euler
element in ${\bf H}$, see \ref{SUBSECTION_OCath}.

The derivation $\frac{1}{h}\ad(\euu^{sph})$ acts on ${\bf U}$
diagonalizably with integral eigenvalues: we write ${\bf U}:=\bigoplus_{i\in \BZ}{\bf U}(i)$
for the corresponding grading. We set ${\bf U}(<0):=\bigoplus_{i<0}{\bf U}(i),
{\bf U}(\leqslant 0):=\bigoplus_{i\leqslant 0}{\bf U}(i), {\bf U}(\leqslant 0)^-:=
{\bf U}(\leqslant 0)\cap {\bf U}{\bf U}(<0)$. Clearly, ${\bf U}(\leqslant 0)$ is a subalgebra
in ${\bf U}$, while both ${\bf U}(<0),{\bf U}(\leqslant 0)^-$ are ideals in ${\bf U}(\leqslant 0)$.
Set ${\bf U}^0:={\bf U}(\leqslant 0)/{\bf U}(\leqslant 0)^-$. Of course, ${\bf U}^0={\bf U}(0)/({\bf U}(0)\cap {\bf U}{\bf U}(<0))$.

Let ${\bf U}_p(<0),{\bf U}_p(\leqslant 0)$, etc., denote the specializations of the corresponding objects
to the parameter $p$. Alternatively, these objects may be obtained from ${\bf U}_p$
in the same way as in the previous paragraph.

Let us record the following lemma for later use.

\begin{lemma}\label{lemma:Noetherian}
The algebra ${\bf U}(\leqslant 0)$ is noetherian.
\end{lemma}
\begin{proof}
Since the $\C[\param]$-algebra ${\bf U}(\leqslant 0)$ is graded,
it is enough to prove this property for ${\bf U}_0(\leqslant 0)={\bf U}(\leqslant 0)/(\param^*)$.
Now consider the algebra ${\bf U}_0\otimes \C[x]$ with the diagonal $\C^\times$-action
where $x$ is given degree $1$. Then ${\bf U}_0(\leqslant 0)\cong({\bf U}_0\otimes \C[x])^{\C^\times}$.
Being the invariant subalgebra for a reductive group action on a finitely generated commutative algebra,
the algebra  ${\bf U}_0(\leqslant 0)$ itself is finitely generated.
\end{proof}

\subsubsection{Definition of $\OCat^{sph}$} We will give a definition of an ``internal'' category $\OCat_p^{sph}$
for the algebra ${\bf U}_p$. The definition is inspired by the  definition
of the category $\OCat^{sph}$ for W-algebras, \cite{BGK} and \cite{LOcat}, and for hypertoric varieties, \cite{BLPW}.

We say that an ${\bf U}_p$-module $M$ lies in the category $\mathcal{O}_p^{sph}$ if
\begin{itemize}
\item[($O_1$)] $M$ is finitely generated;
\item[($O_2$)] $\euu^{sph}$ acts on $M$ locally finitely;
\item[($O_3$)] ${\bf U}_p(<0)$ acts on $M$ by locally nilpotent endomorphisms.
\end{itemize}
We remark that given the first two conditions, $(O_3)$ is equivalent to either of the following:
\begin{itemize}
\item[($O_3'$)] for any $m\in M$ there is $\alpha\in \BZ$ such that ${\bf U}_p(\beta)m=0$
for any $\beta\leqslant \alpha$.
\item[($O_3''$)] the set of $\euu^{sph}$-weights in $M$ is bounded below,
meaning that there are $\alpha_1,\ldots,\alpha_k\in \C$ such that
the weight subspace $M_\alpha$ is zero whenever $\alpha-\alpha_i\not\in \BZ_{\geqslant 0}$
for all $i$.
\end{itemize}

\subsubsection{Weight spaces}\label{lemma:SphO2}
\begin{lemma}
The algebra ${\bf U}_p^0$ is finite dimensional.
\end{lemma}
\begin{proof}
Clearly, $\gr {\bf U}_p(<0)={\bf U}_0(<0)$. It follows that $\gr {\bf U}_p^0$ is a quotient of ${\bf U}_0^0$ and so it is enough to prove that ${\bf U}_0^0$ is finite dimensional.
Consider  homogeneous free generators $a_1,\ldots,a_n$ of $\C[\h]^W$
and $b_1,\ldots, b_n$ of $S(\h)^W$.  The algebra ${\bf U}_0$ is a finitely generated
module over $\C[a_1,\ldots,a_n,b_1,\ldots,b_n]$ so we can pick bihomogeneous generators
$c_1,\ldots,c_m$. Thus monomials in $a_i,b_j,c_k$ with the $c_k$'s appearing at most once span ${\bf U}_0$ as a vector space.
The algebra ${\bf U}_0^0$ is spanned by the images of the monomials lying in ${\bf U}_0(0)$
and not containing $b_j$'s. Since there are only finitely many of such,
the algebra ${\bf U}_0^0$ is finite dimensional.
\end{proof}

The following is a standard corollary of Lemma \ref{lemma:SphO2}.

\begin{corollary}\label{proposition:SphO3}
All generalized eigenspaces of a module $M\in \OCat_p^{sph}$ are finite dimensional.
\end{corollary}

\subsubsection{Functors $\Delta^{sph}_p$ and $\bullet^{{\bf U}_p(<0)}$} Let $M^0$ be a finitely generated ${\bf U}_p^0$-module. We can view $M^0$ as an ${\bf U}_p(\leqslant 0)$-module
via the epimorphism ${\bf U}_p(\leqslant 0)\twoheadrightarrow {\bf U}_p^0$. We define $\Delta_p^{sph}(M^0):={\bf U}_p\otimes_{{\bf U}_p(\leqslant 0)}M^0$, an object of $\OCat^{sph}$. Hence we have a right exact functor $$\Delta_p^{sph}:{\bf U}_p^0\md \rightarrow \OCat^{sph}.$$

\begin{lemma}\label{lemma:SphO1}
There is a right adjoint functor to $\Delta_p^{sph}$, mapping $M\in \OCat_p^{sph}$ to
the annihilator $M^{{\bf U}_p(<0)}$ of ${\bf U}_p(<0)$ in $M$.
\end{lemma}
\begin{proof}
We only need to prove that the ${\bf U}_p^0$-module $M^{{\bf U}_p(<0)}$ is finitely generated.  But since ${\bf U}_p^0$ is finite dimensional, $M^{{\bf U}_p(<0)}$ is a sum of finitely many distinct generalized $\euu^{sph}$-eigenspaces, and by the above corollary each of these is finite dimensional.
\end{proof}

\subsubsection{Simple objects}\label{sphsimp}
Now let $M^0$ be an irreducible ${\bf U}_p^0$-module. Then $\Delta^{sph}_p(M^0)$ is indecomposable and, moreover, has a unique simple quotient which we denote by $L^{sph}_p(M^0)$. The map $M^0\mapsto L^{sph}_p(M^0)$ is a bijection between $\irr{{\bf U}_p^0}$ and the set of simples in $\OCat_p^{sph}$. Its inverse is given by taking the lowest weight
subspace or, alternatively, by applying the functor $\bullet^{{\bf U}_p(<0)}$.

\subsection{Main result}\label{SS:sph_gen_main}
\subsubsection{Functors between $\OCat^{sph}$ and $\OCat$}
It is clear that the assignment $M\mapsto eM$ defines an exact functor
$\OCat_p\rightarrow \OCat_p^{sph}$.

\begin{proposition}\label{proposition:SphO4}
The functor $N\mapsto {\bf H}_p e\otimes_{{\bf U}_p}N$ maps $\OCat_p^{sph}$ to $\OCat_p$.
\end{proposition}
\begin{proof}
Let $N\in \OCat_p^{sph}$. Observe that ${\bf H}_p e$ is finitely generated both as a left ${\bf H}_p$-module and as a right ${\bf U}_p$-module. In particular, since $N$ is a finitely generated ${\bf U}_p$-module, ${\bf H}_p e\otimes_{{\bf U}_p}N$ is finitely generated over ${\bf H}_p$. Now let us choose eigenvectors $f_1,\ldots,f_k$ for the adjoint action of $\euu$ that generate ${\bf H}_pe$
over ${\bf U}_p$. Then ${\bf H}_p e\otimes_{{\bf U}_p}N$ is spanned by elements of the form $f_i\otimes v, v\in N$.
It follows that the $\euu$-weights of ${\bf H}_p e\otimes_{{\bf U}_p}N$ are bounded below, so $S(\h)$ acts locally nilpotently. Hence ${\bf H}_p e\otimes_{{\bf U}_p}N\in\OCat_p$.
\end{proof}

\subsubsection{Correspondence between simples}\label{sphequiv} If $p\in \param_1^{sph}$ then \ref{sphdefn} and \ref{proposition:SphO4} show that the categories
$\OCat_p$ and $\OCat_p^{sph}$ are equivalent as abelian categories. In particular, the ${\bf U}_p$-module
$e L_p(\lambda)$ is irreducible for any irreducible $W$-module $\lambda$ and the collection
$\{ e L_p(\lambda):  \lambda\in \irr{W}\}$ exhausts the set of irreducible objects in $\mathcal{O}_p^{sph}$.
By \ref{sphsimp} this produces a bijection $$\iota^L_{p}: \irr{W}\xrightarrow{\sim}\irr{{\bf U}_{p}^0}$$
with $e L_p(\lambda)=L^{sph}(\iota^L_p(\lambda))$.

\subsubsection{Homomorphism $\varphi$} \label{morphismsph}There is another way to relate $\irr{W}$ and representations of ${{\bf U}_p^0}$.
Consider the universal Verma ${\bf H}$-module ${\bf \Delta}(\lambda)=\C[\param]\otimes \C[\h]\otimes \lambda$
and take its spherical part $e{\bf \Delta}(\lambda)$. As in \ref{SSS_fake}, let $\fd(\lambda^*)$ denote the smallest
degree $d$ such that $\lambda^*$ appears in $\C[\h]_d$.
So the $\fd(\lambda^*)$-degree component of  $e{\bf \Delta}(\lambda)$ is naturally identified
with a free $\C[\param]$-module. This component is stable under ${\bf U}(\leqslant 0)$
and is annihilated by ${\bf U}(<0)$, so it becomes a ${\bf U}^0$-module.
We will denote it by $\iota(\lambda)$.
By definition we have a homomorphism $\varphi:{\bf \Delta}^{sph}(\iota(\lambda))
\rightarrow e{\bf \Delta}(\lambda)$ which is the identity on the lowest weight $\C[\param]$-submodule $\iota(\lambda)$.

Specializing at $p$, we get a finite dimensional ${\bf U}^0_p$-module denoted by
$\iota_p(\lambda)$.
\subsubsection{Main result}\label{sph_gen_main}
The following conjecture is key to understanding the structure of $\mathcal{O}_p^{sph}$.

\begin{conjecture} Keep the above notation.
For all $p\in\param_1^{sph}$ we have
\begin{itemize}
\item[(i)]
$\iota^L_p=\iota_p$;
\item[(ii)] the homomorphism
$\varphi_p:\Delta^{sph}_p(\iota(\lambda))\rightarrow e\Delta_p(\lambda)$ is an isomorphism.
\end{itemize}
\end{conjecture}

We can prove a weaker statement.

\begin{theorem}\label{proposition:sph_main_relation}
In the above conjecture, (i) and (ii) hold for  $p$ in some non-empty Zariski open
subset $\param_1^!\subset \param_1^{sph}$.
\end{theorem}

The proof proceeds in two halves. In Subsection \ref{SUBSECTION_VermaI} we will show that the homomorphism $\Delta_p^{sph}(\iota(\lambda))\rightarrow
e\Delta_p(\lambda)$ is surjective on a Zariski open subset of $\param_1$.  The injectivity of $\Delta^{sph}_p(\iota(\lambda))\rightarrow e\Delta_p(\lambda)$
is more subtle, and will be established in Subsection \ref{SUBSECTION_VermaII}
after restricting to a smaller Zariski open subset.

\subsubsection{$\param_1^{sph}$ is Zariski open}\label{SSS:Zariski_open} Recall that the {\it aspherical locus} $\param^{asph}_1$ is the complement to the set of spherical parameters $\param^{sph}_1$ in $\param_1$. 

\begin{lemma} The apsherical locus is Zariski closed in $\param_1$.
\end{lemma}
\begin{proof}
Consider ${\bf H}_1$, the $\C[\param_1]$-algebra obtained by specializing $h$ to $1$ in ${\bf H}$. Let ${\bf I}$ be the two-sided ideal ${\bf H}_1e{\bf H}_1$
and let ${\bf I}^1,\ldots, {\bf I}^k$ be its minimal prime ideals. Set ${\bf J}:={\bf H}_1e{\bf H}_1\cap \C[\param_1],
{\bf J}^i:={\bf I}^i\cap \C[\param_1]$. Then the ${\bf J}^i$ for $i=1,\ldots,k$ are all minimal prime ideals
of ${\bf J}$. By generic freeness we that $\gr {\bf H}_1/{\bf H}_1 e {\bf H}_1$
and hence ${\bf H}_1/{\bf H}_1 e {\bf H}_1$ are both generically free over $\C[\param_1]/{\bf J}$. It follows that $\param_1^{asph}\subset \bigcup_i \Spec(\C[\param_1]/{\bf J}^i)$.

 %It is clear that $p\not\in \Spec(\C[\param_1]/{\bf J})$ is spherical. We need to prove
%that $p\in \Spec(\C[\param_1]/{\bf J})$ is aspherical. Since $\Spec(\C[\param_1]/{\bf J})=\bigcup_i \Spec(\C[\param_1]/{\bf J}^i)$
%it is enough to prove that any $p\in \Spec(\C[\param_1]/{\bf J}^i)$ is aspherical.
Let $X^i$ be the associated variety of ${\bf I}^i$: by definition this is the support of the $\C[\h\oplus\h^*]^W$-module
$\gr {\bf H}_1/({\bf I}^i+\param_1^* {\bf H}_1)$. If $X^i=\{ 0\}$, then ${\bf H}_1/{\bf I}^i$ is finite over $\C[\param_1]/{\bf J}^i$.
It follows that for any $p\in \Spec(\C[\param_1]/{\bf J}^i)$ the algebra ${\bf H}_p/{\bf I}^i_p$ is nonzero.
By the construction of ${\bf I}^i$, we know that the image of $e$ in ${\bf H}_p/{\bf I}^i_p$ is zero. Thus  any $p\in
\Spec(\C[\param_1]/{\bf J}^i)$ is aspherical. If we let $\overline{\param}_1^{asph}$ denote the union of
$\Spec(\C[\param_1]/{\bf J}^i)$ with $X^i=\{0\}$ then we have that $\overline{\param}_1^{asph}$ is
a Zariski closed subset in $\param_1^{asph}$.

For a parabolic subgroup $\underline{W}$ let $\param^{asph}_1(\underline{W})\subset \param_1$ denote the aspherical locus
for the algebras ${\bf H}_p(\underline{W})$. According to \cite[Theorem 4.1]{BE}, we have $\param^{asph}_1(\underline{W})\subset
\param_1^{asph}$.  Let $\underline{\param}_1^{asph}$ be the union of $\param_1^{asph}(\underline{W})$ over all
proper parabolic subgroups $\underline{W}\subset W$.
By induction, we may assume that $\underline{\param}^{asph}_1$ is Zariski closed.

We now claim that $\param_1^{asph}=\overline{\param}_1^{asph}\cup \underline{\param}_1^{asph}$. If not
then there is $i$ such that $\Spec(\C[\param_1]/{\bf J}^i)$ is not contained in $\overline{\param}_1^{asph}\cup \underline{\param}_1^{asph}$ and we deduce in particular that $X^i\neq \{0\}$. Thus if we pick  a general point $p\in \Spec(\C[\param_1]/{\bf J}^i)$ then ${\bf H}_p e {\bf H}_p$ is an ideal
of infinite codimension in ${\bf H}_p$. Let $J$ be a minimal prime ideal of ${\bf H}_p e {\bf H}_p$ with infinite codimension. Thanks to \cite[Corollary 6.6]{ginprim}, there is an irreducible module $L$ in $\OCat_p$ with $J=\Ann_{{\bf H}_p}L$.
and hence there is a proper parabolic subgroup $\underline{W}\subset W$ such that $\Res^W_{\underline{W}}L\neq 0$.
It follows from end of the proof of \cite[Theorem 4.1]{BE}, that this latter module is aspherical. This contradicts $p\not\in \underline{\param}_1^{apsh}$, and so our claim, and hence the lemma, is proved.
\end{proof}

\subsection{Proof of Theorem \ref{proposition:sph_main_relation}: surjectivity}\label{SUBSECTION_VermaI}
\subsubsection{Subset $\param_1^{ss}\subset \param_1$}
Consider the subset $\param_1^{ss}\subset \param_1$ that consists of all
$p$ such that ${c}_\lambda(p)-{c}_\mu(p)\not\in \BZ\setminus \{0\}$
 for different $\lambda,\mu\in \irr{W}$. It is the complement to the union of countably
many hyperplanes. By \cite[Corollary 2.20]{GGOR} all
$\Delta_p(\lambda)$ are simple for $p\in \param^{ss}_{1}$. The algebra ${\bf H}_p$
is known to be simple for $p\in \param_1^{ss}$, this follows, for instance,
from results of \cite{ginprim}. So $\param_1^{ss}\subset \param_1^{sph}$. Therefore the simplicity
of $\Delta_p(\lambda)$  implies that of  $e\Delta_p(\lambda)$.

\subsubsection{Subset $\param_1(\lambda)\subset \param_1$}
\begin{lemma}\label{proposition:SphO5}
Let $\lambda\in\irr{W}$. The subset $$\param_1(\lambda):= \{ p\in \param_1 : e\Delta_p(\lambda) \text{ is generated by its lowest weight space}\}$$
is Zariski open and contains  $\param^{ss}_{1}$.
\end{lemma}
\begin{proof}
Let us fix an identification of ${\bf U}_p$ with $\C[\h\oplus\h^*]^W$ as follows.
Recall that $\C[\h]$ is a free graded $\C[\h]^W$-module and that there is a $W$-stable
graded  subspace ${\sf Harm}$ of harmonic polynomials inside $\C[\h]$ such that $\C[\h]={\sf Harm}\otimes
\C[\h]^W$. We have an isomorphism ${\sf Harm}\cong \C W$ of $W$-modules. Let ${\sf Harm}^*$
denote the space of harmonic elements in $S(\h)$. The triangular decomposition of ${\bf H}_p$ implies that ${\bf H}_p$ is
a free $\C[\h]^W\otimes \C[\h^*]^W$-module (with the first factor acting
from the left and the second factor acting from the right) with generating space ${\sf Harm}\otimes
\C W\otimes {\sf Harm}^*$. So the space $e({\sf Harm}\otimes \C W\otimes {\sf Harm}^*)e$
freely generates ${\bf U}_p$ as a $\C[\h]^W\otimes \C[\h^*]^W$-module, and this produces the identification of ${\bf U}_p$ with $\C[\h\oplus\h^*]^W$
we need.

We identify $\Delta(\lambda)$ with $\C[\h]\otimes \lambda$:
$W$ acts on $\C[\h]\otimes \lambda$ diagonally, $\C[\h]$
 by multiplication on the first tensorand, and for $y\in \h$ we have
$y.(f\otimes v)=[y,f]\otimes v$, where $f\in \C[\h], v\in \lambda$ and we note that $[y,f]\in \C[\h]\# W$. Now the operator $y:\C[\h]_i\otimes \lambda\rightarrow \C[\h]_{i-1}\otimes \lambda$
depends polynomially on $p$, so as a consequence of our identification of ${\bf U}_p$ with $\C[\h\oplus\h^*]^W$, we see that for any
$a\in \C[\h\oplus\h^*]^W(i)$ the operator
$l_a: e(\C[\h]_j\otimes \lambda)\rightarrow e(\C[\h]_{j+i}\otimes\lambda)$ depends polynomially on $p$.

Recall the  non-negative integer $\fd(\lambda^*)$.
Let $F = e(\C[\h]_{\fd(\lambda^*)}\otimes \lambda)$.
Then $F$ consists of lowest weight vectors in $e\Delta(\lambda)$, and $\param_1(\lambda)$ is the set
of all parameters such that $F$ generates $e\Delta(\lambda)=e(\C[\h]\otimes \lambda)$. But
as a $\C[\h]^W$-module,  $e(\C[\h]\otimes \lambda)$ is generated by $e({\sf Harm}\otimes \lambda)$.
It follows that $\param_1(\lambda)$ coincides with the set of parameters such that ${\bf U}_p(i) F=e(\C[\h]_i\otimes\lambda)$
for all $i$ between $1$ and $N$, where $N$ denotes the largest degree appearing in ${\sf Harm}$.
From the previous paragraph we deduce that $\param_1(\lambda)$ is Zariski open.

It remains to show that $\param^{ss}_{1}\subset \param_1(\lambda)$. But for $p\in \param_1^{ss}$ we know that $\Delta(\lambda)$
and hence $e\Delta(\lambda)$ are simple and the claim follows.
\end{proof}

\subsubsection{Proof of surjectivity}\label{SSS:surj_proof} Set $\param_1(W)=\bigcap_{\lambda\in \irr{W}}\param_1(\lambda)$.

\begin{lemma}\label{Cor:SphO5}
Let $p\in \param_1(W)$. Then for any $\lambda\in\irr{W}$ the
homomorphism $\varphi_p:\Delta^{sph}(\iota_p(\lambda))\rightarrow
e\Delta(\lambda)$  is surjective and so $\iota_p^L=\iota_p$.
\end{lemma}
\begin{proof}
The claim on the surjectivity is a consequence of Lemma \ref{proposition:SphO5}.
This implies $\iota^L_p=\iota_p$.
\end{proof}

\subsection{Proof of Theorem \ref{proposition:sph_main_relation}: injectivity}\label{SUBSECTION_VermaII}
In our proof we will need four lemmas.

\subsubsection{Subset $\param_1^{irr}\subset \param_1$}\label{sphss}
\begin{lemma}
There is a Zariski open subset $\param_1^{irr}\subset \param_1(W)\cap \param_1^{sph}$ such that ${\bf U}_p^0$ is semisimple for all $p\in \param_1^{irr}$, the set $\{ \iota_p(\lambda) : \lambda\in\irr{W}\}$ is a complete set of irreducible representations of ${\bf U}_p^0$, and $\dim\iota_p(\lambda)=\rank_{\C[\param]}\iota(\lambda)$. Moreover,  $\param_1^{ss}\subset \param_1^{irr}$ and $\varphi_p$ is an isomorphism for all $p\in \param_1^{ss}$.
\end{lemma}
\begin{proof}
Suppose that $p\in\param_1^{ss}$. Since $e\Delta_p(\lambda)$ is irreducible, any non-zero element $f\in \iota_p(\lambda) = (e\Delta_p(\lambda))_{\fd(\lambda^*)}$ generates $e\Delta_p(\lambda)$ and so ${\bf U}_p(0)f = \iota_p(\lambda)$. This shows that $\iota_p(\lambda)$ is an irreducible ${\bf U}^0_p$-module for all $p\in \param_1^{ss}$. Moreover, since $\mathcal{O}_p^{sph}$ is semisimple by \ref{sphequiv} it follows that the indecomposable $\Delta_p^{sph}(\iota_p(\lambda))$ is necessarily irreducible and hence that $\varphi_p : \Delta_p^{sph}(\iota_p(\lambda)) \rightarrow e\Delta_p(\lambda)$ is an isomorphism. We also note that since $\mathcal{O}_p^{sph}$ is semisimple for $p\in \param_1^{ss}$ the functors $\Delta_p^{sph}(\bullet)$ and $\bullet^{{\bf U}_p(<0)}$ insure that ${\bf U}_p^0$ is a semisimple algebra. Its dimension is necessarily $$\sum_{\lambda\in\irr{W}} \dim (\iota_p(\lambda))^2=\sum_{\lambda\in\irr{W}}(\rank_{\C[\param]} \iota(\lambda))^2.$$

We will prove in Step 2 of Lemma \ref{lemma:weak_flatness} that $\dim ({\bf U}_p^0)$ depends upper semi-continuously on $p$, so that there is a Zariski open set on which $\dim {\bf U}_p^0$ is minimal. (This result is independent of the material in between.) Applying \cite[Corollary 2.6]{gab} we see that there is a Zariski open subset $\param_1^{irr}\subset \param_1(W)$ on which all ${\bf U}_p^0$ are isomorphic to the same semisimple algebra. By the previous paragraph this open set contains $\param_1^{ss}$ and the semisimple algebra has simples labeled by $\lambda \in \irr{W}$. By Lemma \ref{SSS:surj_proof}, $\iota_p(\lambda)=\iota_p^L(\lambda)$
for all $p\in \param_1(W)$. Since the assignment $\lambda\mapsto\iota_p^L(\lambda)$ is a bijection between
$\irr{W}$ and $\irr{{\bf U}^0_p}$, we are done.
\end{proof}

We remark that for many $W$, including the family $G_{\ell}(n)$ for all $\ell$ and $n$, all the spaces $\iota(\lambda)$ are actually free $\C[\param]$-modules of rank $1$ and so in such cases the $\iota_p(\lambda)$ are irreducible for all $p\in\param_1$.

\subsubsection{Weak flatness}\label{lemma:weak_flatness}
One problem with the modules $\Delta_p^{sph}(\iota(\lambda))$ is that it is unclear whether they  depend
on $p$ in a flat way. However, we still can produce a much weaker statement. Consider the grading on
$\Delta^{sph}_p(\iota(\lambda))$ induced from the grading on ${\bf U}_p$, where we assume that
$\iota(\lambda)\subset \Delta^{sph}(\iota(\lambda))$ has degree 0.
Equip $e\Delta(\lambda)=e(\C[\h]\otimes \lambda)$ with the grading induced from the standard grading $\C[\h]$.  Then  $\varphi_p:\Delta^{sph}_p(\iota(\lambda))\rightarrow e\Delta(\lambda)$
is graded.

\begin{lemma}
For each $i$ consider the set $\param^i_1\subset \param_1^{irr}$ of all parameters $p$
such that
\begin{enumerate}
\item ${\bf U}_p^0$ is a semisimple algebra of dimension $\sum_{\lambda\in\irr{W}}(\rank_{\C[\param]} \iota(\lambda))^2$.
\item
the homomorphism $\Delta^{sph}_p(\iota(\lambda))\rightarrow e\Delta_p(\lambda)$
is an isomorphism on the $i$th graded component for all $\lambda\in \irr{W}$.
\end{enumerate}
Then
$\param^i_1$ is Zariski open and non-empty.
\end{lemma}
\begin{proof}

{\it Step 1.} For $p\in \param_1$ consider the (universal Verma) ${\bf U}_p$-module ${\bf \Delta}_p^{sph}:={\bf U}_p/{\bf U}_p{\bf U}_p(<0) = {\bf U}_p \otimes_{{\bf U}_p(\leq 0)} {\bf U}_p^0$.
We will first prove an upper semicontinuity  statement for its graded component $({\bf \Delta}_p^{sph})_i$.
According to Lemma \ref{lemma:Noetherian} we can pick a finite collection $f_1,\ldots, f_m$
of homogeneous generators of the left ideal ${\bf U}(<0)\subset {\bf U}(\leqslant 0)$
so that ${\bf \Delta}^{sph}={\bf U}/{\bf U}\langle f_1,\ldots, f_m\rangle$. Here and below the triangular brackets $\langle,\rangle$
mean the linear span. So $({\bf \Delta}^{sph})_i={\bf U}(i)/\sum_{j=1}^m {\bf U}({i-\deg f_j})f_j$.
We would like to say that the operator of the right multiplication by $f_j: {\bf U}_p({i-\deg f_j})\rightarrow
{\bf U}_p(i)$ depends polynomially on $p$ and so $\dim ({\bf \Delta}^{sph}_p)_i$ is upper semicontinuous.
Unfortunately this does not work because the spaces ${\bf U}_p(i)$ are infinite dimensional.

To fix this recall from the proof of Lemma \ref{proposition:SphO5} that we have a basis $a^{\alpha} v_k b^{\beta}$ of ${\bf U}$, where $a^{\alpha}:=a_1^{\alpha_1}\ldots a_n^{\alpha_n}, b^{\beta}:=b_1^{\beta_1}\ldots b_n^{\beta_n}$ with $a_1, \ldots , a_n$ and $b_1,\ldots, b_n$
being free homogeneous generators of $\C[\h]^W$ and $S(\h)^W$ respectively, and the $v_k$ comprise a basis in $e ({\sf Harm} \otimes \C W\otimes {\sf Harm}^*)e$. Furthermore, the operators $\ad(b_k)$ on ${\bf U}$ are locally nilpotent.
So there is $N\in \BZ_{>0}$ such that $b^{\beta} f_j\in {\bf U} \langle b_1,\ldots, b_n\rangle$ for all $j$
and $\beta$ with $|\beta|= N$. Thus $({\bf U}_p \langle b^\beta\rangle_{|\beta|=N})({i-\deg f_j})f_j\subset ({\bf U}_p\langle b_1,\ldots, b_n\rangle)(i)$ for all $i$.
So right multiplication by $f_j$ induces a linear map of finite dimensional
spaces $$({\bf U}_p/{\bf U}_p \langle b^\beta\rangle_{|\beta|=N})({i-\deg f_j})\rightarrow ({\bf U}_p/{\bf U}_p\langle b_1,\ldots,b_n\rangle)(i).$$
These two spaces don't depend on $p$ and the linear map in consideration depends
polynomially on $p$, which can be checked similarly to the statements in the proof
of Lemma \ref{proposition:SphO5}. Since ${\bf U}_p\langle b_1,\ldots,b_n\rangle \subset {\bf U}_p{\bf U}_p(<0)$, this proves
the claim in this step.

{\it Step 2.} The space $({\bf \Delta}_p^{sph})_0$ equals  ${\bf U}_p^0$. Thus $\dim({\bf U}_p^0)$ is upper semicontinuous for $p\in \param_1$. Thanks to this we can now invoke Lemma \ref{sphss} and deduce that $\param_1^{irr}$ is a Zariski open set where Property (1) of the
Lemma \ref{lemma:weak_flatness} holds.

{\it Step 3.} We may assume now that ${\bf U}_p^0$ is semisimple.
We have $\dim \Delta^{sph}_p(\iota(\lambda))_i\geqslant \dim e\Delta_p(\lambda)_i$
for all $i$, with the equality achieved, at least, for all $p\in \param_1^{ss}$. We have an
${\bf U}_p$-${\bf U}_p^0$-bimodule structure on $\Delta_p^{sph}$ and the  isotypic components for ${\bf U}_p^0$ are
precisely $\Delta^{sph}_p(\iota(\lambda))^{\dim \iota_p(\lambda)}$ for $\lambda\in \irr{W}$. So
$$\dim (\Delta_p^{sph})_i=\sum_{\lambda}{\dim \iota_p(\lambda)}\dim\Delta_p^{sph}(\iota(\lambda))_i\geq\sum_\lambda {\dim \iota_p(\lambda)} \dim e\Delta_p(\lambda)_i.$$
The right hand side is constant on $\param_1^{irr}$. Since the left hand side is upper semicontinuous, we see that $\param_1^i$
is Zariski open, and since equality holds for $p\in \param_1^{ss}$, we see that $\param_1^i$ is non-empty.
This completes the proof.
\end{proof}
\subsubsection{Homomorphism $\varphi'$}
\begin{lemma}\label{lemma:SphO_inject1}
Fix $p\in \param_1^{sph}$. For each $\lambda\in \irr{W}$ there is $\lambda'_p\in \irr{W}$ such that
there is a nonzero homomorphism $\varphi':e\Delta_p(\lambda'_p)\rightarrow \Delta^{sph}_p(\iota(\lambda))$
with the property that the lowest degree part of  $\varphi'(e\Delta(\lambda'_p))$ lies in $\Delta^{sph}_p(\iota(\lambda))_i$
with $i\leqslant |W|$.
\end{lemma}
\begin{proof}
Consider the ${\bf H}_p$-module $M:={\bf H}_pe\otimes_{{\bf U}_p}\Delta_p^{sph}(\iota(\lambda))$. We remark
that $\Delta^{sph}_p(\iota(\lambda))$ is naturally identified with $eM$. The element
$\euu$ acts diagonalizably on $M$ and its minimal eigenvalue $\alpha$ is less
than or equal to the minimal eigenvalue $\beta$ (in the sense that $\alpha-\beta\in \BZ_{\leqslant 0}$) of $\euu^{sph}$ on $\Delta_p^{sph}(\iota(\lambda))$. The $\alpha$-eigenspace is $W$-stable
and is annihilated by $\h$. Pick an irreducible $W$-submodule $\lambda'(=\lambda'_p)$ of this
space. So we get a nonzero homomorphism $\varphi:\Delta_p(\lambda')\rightarrow M$.
Since the parameter $p$ is spherical, the induced homomorphism
$e\Delta_p(\lambda')\rightarrow eM=\Delta_p^{sph}(\iota(\lambda))$
is nonzero as well. Therefore if we have a set $v_1,\ldots,v_k$ of generators
of $e\Delta_p(\lambda')$, then one of $\varphi(v_j)$ must be nonzero. As $e\Delta_p(\lambda')$
is generated as a $\C[\h]^W$-module by $({\sf Harm}\otimes \lambda')^W$, we may take $v_1,\ldots, v_k$
to be a homogeneous basis of this space. The maximal degree of an element in ${\sf Harm}$ is no greater than $|W|$, so the eigenvalue
of $\euu^{sph}$ on $v_k$, and hence on $\varphi(v_k)$, is less than or equal to $\alpha+|W|$.
This proves the lemma.
\end{proof}
\subsubsection{Subsets $\tilde{\param}^i_1\subset\param_1^{irr}$}
\begin{lemma}\label{lemma:SphO_singvect}
For $i>0$ consider the subset $\tilde{\param}^i_1\subset \param_1$ of parameters $p$
such that $e\Delta_p(\lambda)_i$ contains no nonzero vectors annihilated by ${\bf U}_p(<0)$.
Then the set $\tilde{\param}^i_1$ is Zariski open and non-empty.
\end{lemma}
\begin{proof}
The condition means that the intersection of the kernels of the actions on $e\Delta(\lambda)_i$ of the elements $f_1,\ldots,f_m\in {\bf U}$
introduced in the proof of Lemma \ref{lemma:weak_flatness}  is zero.
The restrictions of $f_j$'s to $e\Delta(\lambda)_i$ depend algebraically on $p$. So the set
$\tilde{\param}^i_1$ is Zariski open. To see that this set is non-empty assume that $p\in \param^{ss}_1$.
Then $e\Delta(\lambda)$ is simple and so the only vector annihilated by ${\bf U}_p(<0)$ lies in degree 0.
So $\param_1^{ss}\subset \tilde{\param}_1^i$ for all $i>0$.
\end{proof}

\subsubsection{Completion of the proof}
We claim that the homomorphism $\varphi:\Delta^{sph}_p(\iota(\lambda))\rightarrow e\Delta_p(\lambda)$
is bijective as long as $p$ lies in the Zariski open non-empty set
$$\bigcap_{i=1}^{|W|}\param_1^{i}\cap \tilde{\param}_1^i.$$

Indeed, we know that the homomorphism is injective in degrees $\leqslant |W|$.
So Lemma \ref{lemma:SphO_singvect} shows that the common kernel of ${\bf U}_p(<0)$ in $\Delta^{sph}_p(\iota(\lambda))_i$
for $0\leqslant i\leqslant |W|$ is $\iota_p(\lambda) = \Delta^{sph}_p(\iota(\lambda))_0$. Now we apply Lemma \ref{lemma:SphO_inject1} to see that the image of the homomorphism $\varphi':e\Delta_p(\lambda')\rightarrow \Delta^{sph}_p(\iota(\lambda))$ from Lemma \ref{lemma:SphO_inject1} must contain $\Delta^{sph}_p(\iota(\lambda))_0$. So $\varphi'$ and hence $\varphi\circ \varphi'$
are surjective. Since the parameter $p$ is spherical, $\varphi\circ\varphi'$ induces
an epimorphism $\Delta(\lambda')\rightarrow \Delta(\lambda)$. Therefore $\lambda'=\lambda$, and $\varphi\circ\varphi'$
is an isomorphism. Thus $\varphi$ and $\varphi'$ are isomorphisms.

\subsection{Spherical restriction functors}\label{SUBSECTION_sph_res}
\subsubsection{Isomorphism of completions}\label{SSS_sph_compl_iso}
Pick  a point $b\in \h$ and let $\underline{W}$ be the stabiliser of $b$ in $W$. Consider the  completions ${\bf U}^{\wedge_b}:=\C[\h/W]^{\wedge_b}\otimes_{\C[\h/W]}{\bf U}$
and ${\bf \underline{U}}^{\wedge_b}:=\C[\h/\underline{W}]^{\wedge b}\otimes_{\C[\h/\underline{W}]}{\bf \underline{U}}$,
where ${\bf \underline{U}}:=\underline{e} {\bf \underline{H}} \underline{e}$ is the spherical subalgebra in ${\bf \underline{H}}$. Let $\theta$ be a $\C[[\param]]$-linear
isomorphism ${\bf U}^{\wedge_b}\xrightarrow{\sim} {\bf \underline{U}}^{\wedge_b}$ with the following two
properties:
\begin{itemize}
\item[(i)] $\theta$ is $\C^\times$-equivariant with respect to the $\C^\times$-action on ${\bf U}$
induced from the action on ${\bf H}$ given by
$t.x=x, t.y=t y, t.w=w, t.a=ta, \quad x\in \h^*,y\in \h, w\in W, a\in \param^*$.
\item[(ii)] modulo $\param^*$ the isomorphism $\theta$ is a standard isomorphism $\C[\h/W]^{\wedge_b}\otimes_{\C[\h/W]}\C[\h\oplus \h^*]^W\xrightarrow{\sim}
\C[\h/\underline{W}]^{\wedge_b}\otimes_{\C[\h/\underline{W}]}\C[\h\oplus \h^*]^{\underline{W}}$.
\end{itemize}

As an example of a $\theta$ satisfying these conditions, we can take the restriction of $\vartheta^b$ from \eqref{def_theta} to  the spherical subalgebras (we remark that for any algebra $A$ having $\BC \underline{W}$ as a subalgebra
we have $\underline{e} A\underline{e}=eZ(W,\underline{W},A)e$). Furthermore, for any two isomorphisms $\theta,\theta'$
satisfying the properties above, there is $f\in \param^*\otimes \C[\h/\underline{W}]^{\wedge_b}$ such that
\begin{equation}\label{eq:compl_iso_conj}\theta'=
\exp(h^{-1}\operatorname{ad}(f))\circ\theta.\end{equation} This is proved analogously to  \cite[Lemma 5.2.1]{complos}.

\subsubsection{Definition of functor}\label{sphresdef}
Now the restriction functor $\,^{sph}\!\Res_b: \OCat_p^{sph}(W,\h)\rightarrow \OCat_p^{sph}(\underline{W},\underline{\h})$
is defined similarly to \ref{SSS_BE_functors1}, but instead of taking $\h$-locally nilpotent vectors we take vectors that
are locally nilpotent for the augmentation ideal in $S(\h)^{\underline{W}}$. Thanks to (\ref{eq:compl_iso_conj}), the restriction functor does not depend on the choice of $\theta$.

A vector in a ${\bf \underline{H}}$-module
is $\h$-locally nilpotent if and only if it is locally nilpotent with respect to the augmentation ideal
in $S(\h)^{\underline{W}}$. So if we choose $\theta$ induced from $\vartheta^b$ we see that the functors $M\mapsto \underline{e}\Res_b(M)$ and $M\mapsto \,^{sph}\!\Res^W_{\underline{W}}(eM)$ are naturally isomorphic.

\section{Equivalences from an $\mathfrak{S}_{\ell}$-action}\label{SECTION_sym_equiv}
In this section we take $W=G_{\ell}(n)$.
\subsection{Main result}\label{SUBSECTION_sym_main}
\subsubsection{$\mathfrak{S}_{\ell}$-action}\label{SSS_sym_action}
A nice feature of the cyclotomic case is that there is an action of $\mathfrak{S}_{\ell}$ on ${\bf U}$ by graded algebra automorphisms preserving $\param^*$, \cite[\S 6.4]{quant}.
To describe the $\mathfrak{S}_{\ell}$-action on $\param^*$ recall from \ref{cyclogener} the basis $h,\epsilon_0,\ldots,\epsilon_{\ell-1}$. The elements
$h,\sum_{i=0}^{\ell-1}\epsilon_i$ are $\mathfrak{S}_{\ell}$-invariant, while $\epsilon_1-\frac{1}{\ell}\sum_{i=0}^{\ell-1}\epsilon_i,\ldots,\epsilon_{\ell-1}-\frac{1}{\ell}\sum_{i=0}^{\ell-1}\epsilon_i$ transform as  the simple roots of the $A_{\ell-1}$-root system. 
%(with $(1,\ell)$ being the simple reflection corresponding to $\epsilon_1$).

Recall that in \ref{eq:c_formula} we defined a parametrization in terms of
$\kappa$ and ${\bf s}$. We have $$\Psi(\sum_{i=0}^{\ell-1}\epsilon_i)=\frac{1}{2}-\kappa, \Psi(\epsilon_i)= \kappa (s_{i}-s_{i-1})$$ so the $\mathfrak{S}_{\ell}$-action on $\param$ agrees with the $\mathfrak{S}_{\ell}$-action fixing $\kappa$
and permuting the $s_i$'s. In other words, the automorphism of ${\bf U}$ induced by the
$\mathfrak{S}_{\ell}$-action defines an isomorphism ${\bf U}_{1,\kappa, {\bf s}}\xrightarrow{\sim} {\bf U}_{1,\kappa, \sigma({\bf s})}$,
where $\sigma({\bf s})$ is obtained by applying the permutation $\sigma$ to ${\bf s}$.

Let us record two more important properties of the $\mathfrak{S}_{\ell}$-action. According to \cite[Proposition 6.7.1]{quant}, the spherical
Euler element $\euu^{sph}$ is $\mathfrak{S}_{\ell}$-invariant. So the $\mathfrak{S}_{\ell}$-action preserves the Euler grading on ${\bf U}$. In other words,
$\mathfrak{S}_{\ell}$ commutes with the $(\C^\times)^2$-action on ${\bf U}$ that is restricted from the
following action on ${\bf H}$:
\begin{equation}\label{eq:two_dim_action}
(t_1,t_2).x=t_1 x, (t_1,t_2).y=t_2 y, (t_1,t_2).w=w, (t_1,t_2).a=t_1t_2a, \quad x\in \h^*,y\in \h, w\in W, a\in \param^*.
\end{equation}
 Also the construction
of the action in \cite{quant} implies that it is trivial on $\C[\h\oplus\h^*]^W={\bf U}/(\param^*)$.

\subsubsection{Structure of $\param_1^{sph}$}\label{SSS_sph_descr}
The set $\param_1^{sph}$ was determined in \cite[Theorem 1.1]{DG}. The set of asphericalparameters, $\param_1^{asph}$, is the union of the following hyperplanes:
\begin{itemize}
\item $h_{H_{\bullet,\bullet}^{\bullet},0}-h_{H_{\bullet, \bullet}^{\bullet},1} = a/b$ with $1\leq a<b \leq n$;
\item there exists $0\leq u\leq \ell -1$, an integer $1-n\leq m \leq n-1$, and an integer $k$ such
that $k\not\equiv 0 \,(\ell)$,
$$ k=\ell (h_{H_{\bullet},u-k} - h_{H_{\bullet},u} + m (h_{H_{\bullet, \bullet}^{\bullet},1}-h_{H_{\bullet,\bullet}^{\bullet},0})) \quad \text{and} \quad 1\leq k\leq u + \left( \sqrt{n+\frac{1}{4}m^2} - \frac{1}{2}m-1\right) \ell.$$
\end{itemize} Rewriting this in terms of the parameters $\kappa$ and ${\bf s}$ we find $\param_1\setminus \param_1^{sph}$ is the union of the hyperplanes:
\begin{itemize}
\item $\kappa = a/b$ with $1\leq a<b \leq n$;
\item there exists $0\leq u\leq \ell -1$, an integer $1-n\leq m \leq n-1$, and an integer $k$ such
that $k\not\equiv 0 \,(\ell)$,
$$ k- \hat{k} =\kappa \ell (s_{u-k} - s_u -m) \quad \text{and} \quad 1\leq k\leq u + \left( \sqrt{n+\frac{1}{4}m^2} - \frac{1}{2}m-1\right) \ell,$$
\end{itemize} where $u+1-\ell \leq \hat{k} \leq u$ and $\hat{k} \equiv k \, (\ell)$.

\subsubsection{Equivalences $\Xi_{\sigma,p}$}\label{SSS_equiv_def}
Take $\sigma\in \mathfrak{S}_{\ell}$. It induces a grading preserving isomorphism ${\bf U}_p\rightarrow {\bf U}_{\sigma p}$ and hence
an equivalence $\Xi^{sph}_\sigma:\OCat^{sph}_p\rightarrow \OCat^{sph}_{\sigma p}$. The description of the aspherical parameters
in \ref{SSS_sph_descr} implies that $\param_1^{sph}$ is $\mathfrak{S}_{\ell}$-stable (alternatively, to see that $\param_1^{sph}$
is $\mathfrak{S}_{\ell}$-stable one can use the fact, see \cite[Theorem 5.5]{etingofaffine}, that $p\in \param_1^{sph}$
if and only if ${\bf U}_p$ has finite global dimension). So for $p\in \param_1^{sph}$ we can compose $\Xi^{sph}_\sigma$
with the equivalences $\OCat_p\xrightarrow{\sim} \OCat_p^{sph}$ and $\OCat^{sph}_{\sigma(p)}\xrightarrow {\sim}\OCat_{\sigma(p)}$ to get an equivalence
$\Xi_{\sigma,p}: \OCat_p\stackrel{\sim}{\longrightarrow}\OCat_{\sigma(p)}.$

\subsubsection{Main result}\label{SSS_sph_sym_main}
We define an action of $\mathfrak{S}_{\ell}$ on $\irr{G_\ell(n)}$ by permuting the components of the corresponding multipartition $\lambda = (\lambda^{(1)}, \ldots , \lambda^{(\ell)}).$

Recall that by the support of a module $M\in \OCat_p$ we mean the support of $M$
viewed as a $\C[\h]$-module. The main result of this section is then a slightly stronger version of Theorem A stated in the Introduction.

\begin{theorem}
Let $p\in \param_1^{sph}$ and $\kappa\neq 0,1$.
Then the  equivalence $\Xi_{\sigma,p}$ from \ref{SSS_equiv_def} maps $\Delta_p(\lambda)$ to
$\Delta_{\sigma p}(\sigma \lambda)$ and preserves the supports of the modules.
\end{theorem}

The claim about supports is easy, while the claim about the images of standards is more
difficult. Both will be proved in \ref{SS_sym_proof}, but the latter will require quite a lot
of preparation in \ref{SS_K_groups}. There we prove that $\Xi_{\sigma,p}$
induces the required map on the level of Grothendieck groups. Finally, at the end of \ref{SS_sym_proof} we will
provide a counter-example to the previous theorem when $p$ is an aspherical parameter.

Let us comment on the status of the restriction $\kappa\neq 0,1$. This technical
restriction appears in \ref{SSS_sym_final}, but we believe it can be removed.
In any case, if $\kappa=0$, then ${\bf H}_p\cong ({\bf H}^0_p)^{\otimes n}\# \mathfrak{S}_n$,
where ${\bf H}^0_p$ is the rational Cherednik algebra constructed from the group $\mu_\ell$. This case is easy as all questions can be reduced to $n=1$. The case $\kappa=1$ is, of course, more subtle.

\subsection{Maps on Grothendieck groups}\label{SS_K_groups}
Our goal in this subsection is to prove that the map $[\Xi_{\sigma,p}]:[\OCat_p]\rightarrow [\OCat_{\sigma p}]$
of the complexified Grothendieck groups sends $[\Delta_p(\lambda)]$ to
$[\Delta_{\sigma p}(\sigma\lambda)]$.
\subsubsection{Map $\nu$}
Let $\xi_\sigma$ be the linear map  $[\OCat_p]\rightarrow [\OCat_{\sigma p}]$
given by $[\Delta_p(\lambda)]\mapsto [\Delta_{\sigma p}(\sigma\lambda)]$. Set $\nu:=\xi_{\sigma}^{-1}\circ[\Xi_{\sigma,p}]$.
We need to prove that $\nu$ is the identity. We start by establishing two results saying that $\nu$ preserves
two linear maps from $[\OCat_p]$.

\subsubsection{$\nu$ vs spherical characters}
Consider a
module $N\in  \OCat^{sph}_p$. Recall, by Corollary \ref{lemma:SphO2}, that $\euu^{sph}$ has finite dimensional generalized eigenspaces in $N$:
let $d_\lambda(N)$ be the dimension of the generalized eigenspace corresponding to $\lambda$. So to
$M\in \OCat_p$ we can assign its spherical character $[M]^{sph}:=\sum_{\lambda\in \C}d_\lambda(eM)q^\lambda$,
where $q$ is a formal variable.
For example, we have
\begin{equation}\label{eq:Verma_sph_char}
[\Delta_p(\lambda)]^{sph}=q^{\hat{c}_\lambda(p)}[q^{-\operatorname{fd}(\lambda^*)}f_{\lambda^*}(q)]\prod_{i=1}^n ({1-q^{d_i}})^{-1}.
\end{equation}
Here  $f_{\lambda^*}(q), \operatorname{fd}(\lambda^*)$ and $\hat{c}_\lambda(p)$ were defined and computed
in \ref{fakecequation}  and $d_1,\ldots, d_n$ are the exponents
of the group $G_{\ell}(n)$. There is a well-defined map $\operatorname{ch}^{sph}$ on $[\OCat_p]$ which sends $[M]$ to $[M]^{sph}$.
Since the element $\euu^{sph}$ is $\mathfrak{S}_{\ell}$-invariant, the construction of $\Xi_{\sigma,p}$
implies that $\operatorname{ch}^{sph}(M)=\operatorname{ch}^{sph}(\Xi_{\sigma,p}(M))$ for any $M\in \OCat_p$.
We will need to modify the map $\operatorname{ch}^{sph}$. Namely, set $$\hat{\operatorname{ch}}(M)=\operatorname{ch}^{sph}(M)\prod_{i=1}^n (1-q^{d_i})\prod_{i=1}^n (1-q^{il})^{-1}$$ so that we have
\begin{equation}\label{eq:mod_sph_char}
\hat{\operatorname{ch}}(\Delta_p(\lambda))=q^{\hat{c}_\lambda(p)}\prod_{A\in\lambda}(1-q^{h(A)\ell})^{-1},
\end{equation}
where we denote the hook length of a box $A$ in $\lambda$ by $h(A)$. We see that $\hat{\operatorname{ch}}(\Delta_p(\lambda))=
\hat{\operatorname{ch}}(\Delta_{\sigma p}(\sigma \lambda))$. It follows that
\begin{itemize}
\item[(A)] $\hat{\operatorname{ch}}\circ\nu=\hat{\operatorname{ch}}$.
\end{itemize}

\subsubsection{Case $n=1$}\label{SSS_n1_case}
This case is easy. As a $\C[\param]$-algebra, ${\bf U}$
is generated by $x^{\ell},y^{\ell}, \euu^{sph}$ for non-zero $x\in \h^*, y\in \h$
with a relation of the form $[x,y]=P(\euu^{sph})$ for an appropriate polynomial $P(t)\in \C[\param][t]$
of degree $\ell$. We can take an arbitrary non-zero $\kappa$ and parameterize the points of $\param$
by the $\ell$-tuples ${\bf s}$. The sphericity condition becomes $s_i\neq s_j$ for $i\neq j$.

All $\lambda\in \irr{W}$ are 1-dimensional and $\Delta_p(\lambda)\cong \C[x]$ as   $\C[x]$-modules.
It follows that $e\Delta_p(\lambda)\cong \C[x^{\ell}]$ as  $\C[x^{\ell}]$-modules. In particular, $e\Delta_p(\lambda)$
is generated by its lowest weight subspace. Also all modules $\Delta_p^{sph}(N), N\in \irr{{\bf U}_p^0},$
are isomorphic to $\C[x^{\ell}]$ as
$\C[x^{\ell}]$-modules.

Assume now that $p\in \param_1^{sph}$.  From the previous paragraph it follows that the map $\iota_p:\irr{W}
\rightarrow \irr{{\bf U}_p^0}$ introduced in \ref{morphismsph} is a bijection, and $\varphi: \Delta_p^{sph}(\iota_p(\lambda))
\rightarrow e\Delta_p(\lambda)$ is an isomorphism.

Recall that $\sigma:{\bf U}_p\xrightarrow{\sim} {\bf U}_{\sigma p}$ preserves the Euler gradings. So we have an induced
isomorphism ${\bf U}_p^0\xrightarrow{\sim} {\bf U}_{\sigma p}^0$. It follows that $\Xi_{\sigma}^{sph}$ maps
each $\Delta_p^{sph}(N)$ to some $\Delta_{\sigma p}^{sph}(\sigma.N)$. So $\Xi_{\sigma,p}$ maps $\Delta_p(\lambda)$
to some $\Delta_{\sigma p}(\lambda')$. In particular, $\nu([\Delta_p(\lambda)])=[\Delta_p(\sigma^{-1}\lambda')]$. But using (\ref{eq:mod_sph_char}) and (A), we see that $\lambda'=\sigma\lambda$.

This, of course, implies $\nu=\operatorname{id}$.

\subsubsection{Inductive assumption} Thanks to \ref{SSS_n1_case} we can prove that $\nu=\operatorname{id}$ using induction.
So we assume that $\nu=\operatorname{id}$ for $n-1$ and will prove that the equality also holds for $n$.

\subsubsection{$\nu$ vs restriction}\label{SSS_eta_restr}
We would like to understand a relationship between the symmetric group equivalences $\Xi_{\sigma,p}$ and
the restriction functor $\Res^n_{n-1}:=\Res_{G_{\ell}(n-1)}^{G_{\ell}(n)}$.

\begin{lemma}
We have a natural equivalence $\Xi_{\sigma,p}\circ \Res^n_{n-1}\cong \Res^n_{n-1}\circ \Xi_{\sigma,p}$.
\end{lemma}
\begin{proof}
Pick $b\in \h$ with $W_b=G_{\ell}(n-1)$.
Thanks to remarks in \ref{sphresdef}, it is enough to prove that the functors $\Xi^{sph}_\sigma\circ \,^{sph}\!\Res_b,
\,^{sph}\!\Res_b\circ \Xi^{sph}_\sigma: \OCat^{sph}_p(W)\rightarrow \OCat^{sph}_p(W_b)$ are isomorphic.

Recall that the functor $^{sph}\!\Res_b$ does not depend (up to a natural equivalence) on the choice of an isomorphism
$\theta$ in \ref{SSS_sph_compl_iso}. From the properties of the $\mathfrak{S}_{\ell}$-action recalled at the end of
\ref{SSS_sym_action} it follows that $\sigma \theta \sigma^{-1}$ satisfies (i) and (ii) of \
\ref{SSS_sph_compl_iso} provided $\theta$ does. However it is clear that the functor $\,^{sph}\!\Res_b$
constructed from $\sigma \theta \sigma^{-1}$ coincides with the functor $\Xi^{sph}_\sigma\circ \,^{sph}\!\Res_b\circ (\Xi^{sph}_\sigma)^{-1}$.
\end{proof}

We claim that
\begin{itemize}
\item[(B)]  $[\Res^n_{n-1}]\circ \nu=[\Res^n_{n-1}]$.
\end{itemize}

Recall that $[\Res^n_{n-1}]: \C^{\mathscr{P}_\ell(n)}\rightarrow \C^{\mathscr{P}_\ell(n-1)}$
sends a multipartition $\lambda$ to the sum (with unit coefficients) of all multipartitions obtained from $\lambda$
by removing one box, see \cite[Corollary 3.12]{ArikiKoike}.
It follows that $\xi_\sigma\circ [\Res^n_{n-1}]=[\Res^n_{n-1}]\circ \xi_\sigma$.
The previous lemma implies  $[\Xi_{\sigma,p}]\circ [\Res^n_{n-1}]=[\Res^n_{n-1}]\circ [\Xi_{\sigma,p}]$. So
$\nu\circ [\Res^n_{n-1}]=[\Res^n_{n-1}]\circ \nu$. By the inductive assumption, $\nu=\operatorname{id}$
for $n-1$, and so (B) is proved.

\begin{remark}
Assume for a moment that Conjecture \ref{sph_gen_main} is true. The equivalence $\Xi^{sph}_\sigma$ sends
each $\Delta_p^{sph}(N)$ to some $\Delta_{p}^{sph}(N')$. It follows that $\Xi_{\sigma,p}$ sends standards
to standards. We remark that $\lambda\in \mathscr{P}_\ell(n)$ for $n>1$ can be uniquely recovered
from the set of multi-partitions obtained from $\lambda$ by removing a box. So the arguments of
\ref{SSS_n1_case} and \ref{SSS_eta_restr} imply $\Xi_{\sigma,p}(\Delta_p(\lambda))=\Delta_{\sigma p}(\sigma \lambda)$
modulo Conjecture \ref{sph_gen_main}.
\end{remark}

\subsubsection{Injectivity claim}
Now that we have seen that $\nu:[\OCat_p]\rightarrow [\OCat_p]$ satisfies (A) and (B) our next goal
is to prove that the only map with these properties is the identity. This amounts to checking
that  $([\Res^n_{n-1}], \hat{\operatorname{ch}})$  is injective.

Let us put a lexicographic ordering on the bases $[\Delta_p(\lambda)]$ of $[\OCat_p(n)]$ and $[\OCat_p(n-1)]$.
Namely, we represent $\lambda=(\lambda^{(1)},\ldots,\lambda^{(\ell)})$ as a collection of columns
$(\lambda^{(1)t}_1,\ldots, \lambda^{(1)t}_{d_1},\lambda^{(2)t}_1,\ldots, \lambda^{(\ell)t}_{d_\ell})$
and introduce the lexicographic ordering according to this representation. Then the largest
$\lambda'$ such that $[\Delta_p(\lambda')]$ appears in $[\Res^n_{n-1}(\Delta_p(\lambda))]$
is obtained from $\lambda$ by removing the rightmost removable box. We denote this largest
$\lambda'$ by $t(\lambda)$. It is easy to see that $\lambda>\mu$ implies $t(\lambda)\geqslant t(\mu)$.
So to prove the injectivity claim it is enough to show that for any $\lambda$ we have the following:
\begin{itemize}
\item[($\dagger$)] The elements $\hat{\operatorname{ch}}(\Delta_p(\mu))$ are linearly independent for all
$\mu$ with $\mu\leqslant \lambda, t(\mu)=t(\lambda)$.
\end{itemize}
Recall that we assume that $p$ is spherical, which, in particular, means that
all $s_i$'s are distinct.

It is not difficult to list all $\lambda$'s such that there is $\mu<\lambda$ with $t(\mu)=t(\lambda)$.
In the remainder of the subsection we discuss the most difficult combinatorially case. All other cases are obtained by some degeneration
of this one and are easier. Below $k$ denotes the maximal
index such that $\lambda^{(k)}\neq \varnothing$.

\subsubsection{Multipartitions $\mu_j$}
 Suppose that $\lambda^{(k)}_1=\lambda^{(k)}_2$. Let
$m$ be such that $\lambda_m^{(k)}=\lambda_1^{(k)}, \lambda_{m+1}^{(k)}<\lambda_1^{(k)}$.
Then $t(\lambda)$ is obtained from $\lambda$ by removing the box in the $m$-th row and the $\lambda_1^{(k)}$-th
column. The  multipartitions $\mu_k,\mu_{k+1},\ldots, \mu_{\ell}$ constructed below exhaust the
list of $\mu$'s with $t(\mu)=t(\lambda), \mu<\lambda$.

A multipartition $\mu_k$
can be described as follows: $\mu_k^{(i)}=\lambda^{(i)}$ for $i\neq k$, while $\mu_k^{(k)}$ is obtained
from $\lambda^{(k)}$ by moving the box in the $m$-th row and the $\lambda_1^{(k)}$-th column to the
first row. In other words, $(\mu_k^{(k)})_i=\lambda_i^{(k)}$ for $i\neq 1,m$, $(\mu_k^{(k)})_1=\lambda_1^{(k)}+1,
(\mu_2^{(k)})_m=\lambda_m^{(k)}-1$. Multipartitions $\mu_j, j>k$ are obtained as follows: $\mu_j^{(i)}=\lambda^{(i)}$
for $i\neq k,j$, $|\mu_j^{(j)}|=1$, while $\mu_j^{(k)}=t(\lambda)^{(k)}$, i.e., $\mu_j$ is obtained
from $\lambda$ by moving the  box in the $m$-th row and the $\lambda_1^{(k)}$-th column to the $j$-th partition.

We will prove that $f:=\hat{\operatorname{ch}}(\Delta_p(\lambda)), f_{j}:=\hat{\operatorname{ch}}(\Delta_p(\mu_j)),j\geqslant k, $
are linearly independent by assuming the converse, i.e., the existence of a
non-trivial linear combination
\begin{equation}\label{eq:lin_comb}\alpha f+ \sum_{j\geq k}\alpha_{j}f_{j}, \quad\alpha,\alpha_j\in \C,\end{equation}
equal to $0$. We remark that we can assume that the powers of $q$
appearing in the numerators of the elements $f_\bullet$ differ by an element of $\ell\mathbb{Z}$. Otherwise the linear
combination will split into parts according to the value of the power modulo $\ell\mathbb{Z}$ and we will
just need to deal with each part separately. Hence we can view the expressions $f_\bullet$
as rational functions in $q^{\ell}$. Below, to simplify the notation, we replace $q^{\ell}$ with $q$.

\subsubsection{Case $\lambda^{(k)}\neq (1,1,\ldots,1)$}
In this case the maximal hook length in $\mu^{(k)}_k$
is $\lambda_1^{(k)}+\lambda_1^{(k)t}$, while the maximal hook lengths in $\lambda^{(k)},\mu_j^{(k)},j>k$
are less than $\lambda_1^{(k)}+\lambda_1^{(k)t}$. Let $\epsilon$ be a primitive root of $1$
of order $\lambda_1^{(k)}+\lambda_1^{(k)t}$.
Then the order of pole at $\epsilon$ for $f_{k}$ is strictly bigger than the analogous orders for
the other $f_\bullet$. So $\alpha_{k}=0$.

Let us show that $\alpha=0$. Multiply the functions
$f, f_{j}, j>k,$ by the GCD of their denominators. Then the term attached to $f$ has the factor $(1-q^{\lambda_1^{(k)t}+\lambda_1^{(k)}-m})$
in its denominator, while the other functions do not. This implies the claim.

So we have a linear combination of $f_{j}, j>k,$ that is 0. Dividing the $f_{j}$'s by the common powers of $q$ in the numerators
and multiplying by the common denominator we get $\sum_{j>k} \alpha_j q^{\kappa s_j}=0$. This implies that
some of $s_j$'s are equal to each other, contradiction.

\subsubsection{Case $\lambda^{(k)}=(1,1,\ldots,1)$}\label{SSS_sym_final}
 In this case $\mu_k^{(k)}=(2,1\ldots,1)$.
We need to check
that the following  expressions are linearly independent (we divide $f$'s by the common factors of the numerators
and multiply by their common denominator):
$$g_\lambda:=q^{\kappa(s_k-m+1)+m-1}(1-q), g_{\mu_k}:=q^{\kappa(s_k+1)}(1-q^{m-1}), g_{\mu_j}:=q^{\kappa s_j}(1-q^m), j>k.$$
So (\ref{eq:lin_comb}) can be rewritten as
\begin{equation}\label{eq:lin_comb2}
\begin{split}
&\alpha(q^{\kappa(s_k-m+1)+m-1}-q^{\kappa(s_k-m+1)+m})=\\
&-\alpha_k(q^{\kappa(s_k+1)}-q^{\kappa(s_k+1)+m-1})-\sum_{j>k} \alpha_j(q^{\kappa s_j}-q^{\kappa s_j+m}).
\end{split}
\end{equation}
If $\alpha=0$, then we can plug a primitive $m$-th root of $1$ for $q$ and see that $\alpha_k=0$.
Then we get a contradiction as above. So $\alpha\neq 0$. After all possible cancellations in the right hand side of
(\ref{eq:lin_comb2}) we need to get $q^{\kappa(s_k-m+1)+m-1}$ with coefficient $\alpha$, $q^{\kappa(s_k-m+1)+m}$
with coefficient $-\alpha$ and no other powers of $q$. Let us make all cancellations in the $\sum_{j>k}$
part.
Recall that all the differences between the powers of $q$ will have to be integers. Let $M$ be the maximal difference
occurring after the cancelations in $\sum_{j>k}$. We claim that $M\leqslant m$. Assume the converse.
Then after the cancellation with the additional term in the right hand side there will be two powers of $q$
differing by more than $1$. This is a contradiction, so $M\leqslant m$.

The latter is only possible when we have only one nonzero $\alpha_j$ with $j>k$. We rewrite (\ref{eq:lin_comb2})
as
\begin{equation}\label{eq:lin_comb3}
-\alpha_j(q^{\kappa s_j}-q^{\kappa s_j+m})=
\alpha(q^{\kappa(s_k-m+1)+m-1}-q^{\kappa(s_k-m+1)+m})+\alpha_k(q^{\kappa(s_k+1)}-q^{\kappa(s_k+1)+m-1}).
\end{equation}
To get the required cancellation in the right hand side we must have either $\kappa(s_k-m+1)+m=\kappa(s_k+1)$ or $\kappa(s_k+1)+m-1=\kappa(s_k-m+1)+m-1$.
The first equality is equivalent to $\kappa=1$, while the second one is equivalent to $\kappa=0$.

This completes  (the inductive step of) the proof that $\nu=\operatorname{id}$.

\subsection{Proof of Theorem \ref{SSS_sph_sym_main}}\label{SS_sym_proof}
\subsubsection{Supports}
First, note that for $M\in \OCat_p$ we have $\operatorname{Supp}_{\C[\h]}M=\pi^{-1}(\operatorname{Supp}_{\C[\h]^W}M)$,
where $\pi:\h\rightarrow \h/W$ is  the quotient morphism. So we just need to check that $\operatorname{Supp}_{\C[\h]^W}M=
\operatorname{Supp}_{\C[\h]^W}\Xi_{\sigma,p}(M)$.

Next, notice that $\operatorname{Supp}_{\C[\h]^W}(eM)\subseteq \operatorname{Supp}_{\C[\h]^W}(M)$
for all $M\in \OCat_p$ and also that $\operatorname{Supp}_{\C[\h]^W}{\bf H}_{p} e\otimes_{{\bf U}_{p}}N\subseteq \operatorname{Supp}_{\C[\h]^W}N$ for all $N\in \OCat_{p}^{sph}$. It follows that for $p\in \param_1^{sph}$
these inclusions are equalities. So we have reduced the proof to checking that
$$\operatorname{Supp}_{\C[\h]^W}N=\operatorname{Supp}_{\C[\h]^W}\Xi^{sph}_{\sigma}(N), \quad N\in \OCat_p^{sph}.$$
To verify this it suffices to check that the automorphism $\sigma$ of ${\bf U}$ fixes $\C[\h]^W\subset {\bf U}$
pointwise.

Now in ${\bf U}$ the subalgebras $S(\h)^W$ and $\C[\h]^W$ are precisely the fixed point subalgebras for the
first and the second copy of $\C^\times$ under the action defined in (\ref{eq:two_dim_action}).
The subalgebras $S(\h)^W,\C[\h]^W\subset {\bf U}$ map isomorphically  into
their images in $\C[\h\oplus\h^*]^W$. But the $\mathfrak{S}_{\ell}$-action is trivial
modulo $\param^*$. It therefore follows that these subalgebras are pointwise fixed by the action of $\mathfrak{S}_{\ell}$.

The proof that $\Xi_{\sigma,p}$ preserves the supports is now complete.

\subsubsection{Images of standards}\label{SSS_high_wt_coinc}
Recall that both $\OCat_p$ and $\OCat_{\sigma p}$ are highest weight categories. Set $\Delta'_p(\lambda):=\Xi_{\sigma,p}^{-1}(\Delta_{\sigma p}(\sigma \lambda))$. Then, according to the previous subsection, $[\Delta'_p(\lambda)]=[\Delta_p(\lambda)]$. Moreover, both collections $\Delta_p(\bullet), \Delta_p'(\bullet)$
are the sets of standard objects for two, a priori different, highest weight structures on $\OCat_p$. The following
lemma shows these two highest weight structures must coincide and therefore completes the proof of Theorem \ref{SSS_sph_sym_main}.

\begin{lemma}
Let $\mathcal{C}$ be an abelian category over a field $k$ and let $\{\Delta_\lambda\}, \{\Delta'_\lambda\}$ be two collections
of objects indexed by a finite set $\Lambda$ that are the standard objects for two highest weight category structures on $\mathcal{C}$. If $[\Delta_\lambda]=[\Delta'_\lambda]$ for all $\lambda\in\Lambda$, then
$\Delta_\lambda\cong\Delta'_\lambda$.
\end{lemma}
\begin{proof}
Let $L_\lambda$ be the simple head of $\Delta_\lambda$. Then the set $\{ L_\lambda : \lambda\in\Lambda\}$ is a complete set of irreducible objects in $\mathcal{C}$. We will let $<$ denote the ordering on $\Lambda$ that defines the highest weight structure associated to the $\Delta_{\lambda}$'s. We will prove the lemma by induction.

To begin, suppose that $\lambda$ is minimal in the partial ordering $<$. Then $\Delta_\lambda$ is irreducible, so we deduce that $\Delta_\lambda \cong \Delta'_\lambda$.

Now assume that we have shown that $\Delta_\mu \cong \Delta'_\mu$ for all $\mu < \lambda$. Let $\mathcal{C}^{\leq \lambda}$ be the full subcategory of $\mathcal{C}$ whose objects have filtrations by $L_\mu$ with $\mu\leq \lambda$. By construction $\Delta_\lambda$ belongs to $\mathcal{C}^{\leq \lambda}$ and is a projective object in this category, see for example \cite[Proposition 4.13]{rouqqsch}. Since $[\Delta_\lambda] = [\Delta'_\lambda]$ we see that $\Delta'_\lambda$ also belongs to $\mathcal{C}^{\leq \lambda}$. Moreover, we know that the simple head of $\Delta'_\lambda$ must be $L_\lambda$, for otherwise it would be $L_\mu$ for some $\mu < \lambda$ and this is already the simple head of $\Delta'_\mu$ by induction. Hence we have a commutative diagram in $\mathcal{C}^{\leq \lambda}$
$$\xymatrix{ & \Delta_\lambda \ar[d] \ar@{.>}[dl] \\ \Delta'_\lambda \ar@{->>}[r] & L_{\lambda},
}
$$
where the induced morphism from $\Delta_\lambda$ to $\Delta'_{\lambda}$ is surjective. Thus, since they represent the same element of the Grothendieck group, $\Delta_\lambda$ and $\Delta'_\lambda$ are isomorphic. This completes the induction step. \end{proof}

\subsection{Counterexample for aspherical parameters}
Let us now provide a counterexample to Theorem \ref{SSS_sph_sym_main} for $p\in \param_1\setminus \param_1^{sph}$ --
we will see that there is no equivalence $\OCat_p\rightarrow \OCat_{\sigma p}$ sending
$\Delta_p(\lambda)$ to $\Delta_{\sigma p}(\sigma \lambda)$.

Let $\ell = 2$ and $n=2$ so that $W = G(2,1,2)$ is the Weyl group of type $B_2$. Take ${\bf s} = (-1,0)$ and $\kappa$ to be irrational. These are aspherical parameters (take $u=k=m=1$ in the final bullet point of \ref{SSS_sph_descr}). Then \cite[Theorems 3.2.3 and 3.2.4]{chmutova} shows that $\Delta_{\kappa, {\bf s}} ((1^2), \emptyset), \Delta_{\kappa, {\bf s}} (\emptyset, (1^2))$ and $\Delta_{\kappa, {\bf s}} (\emptyset, (2))$ are irreducible and the two other standard modules are reducible. On the other hand if we take $\sigma = (1 \, 2)$ so that ${\bf s}' = \sigma ({\bf s}) = (0,-1)$, then we find that  $\Delta_{\kappa, {\bf s}'} ((2), \emptyset), \Delta_{\kappa, {\bf s}'} (\emptyset, (1^2))$ and $\Delta_{\kappa, {\bf s}'} (\emptyset, (2))$ are irreducible.

\section{Derived equivalences}\label{SECTION_derived}
\subsection{Main result}
\subsubsection{Setting}
We fix $\ell>1,n\geqslant 1$ and take $W=G_\ell(n)= (\mu_\ell)^n\rtimes \mathfrak{S}_n$ acting on its reflection representation $\h=\C^n$. Pick $p,p'\in \param_1$ and let $(h_{H,j}),(h'_{H,j})$ for $(H,j)\in U$ be the parameters representing $p,p'$ as in \ref{SSS_alt_pres}. We say that $p,p'$ have {\it integral difference} if for each conjugacy class $H$ of reflection
hyperplanes there is $a_H\in \C$ with $h'_{H,j}-h_{H,j}-a_H\in \Z$ for all  $j$. This notion agrees with what we wrote in the introduction.

\subsubsection{Support filtrations}
We are interested in  the bounded derived categories $D^b(\OCat_p)$. These are triangulated categories that come
equipped with filtrations by triangulated subcategories -- the support filtrations. Namely, for a parabolic subgroup
$\underline{W}\subset W$ consider the full subcategory $D^b_{\underline{W}}(\OCat_p)$ consisting of all complexes
whose cohomology is supported on $W\h^{\underline{W}}$ where $\h^{\underline{W}} = \{ y\in \h: wy = y \text{ for all }w\in \underline{W}\}$.

\subsubsection{Main theorem}\label{der_main_thm}
\begin{theorem}
Suppose that the difference between $p,p'\in \param_1$ is integral. Then there is an equivalence $D^b(\OCat_p)\rightarrow D^b(\OCat_{p'})$
of triangulated categories that preserves the support filtration.
\end{theorem}

This theorem confirms \cite[Conjecture 5.6]{rouqqsch} for the groups $W=G_{\ell}(n)$.

\subsubsection{Symmetric group equivalence}\label{sym_derived} Recall the action of $\mathfrak{S}_{\ell}$ on $\param_1$ introduced in \ref{SSS_sym_action}. Together with Theorem \ref{SSS_sph_sym_main}, Theorem \ref{der_main_thm} implies the following claim.

\begin{corollary}
For any $p\in \param_1$ and $\sigma\in \mathfrak{S}_{\ell}$ there is an equivalence $D^b(\OCat_p)\rightarrow D^b(\OCat_{\sigma p})$ of triangulated categories
that preserves the support filtration.
\end{corollary}

To prove this corollary one notes that the symmetric group action preserves the integral difference and that for each $p$ there is $p'$ that differs integrally from $p$ and satisfies the assumptions of Theorem \ref{SSS_sph_sym_main}.

\subsubsection{The scheme of the proof}\label{SSS_derived_content}
Derived equivalences as in Theorem \ref{der_main_thm} are supposed to exist for
all complex reflection groups. The reason why we can deal with the $W=G_{\ell}(n)$ case is because of the existence of a symplectic resolution of singularities for $(\h\oplus\h^*)/W$ and the general notion of a quantization for a non-affine symplectic variety.
The resolution and its quantizations are constructed using hamiltonian reduction, which we recall in \ref{SUBSECTION_Ham}.

Quantizations of the  resolution are sheafified versions of the spherical subalgebras
of ${\bf H}_p$. To recover the whole algebra ${\bf H}_p$ we need a quantization of a special vector bundle
on the resolution -- a (weakly) Procesi bundle. With this we can introduce the notion of
the subcategories $\OCat$ in the categories of (sheaves of) modules over the quantizations. These relationships are recalled in \ref{SUBSECTION_RCA_vs_quant}.

A somewhat unpleasant issue about the sheaf categories $\OCat$ is that we can only define them
in the equivariant setting: we need a torus action corresponding to the Euler grading. So the global version of the sheaf category $\OCat$ is not the usual category $\OCat$
but its graded counterpart that has already appeared in \cite{GGOR}. Furthermore, to work
with quantizations we need to introduce a formal variable $\hbar$ and to work with algebras
and sheaves over $\C[[\hbar]]$. To get rid of $\hbar$ we need another ``contracting'' torus action. In
Subsection \ref{glob cats} we treat the tori actions and the corresponding algebras
and categories of equivariant modules.

A nice feature of the categories on the sheaf level, however, is that they are equivalent as abelian
categories for integrally different parameters. This comes from the possibility
of twisting by a line bundle and is described in detail in Subsection \ref{SUBSECTION_translations}.

Now the main point of our proof is that the categories on the sheaf level and on the
algebra level can be related. Classically, at the level of algebraic geometry, this is the derived McKay correspondence. Recall that, according to \cite{BeKa2},
the derived categories of $\C[\h\times\h^*]\# W$-modules and of coherent sheaves on the
resolutions are equivalent and the equivalences are expressed via a weakly Procesi bundle.
This statement can be quantized to produce equivalences between what we call the categories of
modules over the ``homogenized'' version of ${\bf H}_p$ and modules over the corresponding quantization.
The explicit form of the equivalences implies that the subcategories of all objects whose cohomology
are in $\OCat$ are also equivalent. This is all proved in  Subsection \ref{SUBSECTION_qMcKay}. We remark that
this subsection is quite technical partly because we need to establish certain basic homological
properties of the modules over quantizations that we were not able to find in the literature.

After the quantized derived McKay equivalence is dealt with, to establish the equivalence in the theorem we will proceed through several auxiliary equivalences listed in \ref{SUBSECTION_equiv_list}.
First, we get rid of $\hbar$ and one torus action. Second, we pass from  complexes whose cohomology are in $\OCat$
to complexes in $\OCat$. And third, we pass from the equivariant categories $\OCat$
to the usual ones. These three steps are taken in Subsection \ref{SUBSECTION_derived_complet}.

\subsubsection{Ramifications} Theorem \ref{der_main_thm} can be carried over to the $\ell=1$ case in a straightforward way.
Also, with the same techniques one can prove an equivalence of derived categories $D^b({{\bf H}_p}\md)\stackrel{\sim}{\rightarrow}
D^b({{\bf H}_{p'}}\md)$, where the ${\bf H}_\bullet$'s are the symplectic reflection algebras corresponding to
arbitrary wreath products of symmetric groups and kleinian groups. By definition, parameters $p,p'$ have integral
difference if  $p-p'$ is a character of an appropriate reductive algebraic group, compare with  \ref{int_diff_rev}
below. This equivalence preserves the subcategories with finite dimensional cohomology.

\subsection{Auxiliary algebras and categories of modules}\label{glob cats}
To prove the existence of an equivalence in Theorem \ref{der_main_thm} we will establish
equivalences between several related categories. It is more convenient for us to
work with right modules, but since  $\OCat$ is isomorphic to $\OCat^{op}$
via a duality, see \cite[\S 4.2.1]{GGOR}, and the duality preserves the integral difference, we can indeed do this.
\subsubsection{Torus actions and equivariant modules}\label{actionsonH}
Set $T= T_1\times T_2 := \C^\times \times \C^\times $. Let $\hbar$ be a formal variable and pick $p\in \param_1$.
We let ${\bf H}_{p\hbar}$ denote the $\C[\hbar]$-algebra obtained by specializing ${\bf H}$ at
$p\hbar$ (so that $h\mapsto \hbar$) and let ${\bf H}_{p\hbar}^{\wedge_\hbar}$ be its $\hbar$-adic completion.
These algebras come equipped with an action of $T$ given by
\begin{equation}\label{eq:de_action1}
(t_1,t_2).x=t_1t_2x, (t_1,t_2).y=t_1^{-1}t_2y, (t_1,t_2).w=w, (t_1,t_2).\hbar=t_2^2\hbar \end{equation} where $(t_1,t_2)\in T, x\in \h^*, y\in \h$ and $w\in W$.

We let $\mods^{T}-{\bf H}_{p\hbar}$ denote the category of (algebraic) $T$-equivariant finitely generated right modules over ${\bf H}_{p\hbar}$ and $\mods^{T}-{\bf H}^{\wedge_\hbar}_{p\hbar}$
denote the category of (proalgebraic) $T$-equivariant finitely generated  right modules over ${\bf H}^{\wedge_\hbar}_{p\hbar}$. More precisely, an object $M$ of $\mods^{T}-{\bf H}_{p\hbar}$ is a finitely generated right ${\bf H}_{p\hbar}$-module together with a coaction $\Delta_M: M\rightarrow M\otimes \C[T]$ which satisfies $\Delta_M(mz) = \Delta_M(m)\Delta_{{\bf H}_{p\hbar}}(z)$ for $m\in M, z\in {\bf H}_{p\hbar}$ and where $\Delta_{{\bf H}_{p\hbar}}$ denotes the coaction of $T$ on ${\bf H}_{p\hbar}$. Morphisms are homomorphisms of ${\bf H}_{p\hbar
}$-modules compatible with the coactions. Clearly $T$ acts on any $T$-equivariant ${\bf H}_{\hbar p}$-module
$M$ by $m\cdot (t_1,t_2) = \sum m_{(0)} m_{(1)}(t_1,t_2)$ where we use the Sweedler notation $\Delta_M(m) = \sum m_{(0)} \otimes m_{(1)}$. The space $M$ is algebraic for the action of $T$, meaning that it decomposes as a direct sum of $T$-eigenspaces.
An object $N$ of $\mods^{T}-{\bf H}^{\wedge_\hbar}_{p\hbar}$ is defined similarly using a coaction $\Delta_N: N\rightarrow N\hat{\otimes} \C[T]$ where $N\hat{\otimes} \C[T]$ denotes the completed tensor product $\varprojlim(N/N\hbar^n\otimes \C[T])$. In this case each $M/\hbar^nM$ is algebraic for the action of $T$ and $M$ is the inverse limit of these spaces. The latter is proved analogously to
\cite[Lemma 2.4.4]{HC}.

Since the grading on ${\bf H}_{p\hbar}$ induced by the $T_2$-action is positive, the functor of $\hbar$-adic completion is an equivalence $$\mods^{T}-{\bf H}_{p\hbar}\stackrel{\sim}{\longrightarrow} \mods^{T}-{\bf H}^{\wedge_\hbar}_{p\hbar}.$$ Its inverse is given by sending an object of ${\bf H}^{\wedge_\hbar}_{p\hbar}$ to its largest rational $T_2$-submodule, that is to its subspace spanned by the eigenvectors for the action of $T_2$. The proof that this is indeed inverse is standard and repeats the proof of (1) and (2) in \cite[Proposition 3.3.1]{HC}.

\subsubsection{Categories $\OCat$} We let $\OCat_{p\hbar}^{T}\subset \mods^{T}-{\bf H}_{p\hbar}$ denote the full subcategory consisting of all objects that are finitely generated over $\C[\h][\hbar]$ and
$\OCat_{p\hbar}^{\wedge_\hbar, T}\subset \mods^{T}-{\bf H}^{\wedge_\hbar}_{p\hbar}$ the full subcategory consisting of all objects that are finitely generated over $\C[\h][[\hbar]]$.

The $T_1$-action descends to ${\bf H}_p = {\bf H}_{p\hbar}/(\hbar - 1 )$, so we have a category $\mods^{T_1}-{\bf H}_p$ of $T_1$-equivariant finitely generated ${\bf H}_p$-modules. The full subcategory $\OCat_{p}^{T_1}$
consisting of all modules that are finitely generated over $\C[\h]$ is just the $T_1$-equivariant
lift of $\OCat_p$: $T_1$-equivariant (or equivalently graded) modules that belong to $\OCat_{p}$.

\subsection{Quantization via hamiltonian reduction}\label{SUBSECTION_Ham}
We will require a $T$-equivariant resolution of singularities $\pi : X\longrightarrow (\h\oplus \h^*)/W$ and its quantizations.

\subsubsection{Classical hamiltonian reduction}
Consider the affine Dynkin quiver $Q$ of type $\widetilde{A}_{\ell-1}$ with $\ell$ vertices labelled $\{ 0, \ldots , \ell -1\}$.
Let $Q^{\sf CM}$ be its {Calogero-Moser} extension: this has one additional vertex $s$ and one new
arrow which starts at $s$ and ends at the extending vertex $0$ of $Q$. Consider the representation space
$V:=\operatorname{Rep}(\overline{Q}^{\sf CM}, n\delta+\epsilon_s)$, where $\overline{Q}^{\sf CM}$ is the doubled
quiver of $Q^{\sf CM}$ and $\delta$ is the indecomposable imaginary root for $Q$, $\delta=\sum_{i=0}^{\ell-1}\epsilon_i$.
The space $V$ comes equipped with a symplectic form and a hamiltonian action of
the group $\GL(n\delta):=\prod_{i=0}^{\ell-1}\GL(n)$, where the $i^{\text{th}}$ copy of $\GL(n)$ acts on $V$ by base-change at the $i^{\text{th}}$ vertex. A moment map $\mu: V\rightarrow \gl(n\delta)$ results.

There is an action of $T$ on $V$ induced by $$(t_1,t_2)\cdot M_ a =\begin{cases} t_1t_2 M_a \qquad & \text{if $a$ is an arrow of $Q^{\sf CM}$}, \\ t_1^{-1}t_2 M_a & \text{if $a$ is an arrow of $\overline{Q}^{\sf CM}\setminus Q^{\sf CM}$},\end{cases}$$ where $M_a$ denotes the transformation associated to the arrow $a$ at a point $M\in V$. The moment map is $T_1$-invariant and satisfies $\mu(t_2\cdot v)=t_2^2\mu(v)$.

\subsubsection{Resolution of singularities}
One can fix a ``general enough'' character $\chi:\GL(n\delta)\rightarrow \C^\times$ and use it to
specify the semistable locus $V^{ss}\subset V$. Then for $X$ we take the hamiltonian reduction
$\mu^{-1}(0)^{ss}/\GL(n\delta)$, where $\mu^{-1}(0)^{ss}=\mu^{-1}(0)\cap V^{ss}$.
The resolution morphism $\pi: X\rightarrow (\h\oplus\h^*)/W$ is the natural projective morphism $X\twoheadrightarrow \mu^{-1}(0)/\!/\GL(n\delta):=\operatorname{Spec}(\C[\mu^{-1}(0)]^{\GL(n\delta)})$
followed by the identification of $\mu^{-1}(0)/\!/\GL(n\delta)$ with $(\h\oplus\h^*)/W$, see \cite[Theorem 1.1]{CB2} combined with \cite[Lemma 9.2]{CB1}.

 The $T$-action
on $V$ induces a $T$-action on $X$ and on $\mu^{-1}(0)/\!/\GL(n\delta)$ and $\pi$ is $T$-equivariant. Under the identification of $ \mu^{-1}(0)/\!/\GL(n\delta)$ with $(\h\oplus\h^*)/W$ the $T$-action corresponds to the one on $\h\times \h^*$ introduced in \eqref{eq:de_action1}.

\subsubsection{Quantizations}
Recall that  $(t_1,t_2).\hbar=t_2^2 \hbar$ for $(t_1,t_2)\in T$. By a $T$-equivariant quantization of $X$ we mean a pair $(\Dcal_\hbar,\theta)$,
where $\Dcal_\hbar$ is a $T$-equivariant sheaf of flat $\C[[\hbar]]$-algebras that
is complete in the $\hbar$-adic topology and such that $\theta: \Dcal_\hbar/\hbar \Dcal_\hbar\xrightarrow{\sim}\Str_X$ is a $T$-equivariant
isomorphism of Poisson sheaves. Up to appropriate equivalence, the quantizations are naturally
parameterized by $H^2_{dR}(X)$, the $2^{\text{nd}}$ de Rham cohomology of $X$, see \cite[Theorem 1.8]{BeKa} and \cite[Corollary 2.3.3]{quant}. For a class $\alpha\in H^2_{dR}(X)$ we let $\Dcal_\hbar^\alpha$
denote the corresponding quantization.

\subsubsection{Quantum hamiltonian reduction} We can give an explicit realization of the quantizations $\Dcal^{\alpha}_{\hbar}$ thanks to the quiver theoretic description of $X$ via ``quantum hamiltonian reduction".

Set $\z:=(\gl(n\delta)^*)^{\GL(n\delta)}$. The Duistermaat-Heckman
theorem gives a linear map from $\z$ to $H^2_{dR}(X)$, see \cite[\S 3.2]{quant}. This map is an isomorphism and we use it to identify $H^2_{dR}(X)$ with $\z$.

Now consider the sheaf $\W_{\hbar,V}$ of $\C[[\hbar]]$-algebras obtained by first $\hbar$-adically completing the homogenized
Weyl algebra $\W_\hbar(V)$ and then localizing on $V$. We have a natural ``symmetrized'' lifting $\mu^*: \gl(n\delta)\rightarrow \W_\hbar(V)$
of the comorphsim of the moment map $\mu^*: \gl(n\delta)\rightarrow \C[V]$. We then form the reduction
$$ (\W_{\hbar, V^{ss}}/ \W_{\hbar,V^{ss}}\mu^{*}\gl(n\delta)_{\hbar\alpha})^{\GL(n\delta)}$$ where $\alpha\in \z$ and $\mu^*\gl(n\delta)_{\hbar\alpha} :=\{\xi-\hbar\langle\alpha,\xi\rangle : \xi\in \gl(n\delta)\}$. According to \cite[Proposition 3.2.1]{quant}, this reduction is a quantization of $X$. Moreover, \cite[Corollary 2.3.3 and Theorem 5.4.1]{quant} imply that the reduction   is isomorphic to $\Dcal^\alpha_\hbar$.

\subsubsection{Categories of $\Dcal_\hbar^\alpha$-modules}
Let $\mods^{T}-\Dcal_\hbar^\alpha$ denote the category
of (proalgebraic) $T$-equivariant coherent right $\Dcal_\hbar^\alpha$-modules. The subvariety $\pi^{-1}(\h/W)\subset X$ is a $T$-stable lagrangian subvariety. We then have the full
subcategory $\OCat_\hbar^{\alpha,T}$ of all modules supported on $\pi^{-1}(\h/W)$.

\subsection{${\bf H}_{p\hbar}^{\wedge_\hbar}$ vs $\Dcal_\hbar^\alpha$.}\label{SUBSECTION_RCA_vs_quant}
Here we will follow \cite{quant} to relate the algebras ${\bf H}^{\wedge_\hbar}_{p\hbar}$
 and the sheaves $\Dcal_\hbar^\alpha$.
\subsubsection{Weakly Procesi bundle} According to \cite[Theorem 2.3]{BeKa2} on $X$ there is a $T$-equivariant vector bundle $\mathcal{P}$ -- a weakly
Procesi bundle in the terminology of \cite[\S 4.4]{quant} -- having the following two properties:
\begin{itemize}
\item[(P1)] one has a $T$-equivariant
isomorphism $\operatorname{End}_{\Str_X}(\Pro,\Pro)\cong \C[\h\oplus \h^*]\# W$ of $\C[\h\oplus\h^*]^W=\C[X]$-algebras;
\item[(P2)] $\operatorname{Ext}^i_{\Str_X}(\Pro,\Pro)=0$ for $i>0$.
\end{itemize}
Thanks to (P2) we can deform $\Pro$ to a locally free right $\Dcal_\hbar^\alpha$-module $\Pro^\alpha_\hbar$.
This module is still $T$-equivariant and satisfies
\begin{itemize}
\item[(P$_\hbar$1)]  $\operatorname{End}_{\Dcal_\hbar^\alpha}(\Pro^\alpha_\hbar,\Pro^\alpha_\hbar)/(\hbar)=
\operatorname{End}_{\Str_X}(\Pro,\Pro)\cong \C[\h\oplus \h^*]\# W$,
\item[(P$_\hbar$2)] $\operatorname{Ext}^i_{\Dcal_\hbar^\alpha}(\Pro^\alpha_\hbar,\Pro^\alpha_\hbar)=0$ for $i>0$.
\end{itemize}
Thus we may consider the sheaf $\widetilde{\Dcal}_\hbar^\alpha$ of local endomorphisms of the right $\Dcal_\hbar^\alpha$-module
$\Pro^\alpha_h$, see \cite[\S 6.3]{quant}. This is a $T$-equivariant sheaf of algebras.

\subsubsection{$\Z/2\Z\times \mathfrak{S}_\ell$-action}\label{paramchange} We identify $\C^{\{0,\ldots,\ell-1\}}$ with $\z$
via $(x_0,\ldots,x_{\ell-1})\mapsto \sum_{i=0}^{\ell-1}x_i \operatorname{tr}_i$, where $\operatorname{tr}_i$ is
the trace at the $i^{\text{th}}$ vertex. We can then identify $\param^*$ and $\C \hbar\oplus \z^*$ as follows. Let $\epsilon_0,\ldots,\epsilon_{\ell-1}$ be the standard basis in $\z^*$
induced by the identification with $\C^{\{0,\ldots,\ell-1\}}$. An isomorphism $\upsilon:\C \hbar\oplus \z^*\rightarrow \param^*$
is basically the same as in \ref{SSS_cyclot_param}, it is given by
\begin{align*}
&\hbar\mapsto h,\\
&\epsilon_i\mapsto \frac{1}{\ell}(h-2\sum_{j=1}^{\ell-1}c_i \zeta_{\ell}^{ji}), \quad i=1,\ldots,\ell-1\\
&\epsilon_0\mapsto \frac{1}{\ell}(h-2\sum_{j=1}^{\ell-1}c_i)+c_0-t/2.
\end{align*}
There is an action of $\Z/2\Z\times \mathfrak{S}_{\ell}$
on $\C \hbar\oplus \z^*$: the action of $\mathfrak{S}_{\ell}$  has already appeared in Section \ref{SECTION_sym_equiv}, it  fixes $\delta=\sum_{i=0}^{\ell-1}\epsilon_i$ and $\hbar$,
preserves the subspace $\delta^{\perp}=\{\sum_{i=0}^{\ell-1}x_i \epsilon_i| \sum_{i=0}^{\ell-1}x_i=0\}$ and acts
on it such that the projections of $\epsilon_1,\ldots,\epsilon_{\ell-1}$ transform as the simple roots in the root
system $A_{\ell-1}$; the action of $\Z/2\Z$ is trivial on $\hbar$ and $\delta^\perp$ and is sign change on $\C \delta$.

\subsubsection{Isomorphism theorem}
The action of $W$ on $\mathcal{P}$ extends to produce a distinguished embedding $\C W\hookrightarrow \Gamma(X,\widetilde{\Dcal}_h^\alpha)$. The results of \cite[\S 6]{quant} may then be interpreted as follows.

\begin{theorem}\label{prop:param_id}
Given $\alpha\in \z = H^2_{dR}(X)$, define $\widehat{\alpha}: \C\hbar \oplus \z^*\rightarrow \C$ by sending $\hbar$ to $1$ and $f\in \z^*$ to $f(\alpha)$. Then there is $w\in \Z/2\Z\times \mathfrak{S}_{\ell}$ such that there is a $\C[[\hbar]]W$-linear and $T$-equivariant isomorphism
$\Gamma(X, \widetilde{\Dcal}^\alpha_\hbar)\cong {\bf H}^{\wedge_\hbar}_{ \upsilon^{\ast-1}(\widehat{w\alpha})\hbar}$ which is the identity modulo $\hbar$.
\end{theorem}
The element $w$ depends on the choice of $\mathcal{P}$. We will denote the map $\alpha\mapsto \upsilon^{*-1}(\widehat{w\alpha})$ by $\iota$.

\subsection{Equivalences}\label{SUBSECTION_equiv_list}
We can now list the equivalences that will be combined to prove Theorem \ref{der_main_thm}.

\subsubsection{Quantized McKay correspondence} We will show first of all that there is a quantized McKay correspondence, meaning an equivalence of bounded derived categories
\begin{equation}\label{eq:equiv_McKay}
D^b(\operatorname{mod}^{T}-\Dcal^\alpha_\hbar)\xrightarrow{\sim} D^b(\operatorname{mod}^{T}-{\bf H}^{\wedge_\hbar}_{\iota(\alpha)\hbar}).
\end{equation}
Its construction will imply that it induces
an equivalence \begin{equation} \label{eq:equiv_McKay1}D_{\OCat}^b(\operatorname{mod}^{T}-\Dcal^\alpha_\hbar)\xrightarrow{\sim} D^b_{\OCat}(\operatorname{mod}^{T}-{\bf H}^{\wedge_\hbar}_{\iota(\alpha)\hbar})\end{equation} between the subcategories of
complexes whose cohomology lies in the categories $\OCat$. Moreover, we will see that this equivalence
preserves the support filtrations.

\subsubsection{Translations} If $p = \iota(\alpha)$ and $p'= \iota(\alpha')$ have integral difference, we will prove that there is a support and $\OCat$-preserving equivalence of abelian categories
\begin{equation}\label{eq:equiv_trans}
\operatorname{mod}^{T}-\Dcal^\alpha_\hbar\xrightarrow{\sim} \operatorname{mod}^{T}-\Dcal^{\alpha'}_\hbar.
\end{equation}
Combined with (\ref{eq:equiv_McKay}) this gives a support preserving equivalence
\begin{equation}\label{eq:equiv_O1}
D^b_\OCat(\operatorname{mod}^{T}-{\bf H}^{\wedge_\hbar}_{p \hbar})\xrightarrow{\sim}
D^b_\OCat(\operatorname{mod}^{T}-{\bf H}^{\wedge_\hbar}_{p'\hbar}).
\end{equation}
We will then see that in (\ref{eq:equiv_O1}) one can get remove $\hbar$ and part of the $T$-action
to  produce an equivalence
\begin{equation}\label{eq:equiv_O2}
D^b_\OCat(\operatorname{mod}^{T_1}-{\bf H}_{p})\xrightarrow{\sim}
D^b_\OCat(\operatorname{mod}^{T_1}-{\bf H}_{p'}).
\end{equation}
\subsubsection{Equivalences between categories $\OCat$}
Our next step will be to prove that the natural functor
\begin{equation}\label{eq:equiv_O4}
D^b(\OCat_p^{T_1})\rightarrow D^b_\OCat(\operatorname{mod}^{T_1}-{\bf H}_{p})
\end{equation}
is an equivalence. With this in hand, the final step will be to get rid of $T_1$-equivariance and thus  to obtain  an equivalence
\begin{equation}\label{eq:equiv_O5}
D^b(\OCat_p)\xrightarrow{\sim} D^b(\OCat_{p'}).
\end{equation}
Tracking this construction we will see that this equivalence has the claimed properties.

\subsection{Quantized McKay correspondence}\label{SUBSECTION_qMcKay}  In this section we are going to produce the equivalence (\ref{eq:equiv_McKay}).

\subsubsection{Functors $\mathcal{F}_\hbar,\mathcal{G}_\hbar$} Set $p=\iota(\alpha)$. We are going to show that the functor $$\mathcal{F}_\hbar:=\operatorname{RHom}(\Pro^\alpha_\hbar,\bullet): D^b(\operatorname{mod}^T-\Dcal^\alpha_\hbar)
\rightarrow D^b(\operatorname{mod}^T-{\bf H}^{\wedge_\hbar}_{p\hbar})$$ is an equivalence in (\ref{eq:equiv_McKay}) with the required
properties. Here we take $\operatorname{RHom}$ in $\operatorname{mod}-\Dcal_{\hbar}^{\alpha}$. A quasi-inverse functor will be $\mathcal{G}_\hbar:=\bullet\otimes^L_{{\bf H}_{p\hbar}^{\wedge_h}}\Pro^\alpha_\hbar$.
First we are going to check that the functors $\mathcal{F}_\hbar,\mathcal{G}_\hbar$ are well-defined.

\subsubsection{Quasi-coherent $\Dcal^\alpha_\hbar$-modules}
Let $\operatorname{Mod}^T-\Dcal^\alpha_\hbar$ be the category of $T$-equivariant {\it quasi-coherent} $\Dcal^\alpha_\hbar$-modules that are,
by definition, the direct limit of their $T$-equivariant coherent $\Dcal_\hbar^\alpha$-submodules. This category has enough injectives. Indeed if $M$ is an object in $\operatorname{Mod}^T-\Dcal^\alpha_\hbar$ we may consider it as a quasi-coherent $\Dcal^{\alpha}_\hbar\# U(\mathfrak{t})$-module, compare with \cite{vdb}, and $M$ has an injective hull $\hat{I}(M)$ inside this larger category. Then $I(M)$, the maximal $T$-equivariant quasi-coherent $\Dcal_\hbar^\alpha$-submodule of $\hat{I}(M)$, is an injective hull in $\operatorname{Mod}^T-\Dcal^\alpha_\hbar$.

It is this larger category we would like to use to derive $\operatorname{Hom}(\Pro^\alpha_\hbar, \bullet)$. The following lemma, which is proved in the same manner as for instance \cite[Proposition 3.5]{huy}, allows us to do this.
\begin{lemma}
Let $D^b_{coh}(\operatorname{Mod}^T-\Dcal^\alpha_\hbar)$ denote the full subcategory of $D^b(\operatorname{Mod}^T-\Dcal^\alpha_\hbar)$ consisting of complexes whose cohomology belongs to $\operatorname{mod}^T-\Dcal^\alpha_\hbar$. Then the natural functor $D^b(\operatorname{mod}^T-\Dcal^\alpha_\hbar)\rightarrow D^b_{coh}(\operatorname{Mod}^T-\Dcal^\alpha_\hbar)$ is an equivalence of triangulated categories.
\end{lemma}
\subsubsection{Image of $\mathcal{F}_\hbar$}\label{correctimage}
It follows from the above lemma that the derived functor $$\mathcal{F}_\hbar=\operatorname{RHom}(\Pro^\alpha_\hbar,\bullet): D^b(\operatorname{mod}^T-\Dcal^\alpha_\hbar)
\rightarrow D^+(\operatorname{Mod}^T-{\bf H}^{\wedge_\hbar}_{p\hbar})$$ exists,  where, again, $\operatorname{Hom}$ is taken
in the non-equivariant category.
\begin{lemma}
The image of $\operatorname{RHom}(\Pro^\alpha_\hbar,\bullet): D^b(\operatorname{mod}^T-\Dcal^\alpha_\hbar)
\rightarrow D^+(\operatorname{Mod}^T-{\bf H}^{\wedge_\hbar}_{p\hbar})$ lies in $D^b(\operatorname{mod}^T-{\bf H}_{p\hbar}^{\wedge_\hbar}) = D^b_{f.g.}(\operatorname{Mod}^T-{\bf H}^{\wedge_\hbar}_{p\hbar})$, where the subscript ``f.g'' means finitely generated cohomology.
\end{lemma}
\begin{proof}
It is enough to see that if $M\in \operatorname{mod}^T-\Dcal^\alpha_\hbar$ then  $R^i\operatorname{Hom}(\Pro^{\alpha}_\hbar, M)$ is a finitely generated ${\bf H}_{\p\hbar}^{\wedge_{\hbar}}$-module for each $i\geq 0$ and that it vanishes for $i>\!\!>0$, \cite[Corollary 2.68]{huy}.
Since any $M$ fits into a short exact sequence $0\rightarrow M_{\sf tor} \rightarrow M \rightarrow M/M_{\sf tor} \rightarrow 0$, the long exact sequence in cohomology shows it is enough to deal with the $\hbar$-torsion and $\hbar$-torsionfree cases separately.

Suppose first that $M$ is $h$-torsion. Since $M$ is finitely generated this means that $\hbar^kM=0$ for $k$ large enough. By induction we can assume that $\hbar M = 0$. It then follows that multiplication by $\hbar$ on $R^i\operatorname{Hom}(\Pro^{\alpha}_\hbar, M)$ is trivial. The short exact sequence $0\rightarrow \Pro^{\alpha}_\hbar \rightarrow \Pro^{\alpha}_\hbar \rightarrow \Pro^{\alpha}_\hbar/\hbar \Pro^{\alpha}_\hbar \rightarrow 0$ then produces short exact sequences $$0\rightarrow R^{i-1}\operatorname{Hom}(\Pro^{\alpha}_\hbar, M) \rightarrow R^{i}\operatorname{Hom}(\Pro^{\alpha}_\hbar/\hbar\Pro^{\alpha}_\hbar, M)\rightarrow  R^{i}\operatorname{Hom}(\Pro^{\alpha}_\hbar, M) \rightarrow 0.$$ Thus it is enough to show that $R^i\operatorname{Hom}(\Pro^{\alpha}_\hbar/\hbar\Pro^{\alpha}_\hbar, M)$ is finitely generated and vanishes for large enough $i$. Let $Z: \operatorname{Mod}^T-\Dcal_{\hbar}^{\alpha}\rightarrow \operatorname{Mod}^T-\mathcal{O}_X$ be the left exact functor which sends $M$ to $M^{\hbar} = \{m\in M: \hbar m = 0\}$. We have $R^1Z (M) = M/\hbar M$ and $R^iZ (M) = 0$ for $i\neq 0,1$. Since $Z$ preserves injectives and $\operatorname{Hom}(\Pro^{\alpha}_\hbar/\hbar\Pro^{\alpha}_\hbar, M) = \operatorname{Hom}_{\mathcal{O}_X}(\Pro^{\alpha}_\hbar/\hbar\Pro^{\alpha}_\hbar, Z(M))$ we have a convergent spectral sequence $$E_2^{i,j} = R^i\operatorname{Hom}_{\mathcal{O}_X}(\Pro^{\alpha}_\hbar/\hbar\Pro^{\alpha}_\hbar, R^jZ(M))\Rightarrow R^{i+j}\operatorname{Hom}(\Pro^{\alpha}_\hbar/\hbar\Pro^{\alpha}_\hbar, M).$$ But since both $R^iZ(M)$ and $\Pro^{\alpha}_\hbar/\hbar\Pro^{\alpha}_\hbar$ are coherent sheaves on $X$ all the cohomology groups on the left hand side are finitely generated over $\operatorname{End}(\Pro^{\alpha}_\hbar/\hbar\Pro^{\alpha}_\hbar) = \mathbb{C}[\h\oplus\h^*]\# W$ and so indeed over $\operatorname{End}(\Pro^{\alpha}_\hbar)$ and vanish for large enough $i$. Finite generation then follows for $R^{i+j}\operatorname{Hom}(\Pro^{\alpha}_\hbar/\hbar\Pro^{\alpha}_\hbar, M)$, as required.

For later use, we remark that if we replace $X$ above with some $T$-stable affine open set $U$ we see from the spectral sequence that $R^i\operatorname{Hom}(\Pro^{\alpha}_\hbar/\hbar\Pro^{\alpha}_\hbar|_U, M|_U)$ vanishes for $i\geq 2$ if $\hbar M = 0$. From this it follows that $R^i\operatorname{Hom}(\Pro^{\alpha}_\hbar|_U, M|_U)=0$ for $i\geq 1$ and then induction gives the same for any $\hbar$-torsion module $M$.

Suppose now that $M$ is $\hbar$-torsionfree. Then let $\Gamma (X, \bullet) : \operatorname{Mod}^T-\mathcal{E}nd(\Pro^{\alpha}_\hbar) \rightarrow \operatorname{Mod}-\C$ be the global sections functor. Since $\Pro^{\alpha}_\hbar$ is locally free we have $$R^i\operatorname{Hom}(\Pro^{\alpha}_\hbar, M) \cong R^i\Gamma (X, \mathcal{H}om(\Pro^{\alpha}_\hbar, M))$$ where the right derived functor $R^i\Gamma$ is calculated using proalgebraic $T$-injective quasi-coherent $\mathcal{E}nd(\Pro^{\alpha}_\hbar)$-modules. Now we claim that these derived functors can be identified with the \v{C}ech cohomology groups of $\mathcal{H}om(\Pro^{\alpha}_\hbar, M)$ provided that we calculate on some $T$-stable affine cover $\{ U_i \}$ of $X$. By Leray's theorem and the fact that $X$ is separated it is enough to check that  $R^i \Gamma (U, \mathcal{H}om(\Pro^{\alpha}_\hbar, M)) = 0$ for any $T$-stable affine open subset $U$, see \cite[Chapter 3, Exercise 4.11]{har}. Since $\mathcal{H}om(\Pro^{\alpha}_\hbar, M)$ is $\hbar$-complete and separated, the argument of \cite[Lemma 2.12]{kashrouq} shows the vanishing holds provided that it holds for $\mathcal{H}om(\Pro^{\alpha}_\hbar, M)/\hbar \mathcal{H}om(\Pro^{\alpha}_\hbar, M)$. But since this an $\hbar$-torsion module, vanishing follows from the previous paragraph.

Now we calculate $R^i\operatorname{Hom}(\Pro^{\alpha}_\hbar, M)$ using \v{C}ech cohomology. Let $C^i_\hbar, Z_\hbar^i, B_\hbar^i$ ($C^i, Z^i, B^i$) denote the cochains, cocycles and coboundaries for \v{C}ech cohomology with respect to $\{ U_i \}$ of an $\hbar$-complete and separated coherent $\mathcal{E}nd(\Pro^{\alpha}_\hbar)$-module $\mathcal{M}$ (respectively of $\mathcal{M}/\hbar \mathcal{M}$). We assume that each $Z^i/B^i$ is finitely generated over $\operatorname{End}(\Pro^{\alpha}_\hbar)/\langle \hbar \rangle = \mathbb{C}[\h\oplus \h^*]\# W$. This assumption
holds if $\M/\hbar\M$ is finitely generated over $\mathcal{E}nd(\Pro)$ (and hence a coherent sheaf on $X$). Now $Z^i_\hbar$
is an $\hbar$-saturated submodule of $C^i_\hbar$, in other words $\hbar Z^i_\hbar=Z^i_\hbar\cap \hbar C^i_\hbar$. It follows that $$\frac{Z^i_\hbar}{B^i_\hbar+\hbar Z^i_\hbar}\cong \frac{Z^i_\hbar + \hbar C^i_\hbar}{B^i_\hbar+\hbar C^i_\hbar}\hookrightarrow \frac{Z^i}{B^i}$$ is finitely generated. Lift a set of generators for this space to $Z^i_\hbar$, say $z_1,\ldots,z_n$. These generate $Z^i_\hbar/B^i_\hbar$. Indeed, pick $z^0\in Z^i_\hbar$.
Then $z^0$ has the form $\sum_{i=1}^n a^0_i z_i+ d(x^0)+\hbar z^1$. Repeating this argument with $z^1$, $z^2$, and so on, and using $\hbar$-completeness and separation then gives the finite generation result. Vanishing is immediate for $i$ greater than the number of $T$-stable affine opens used in the definition of the \v{C}ech cohomology.
\end{proof}

In particular, this lemma implies that $\mathcal{F}_\hbar:D^b(\operatorname{mod}^T-\Dcal^\alpha_\hbar)\rightarrow
D^b(\operatorname{mod}^T-{\bf H}^{\wedge_\hbar}_{p\hbar})$ is well-defined.

\subsubsection{Homological properties of $\operatorname{mod}^{T}-{\bf H}^{\wedge_\hbar}_{p\hbar}$}
The category $\operatorname{mod}^{T}-{\bf H}^{\wedge_\hbar}_{p\hbar}$ has  enough
projective  objects. Indeed since it is equivalent to the category $\operatorname{mod}^T-{\bf H}_{p\hbar}$, each object admits an epimorphism from a graded free module, that is from the sum of modules obtained from  ${\bf H}_{p\hbar}$ by an action
shift. Hence the category has enough projectives. Moreover, since ${\bf H}^{\wedge_\hbar}_{p\hbar}$
is the formal deformation of an algebra with finite homological dimension, we see that
each object in $\operatorname{mod}^{T}-{\bf H}^{\wedge_\hbar}_{p\hbar}$ has a finite projective
resolution. It follows that the functor $\mathcal{G}_h: D^b(\operatorname{mod}^T-{\bf H}^{\wedge_\hbar}_{p\hbar})
\rightarrow D^b(\operatorname{mod}^T-\Dcal^\alpha_\hbar)$ is well-defined.

\subsubsection{Results of Bezrukavnikov and Kaledin} To prove that $\mathcal{F}_\hbar,\mathcal{G}_\hbar$ are mutually quasi-inverse equivalences we will use results of Bezrukavnikov and Kaledin, \cite{BeKa2}. They proved that the functors $\operatorname{RHom}(\Pro,\bullet): D^b(\Coh(X))\rightarrow D^b(\mods-\C[\h\oplus \h^*]\# W)$
and $\bullet\otimes^L_{\C[\h\oplus \h^*]\# W}\Pro: D^b(\mods-\C[\h\oplus \h^*]\# W)\rightarrow D^b(\Coh(X))$
are mutually quasi-inverse equivalences. Recall that $\Pro$ is $T$-equivariant. Then the two functors $\mathcal{F},\mathcal{G}$ induce functors between the equivariant categories $D^b(\Coh^{T}(X))$ and $D^b(\mods^{T}-\C[\h\oplus \h^*]\# W)$. These functors are equivalences since first the non-equivariant versions are and second they also commute with the endofunctors of the bigrading shift induced from the $T$-action.

\subsubsection{Adjunction morphisms}
We have the adjunction morphisms $\operatorname{id}\rightarrow  \mathcal{F}_\hbar\circ \mathcal{G}_\hbar, \mathcal{G}_\hbar\circ \mathcal{F}_\hbar\rightarrow \operatorname{id}$. The adjunction morphism $\operatorname{id}\rightarrow  \mathcal{F}_\hbar\circ \mathcal{G}_\hbar$   is an isomorphism. Indeed, taking a free object $X\in\operatorname{mod}^T-{\bf H}_{p\hbar}^{\wedge_\hbar}$ we see that $$\mathcal{F}_\hbar\circ \mathcal{G}_\hbar (X) = \operatorname{RHom}(\Pro^{\alpha}_\hbar, X\otimes_{{\bf H}_{p\hbar}^{\wedge_\hbar}}\Pro^{\alpha}_\hbar) \cong X\otimes_{{\bf H}_{p\hbar}^{\wedge_\hbar}} \operatorname{RHom}(\Pro^{\alpha}_\hbar,\Pro^{\alpha}_\hbar) \cong X$$ where the last isomorphism follows from (P$_\hbar$2) and the proof of Lemma \ref{correctimage} which shows that $\operatorname{R}^i\operatorname{Hom}(\Pro^{\alpha}_\hbar,\Pro^{\alpha}_\hbar)\cong \operatorname{Ext}^i(\Pro^{\alpha}_\hbar,\Pro^{\alpha}_\hbar)$. The general case follows by taking a projective resolution.

\subsubsection{Cones of adjunction morphisms}\label{SSS_cones}To see that $\mathcal{G}_\hbar\circ \mathcal{F}_\hbar\rightarrow
\operatorname{id}$ is a functor isomorphism, we need to check that for ${M}\in D^b(\operatorname{mod}^{T}-\Dcal^\alpha_\hbar)$
the cone of $\mathcal{G}_\hbar\circ \mathcal{F}_\hbar(M)\rightarrow M$  is isomorphic to zero, that is, exact.
Call this cone $N$.

Now consider triangulated functors \begin{align}\bullet|_{\hbar=0}&:=\C\otimes^L_{\C[[\hbar]]}\bullet: D^b(\operatorname{mod}^{T}-\Dcal^\alpha_\hbar)\rightarrow
D^b(\Coh^{T}(X)),\\\bullet|_{\hbar=0}&:=\C\otimes^L_{\C[[\hbar]]}\bullet: D^b(\operatorname{mod}^{T}-{\bf H}^{\wedge_\hbar}_{p\hbar})
\rightarrow D^b(\mods^T-\C[\h\oplus \h^*]\# W) \\
\bullet|^{\hbar=0}&:=\operatorname{RHom}_{\C[[\hbar]]}(\C,\bullet): D^b(\operatorname{mod}^{T}-\Dcal^\alpha_\hbar)\rightarrow
D^b(\Coh^{T}(X)),\\\bullet|^{\hbar=0}&:=\operatorname{RHom}_{\C[[\hbar]]}(\C,\bullet): D^b(\operatorname{mod}^{T}-{\bf H}^{\wedge_\hbar}_{p\hbar})
\rightarrow D^b(\mods^T-\C[\h\oplus \h^*]\# W).
\end{align}
We have $\bullet|_{\hbar=0}=\bullet|^{\hbar=0}[1]$.

In the proof of Lemma \ref{correctimage} we have seen that $\mathcal{F}(\bullet|^{\hbar=0})$ is the left derived functor of $\operatorname{Hom}_{\mathcal{O}_X}(\mathcal{P}, \ker_\hbar(\bullet))$, where $\ker_\hbar$ means the annihilator of $\hbar$. Similarly  $\mathcal{F}_h(\bullet)|^{h=0}$ is the left derived functor of $\ker_\hbar\operatorname{Hom}_{D^\alpha_\hbar}(\mathcal{P}^\alpha_\hbar,\bullet)$. But $\operatorname{Hom}_{\mathcal{O}_X}(\mathcal{P}, \ker_\hbar(\bullet))$ and $\ker_\hbar\operatorname{Hom}_{D^\alpha_\hbar}(\mathcal{P}^\alpha_\hbar,\bullet)$ are naturally isomorphic. Thus $\mathcal{F}(\bullet|^{\hbar=0}) \cong \mathcal{F}_\hbar(\bullet)|^{\hbar=0}$.
It follows by adjunction then that $\mathcal{G}(\bullet|_{\hbar=0}) \cong \mathcal{G}_\hbar(\bullet)|_{\hbar=0}$. As a result we have a commutative diagram
$$\begin{CD}
\mathcal{G}_\hbar\circ \mathcal{F}_\hbar(M) @>>>  M \\
@V \bullet|_{\hbar=0} VV @VV \bullet|_{\hbar=0} V \\
\mathcal{G}\circ \mathcal{F}(M|_{\hbar=0}) @>>> M|_{\hbar=0}
\end{CD}$$
From this it follows that ${N}|_{\hbar=0}$ is the cone of $\mathcal{G}\circ\mathcal{F}({M}|_{\hbar=0})\rightarrow {M}|_{\hbar=0}$. The latter is zero since $\mathcal{F}$ and $\mathcal{G}$ are mutually quasi-inverse. Thus we only need to show that if ${N}|_{\hbar=0}=0$
for an object ${N}\in D^b(\operatorname{mod}^{T}-\Dcal^\alpha_\hbar)$, then ${N}=0$.
\subsubsection{Derived Nakayama lemma}
Any object in $\operatorname{mod}^{T}-\Dcal^\alpha_\hbar$ has a resolution
by $\C[[\hbar]]$-flat objects, see Lemma \ref{SSS_flat_resol} which is independent of the present material.
So to prove our claim in the previous paragraph we represent
${N}_\hbar$ as a (bounded from the right) complex of $\C[[\hbar]]$-flat sheaves. For such a complex $N_\hbar$
applying the functor $\bullet|_{\hbar=0}$ is just the usual quotient by $\hbar$.
Since for sheaves exactness is a local property, our claim from the previous paragraph reduces to the following ``derived Nakayama
lemma''.

\begin{lemma}\label{Lem:derived_Nakayama}
Let $\ldots\rightarrow N_k\rightarrow\ldots\rightarrow N_1\rightarrow N_0\rightarrow 0$ be a complex of
$\C[[\hbar]]$-modules that are $\C[[\hbar]]$-flat, and also complete and separated with respect to the  $\hbar$-adic topology.
Suppose that the induced complex $\ldots\rightarrow N_k/\hbar N_k\rightarrow\ldots\rightarrow N_1/\hbar N_1\rightarrow N_0/\hbar N_0\rightarrow 0$
is exact. Then the original complex is exact.
\end{lemma}
\begin{proof}
It is easy to see that, thanks to the completeness and separation properties,
the surjectivity of $N_1/\hbar N_1\rightarrow N_0/\hbar N_0$ implies that
$N_1\rightarrow N_0$ is surjective too. Now we are going to argue by induction.
For this we will need to show that for the kernels $Z_\hbar,Z$ of $N_1\rightarrow N_0,
N_1/\hbar N_1\rightarrow N_0/\hbar N_0$ the natural homomorphism $Z_\hbar/\hbar Z_\hbar\rightarrow Z$
is an isomorphism. To prove this we repeat the argument in the last paragraph of the proof of
Lemma \ref{correctimage}.
\end{proof}

We have thus shown that \eqref{eq:equiv_McKay} is an equivalence.
\subsubsection{Properties of $\mathcal{F}_\hbar, \mathcal{G}_\hbar$}
We will need some other properties of the equivalences $\mathcal{F}_\hbar$ and $\mathcal{G}_\hbar$. They  are summarized in the following
lemma.

\begin{lemma}\label{Lem:der_equiv_propert}
Keep the notation from above.
\begin{enumerate}
\item The functors $\mathcal{F}_\hbar$ and $\mathcal{G}_\hbar$ are $\C[[\hbar]]$-linear.
\item The functors $\mathcal{F}_\hbar$ and $\mathcal{G}_\hbar$ preserve the subcategories
of all complexes whose cohomologies are $\hbar$-torsion.
\item Let $Y\subset (\h\oplus\h^*)/W$ be a $T$-stable closed subvariety.
Consider the full subcategories
$$D_Y^b(\operatorname{mod}^{T}-{\bf H}^{\wedge_\hbar}_{p\hbar})\subset D^b(\operatorname{mod}^{T}-{\bf H}^{\wedge_\hbar}_{p\hbar})\quad \text{and} \quad
D_{\pi^{-1}(Y)}^b(\operatorname{mod}^{T}-\Dcal^\alpha_\hbar)\subset D^b(\operatorname{mod}^{T}-\Dcal^\alpha_\hbar)$$ whose cohomology is supported on $Y$ and on $\pi^{-1}(Y)$ respectively. Then $\mathcal{F}_\hbar$ and $\mathcal{G}_\hbar$ restrict to equivalences between these subcategories.
\end{enumerate}
\end{lemma}
\begin{proof}
(1)  follows directly from the constructions of $\mathcal{F}_h$ and $\mathcal{G}_h$,
while (2) is a formal consequence of (1).

Let us prove (3). First of all, we claim that $M\in D^b(\operatorname{mod}^{T}-\Dcal^\alpha_\hbar)$ lies in
$D_{\pi^{-1}(Y)}^b(\operatorname{mod}^{T}-\Dcal^\alpha_\hbar)$ if and only if $M|_{\hbar=0}$ lies in
$D_{\pi^{-1}(Y)}^b(\operatorname{Coh}^{T}(X))$. Recall that $M$ can be represented as a bounded from the right
complex $\ldots\rightarrow M_k\rightarrow\ldots\rightarrow M_0\rightarrow 0$ of $\C[[\hbar]]$-flat sheaves.
Then $M|_{\hbar=0}$ is just the complex $\ldots\rightarrow M_k/\hbar M_k\rightarrow\ldots\rightarrow M_0/\hbar M_0\rightarrow 0$.
Clearly, $M\in D_{\pi^{-1}(Y)}^b(\operatorname{mod}^{T}-\Dcal^\alpha_\hbar)$ if and only if for any principal open affine
subset $U\subset (\h\oplus\h^*)/W$ the complex $\ldots\rightarrow M_k|_{\pi^{-1}(U)}\rightarrow\ldots\rightarrow M_0|_{\pi^{-1}(U)}\rightarrow 0$ is exact. So our claim becomes the following: the complex $\ldots\rightarrow M_k|_{\pi^{-1}(U)}\rightarrow\ldots\rightarrow M_0|_{\pi^{-1}(U)}\rightarrow 0$
is exact if and only if $\ldots\rightarrow M_k/\hbar M_k|_{\pi^{-1}(U)}\rightarrow\ldots\rightarrow M_0/\hbar M_0|_{\pi^{-1}(U)}\rightarrow 0$ is exact.
The ``if'' part follows from Lemma \ref{Lem:derived_Nakayama}, while the ``only if'' part is clear.
Our claim is proved. Let us remark that an analogous claim (with the same proof) holds in the category
$D^b(\operatorname{mod}^{T}-{\bf H}^{\wedge_\hbar}_{p\hbar})$.

Now recall from \ref{SSS_cones} that the functors $\bullet|_{\hbar=0}$ intertwines $\mathcal{F}_\hbar$
with $\mathcal{F}$ and $\mathcal{G}_\hbar$ with $\mathcal{G}$. So it is enough to prove that $\mathcal{F},
\mathcal{G}$ preserves the objects with appropriate supports. But these functors are local with respect to
$(\h\oplus\h^*)/W$, meaning that  for every $T$-stable principal open affine
subset $U\subset (\h\oplus\h^*)/W$ we have
\begin{eqnarray*}
 \mathcal{F}^U\circ \operatorname{Res}^{X}_{\pi^{-1}(U)}&=&\operatorname{Res}_{U}^{(\h\oplus\h^*)/W}\circ \mathcal{F},\\
\mathcal{G}^U\circ \operatorname{Res}_{U}^{(\h\oplus\h^*)/W}&=&\operatorname{Res}^{X}_{\pi^{-1}(U)}\circ \mathcal{G}.
\end{eqnarray*} Here $\mathcal{F}^U, \mathcal{G}^U$ are the analogs of $\mathcal{F},\mathcal{G}$ defined
for $U$ and $\pi^{-1}(U)$ instead of for $(\h\oplus\h^*)/W$ and $X$. For example, the functors in the first line are derived functors
of $\Hom(\Pro|_{\pi^{-1}(U)}, \bullet|_{\pi^{-1}(U)})$ and $\Hom(\Pro, \bullet)|_U$. These functors are isomorphic
because $U$ is a principal open affine subset.

The locality of $\mathcal{F},\mathcal{G}$ implies the preservation of supports property.
\end{proof}

We remark that the subcategories in (\ref{eq:equiv_McKay1}) have the form indicated in (3) above with $Y=\h/W\subset (\h\oplus\h^*)/W$. So $\mathcal{F}_\hbar$ is the equivalence required in \eqref{eq:equiv_McKay1}. Furthermore the fact that $\mathcal{F}_\hbar$ and $\mathcal{G}_\hbar$ preserve the filtrations by supports also follows since the subcategories in the filtration again have the form appearing in (3).

\subsection{Translations}\label{SUBSECTION_translations}
In this section we are going to establish equivalences in (\ref{eq:equiv_trans}). Our equivalences
are standard: they come from tensor products with quantizations of line bundles on $X$. This should be
a general fact, but we are only going to establish it for hamiltonian reductions, where the construction
is more explicit.

\subsubsection{Translation bimodule} Let $\alpha\in \z$ and $\chi$ be a character of $\GL(n\delta)$. In the notation of Subsection \ref{SUBSECTION_Ham} set $$\Dcal^{\alpha,\chi}_\hbar:=(\W_{\hbar, V^{ss}}/ \W_{\hbar,V^{ss}}\mu^{*}\gl(n\delta)_{\hbar\alpha})^{\GL(n\delta),\chi},$$
where in the right hand side the superscript $\chi$  means taking $\GL(n\delta)$-semiinvariants of weight $\chi$. It is clear from the construction that $\Dcal^{\alpha,\chi}_\hbar$ is a right $\Dcal^\alpha_\hbar$-module. Also there is an action of $\Dcal^{\alpha+\chi}_\hbar$ from the left commuting with $\Dcal^\alpha_\hbar$. This action comes from the observation that the differential of the $\GL(n\delta)$-action
coincides with the modified adjoint action of $\gl(n\delta)$: $\xi\mapsto \frac{1}{\hbar}\operatorname{ad}(\xi)$.
Thus tensoring provides a functor \begin{equation} \label{trans} \bullet \otimes_{\Dcal^{\alpha+\chi}_{\hbar}} \Dcal^{\alpha,\chi}_{\hbar}:
\operatorname{mod}^{T}-\Dcal^{\alpha+\chi}_\hbar\xrightarrow{\sim} \operatorname{mod}^{T}-\Dcal^{\alpha}_\hbar.
\end{equation}
\subsubsection{Invertibility} Since the
$\GL(n\delta)$-action is free, we see that $\Dcal^{\alpha,\chi}_\hbar/\hbar\Dcal^{\alpha,\chi}_\hbar$ coincides with
$\Str_X^\chi:=\Str_{\mu^{-1}(0)}^{\GL(n\delta),\chi}$, see for instance \cite[Proposition 2.4]{holland}. The latter is a line bundle on $X$. There is a homomorphism
\begin{equation}\label{eq:tens_prod}\Dcal^{\alpha+\chi,-\chi}_\hbar\otimes_{\Dcal^\alpha_\hbar}\Dcal^{\alpha,\chi}_\hbar\rightarrow \Dcal^{\alpha+\chi}_\hbar\end{equation}
of $\Dcal^{\alpha+\chi}_\hbar$-bimodules induced by multiplication. Modulo $\hbar$, this homomorphism coincides
with the isomorphism $\Str^{-\chi}_X\otimes_{\Str_X}\Str^\chi_X\rightarrow \Str_X$, so (\ref{eq:tens_prod})
is itself an isomorphism and hence \eqref{trans} is an equivalence. Thus we have shown \eqref{eq:equiv_trans}. Moreover, this equivalence preserves the subcategories $\OCat$ and the support filtration on those subcategories.

\subsubsection{$\C[[\hbar]]$-flat resolutions}\label{SSS_flat_resol}
Using the above translations we have the following result which has been used in \ref{Lem:derived_Nakayama}.
\begin{lemma}
Any object  $M\in \operatorname{mod}^T$-$\Dcal^\alpha_\hbar$ admits a surjection from a $\C[[\hbar]]$-flat object.
\end{lemma}
\begin{proof}
It is enough to show that there is $\chi$ such that $M^\chi:=M\otimes_{\Dcal^\alpha_\hbar}\Dcal^{\alpha+\chi,-\chi}_\hbar$
is generated by its global sections: for then there is a surjection $(\Dcal^{\alpha+\chi}_\hbar)^{\oplus n}\twoheadrightarrow
M^\chi$ and hence also a surjection $(\Dcal^{\alpha,\chi}_\hbar)^{\oplus n}\twoheadrightarrow M$. Let $M_0:=M/M\hbar$, a coherent sheaf on $X$. Then there is some
$\chi$ such that $M_0^\chi:=M_0\otimes_{\Str_X}\Str_X^\chi$ is generated by its global sections and has no higher 
cohomology. From the cohomology vanishing it follows that all the global sections of $M_0^\chi$
lift to sections of $M^\chi$. Indeed if $m_1, \ldots , m_r$ are local sections lifting a global section of $M_0^{\chi}$, then one finds $(m_i-m_j)_{ij}$ give a Cech $1$-cocyle which vanishes modulo $(\hbar)$. Cohomology vanishing ensures that there exists $(\alpha_i)_i$ such that $\hbar^{-1}(m_i-m_j) = \alpha_i - \alpha_j$ modulo $(\hbar)$, and so replacing $m_i$ by $m_i +\hbar \alpha_i$ produces a 1-cocycle which vanishes modulo $(\hbar^2)$. Iterating  this proves that the global sections of $M_0^{\chi}$ lift to $M^{\chi}$. Moreover, since they generate $M^\chi$ modulo $\hbar$, they generate $M^\chi$ itself.   
\end{proof}

\subsubsection{Integral difference}\label{int_diff_rev}
By \ref{cyclogener} and \ref{paramchange} we have that  $\alpha,\alpha'\in\z$ differ by a character $\chi$ of $\GL(n\delta)$
if and only if the corresponding parameters $p,p'$ have integral difference.

\subsection{Equivalences between categories $\OCat$}\label{SUBSECTION_derived_complet}
\subsubsection{Quotients by $\hbar-1$}
The results of the previous two subsections establish an existence of an equivalence (\ref{eq:equiv_O1}).
In this section we will prove that (\ref{eq:equiv_O1}) implies (\ref{eq:equiv_O5}).

First let us prove (\ref{eq:equiv_O2}). Consider the thick subcategories
\begin{align*}&D^b_{\hbar-tor}(\operatorname{mod}^{T}-{\bf H}^{\wedge_\hbar}_{p\hbar})\subset D^b(\operatorname{mod}^{T}-{\bf H}^{\wedge_\hbar}_{p\hbar}),\\
&D^b_{\OCat, \hbar-tor}(\operatorname{mod}^{T}-{\bf H}^{\wedge_\hbar}_{p\hbar})\subset D^b_\OCat(\operatorname{mod}^{T}-{\bf H}^{\wedge_\hbar}_{p\hbar})\end{align*}
consisting of all complexes whose cohomology is $\hbar$-torsion. Thanks to assertion (2) of Lemma \ref{Lem:der_equiv_propert},
it is enough to establish an equivalence
\begin{equation}\label{eq:equiv_O6}
D^b(\operatorname{mod}^{T}-{\bf H}^{\wedge_\hbar}_{p\hbar})/D^b_{\hbar-tor}(\operatorname{mod}^{T}-{\bf H}^{\wedge_\hbar}_{p\hbar})\rightarrow D^b(\operatorname{mod}^{T_1}-{\bf H}_{p})
\end{equation}
that preserves all required subcategories. In (\ref{eq:equiv_O6}) we may replace ${\bf H}^{\wedge_h}_{hp}$ with
${\bf H}_{hp}$ because of the category equivalence mentioned in the last paragraph of \ref{actionsonH}. We have $$D^b(\operatorname{mod}^{T}-{\bf H}_{p\hbar})/D^b_{\hbar-tor}(\operatorname{mod}^{T}-{\bf H}_{p\hbar}) \cong D^b(\operatorname{mod}^{T}-{\bf H}_{p\hbar}/\operatorname{mod}_{\hbar-tor}^{T}-{\bf H}_{p\hbar}),$$ while the isomorphism ${\bf H}_{p\hbar}/(\hbar-1)\cong {\bf H}_p$ gives rise to an equivalence $\operatorname{mod}^{T}-{\bf H}_{p\hbar}/\operatorname{mod}_{\hbar-tor}^{T}-{\bf H}_{p\hbar} \cong \operatorname{mod}^{T_1}-{\bf H}_p$. It follows that (\ref{eq:equiv_O6}) is the equivalence of triangulated
categories which induces an equivalence between the subcategories of all objects
whose cohomology have required supports. This establishes (\ref{eq:equiv_O2}).

\subsubsection{From complexes supported in $\OCat$ to complexes in $\OCat$}
The claim (\ref{eq:equiv_O4}) was essentially obtained by Etingof in \cite[Proof of Proposition 4.4]{etingofaffine}: he considered the ungraded setting. Let us provide details.

We need to produce a functor $D^b_{\OCat}(\mods^{T_1}-{\bf H}_p)\rightarrow D^b(\OCat_p^{T_1})$ that will
be quasi-inverse to $D^b(\OCat_p^{T_1})\rightarrow D^b_{\OCat}(\mods^{T_1}-{\bf H}_p)$.
Following the proof of  \cite[Proposition 4.4]{etingofaffine}, consider the algebra $\hat{\bf H}_p:={\bf H}_p\widehat{\otimes}_{S(\h)}\C[[\h^*]]$. We remark that this algebra is not noetherian.
The group $T_1$ naturally acts on $\hat{\bf H}_p$. Let $\tilde{\bf H}_p$ denote the subalgebra
of $T_1$-finite elements. There is a natural embedding ${\bf H}_p\hookrightarrow \hat{\bf H}_p$
and its image lies in $\tilde{\bf H}_p$.

Introduce the category $\operatorname{Mod}^{T_1}-\tilde{{\bf H}}_p$ of $T_1$-equivariant $\tilde{\bf H}_p$-modules that are quotients of a sum of copies of $\tilde{\bf H}_p$ (with grading shifts). Obviously, $\OCat_p^{T_1}$ is a Serre subcategory
of $\operatorname{Mod}^{T_1}-\tilde{{\bf H}}_p$.
By construction any $T_1$-stable submodule of a direct sum of $\tilde{\bf H}_p$'s (with grading shifts) lies in this category and so it is abelian. Consider the derived
category $D^b_{\OCat}(\operatorname{Mod}^{T_1}-\tilde{{\bf H}}_p)$.

Pick  a complex $M_\bullet\in D^b_{\OCat}(\mods^{T_1}-{\bf H}_p)$ and consider the complexes $\tilde{M}_\bullet:=\tilde{\bf H}_p\otimes_{\bf H}M_\bullet, \hat{M}_\bullet:=\hat{\bf H}_p\otimes_{{\bf H}_p}M_\bullet$.  We are going to prove that the natural morphism $M_\bullet\rightarrow \tilde{M}_\bullet$ is a quasi-isomorphism.

According to  the proof of \cite[Lemma 4.5]{etingofaffine}, $\operatorname{Tor}^i_{{\bf H}_p}(\hat{\bf H}_p, M)=0$
for any $M\in \OCat^{T_1}_p$ and any $i>0$. Since any complex in $D^b(\mods^{T_1}-{\bf H}_p)$ is quasi-isomorphic
 to a complex of graded free modules, it follows by the hypertor spectral sequence
that the natural morphism of complexes $M_\bullet\rightarrow \hat{M}_\bullet$ is a quasi-isomorphism.
Apply this to a complex $$M_\bullet=0\xrightarrow{d^k} M_k\xrightarrow{d^{k-1}}M_{k-1}\xrightarrow{d^{k-2}}\ldots$$ of graded free modules.
The claim $M_\bullet\rightarrow \hat{M}_\bullet$ is a quasiisomorphism is equivalent to $\ker \hat{d}^i=\operatorname{im} \hat{d}^i+\ker d^i$
and $\operatorname{im}d^i=\operatorname{im}\hat{d}^i\cap \ker d^i$. Still all the operators $d^i$ are $\C^\times$-equivariant, the last two equalities imply, respectively, $\ker \tilde{d}^i=\operatorname{im} \tilde{d}^i+\ker d^i$
and $\operatorname{im}d^i=\operatorname{im}\tilde{d}^i\cap \ker d^i$. But this again means that $M_\bullet\rightarrow \tilde{M}_\bullet$
is a quasiisomorphism. The claim from the previous paragraph is proved.

Any object in  $D^b_{\OCat}(\operatorname{Mod}^{T_1}-\tilde{{\bf H}}_p)$ can be represented as
a (perhaps bounded only from the right) complex of free $\tilde{\bf H}_p$-modules (with grading shifts).
But, as an object of $\operatorname{Mod}^{T_1}-\tilde{{\bf H}}_p$,
$\tilde{\bf H}_p$ is a projective limit of objects in $\OCat^{T_1}_p$. It follows
that any projective object in $\OCat^{T_1}$ is projective in $\operatorname{Mod}^{T_1}-\tilde{\bf H}_p$,
so the natural functor $D^b(\OCat_p^{T_1})\rightarrow D^b_{\OCat}(\operatorname{Mod}^{T_1}-\tilde{{\bf H}}_p)$
is an equivalence. Composing $\tilde{\bf H}_p\otimes_{{\bf H}_p}\bullet: D^b_{\OCat}(\mods^{T_1}-{\bf H}_p)\rightarrow
D^b_{\OCat}(\operatorname{Mod}^{T_1}-\tilde{{\bf H}}_p)$ with the equivalence $D^b_{\OCat}(\operatorname{Mod}^{T_1}-\tilde{{\bf H}}_p)
\rightarrow D^b(\OCat^{T_1}_p)$ we get a quasi-inverse functor to $D^b(\OCat_p^{T_1})\rightarrow D^b_{\OCat}(\mods^{T_1}-{\bf H}_p)$.

\subsubsection{Getting rid of $T_1$}
To prove (\ref{eq:equiv_O5})  observe that we now have an equivalence $$E: D^b(\OO_p^{T_1}) \stackrel{\sim}{\longrightarrow} D^b(\OO_{p'}^{T_1})$$ which
commutes with the grading shifts arising from the $T_1$-action. Let $\eta_p : \OO_p^{T_1} \rightarrow \OO_p$ be the forgetful functor. By \cite[Proposition 2.4]{GGOR} all objects in $\mathcal{O}_p$ can be given $T_1$-equivariant structure and we have that $$\operatorname{Ext}^i_{\OO_p}(\eta_p(X), \eta_p(Y)) = \bigoplus_{z\in\mathbb{Z}}\operatorname{Ext}^i_{\OO^{T_1}_p}(X, Y\langle z\rangle)$$ for any $X,Y\in \OO_p^{T_1}$ where $\langle z \rangle$ denotes the grading shift. Let $P$ be an equivariant lift to $\OO_p^{T_1}$ of a projective generator for $\mathcal{O}_p$ and set $T = E(P)$. The above observations show that $\eta_{p'}(T)$ is a tilting object of $\OO_{p'}$ whose endomorphism ring is Morita equivalent to $\OO_p$.
This produces the equivalence
in (\ref{eq:equiv_O5}) and thus Theorem \ref{der_main_thm} follows.

\section{Cherednik category $\mathcal{O}$ vs parabolic category $\mathcal{O}$ in type $A$}\label{parabolicequiv}
\subsection{Content of this section}\label{para_content}
The goal of this section is to relate the category $\OCat_{\kappa,{\bf s}}$ for the cyclotomic
rational Cherednik algebra to the parabolic category $\OCat$ in type $A$ proving Theorem C from
the introduction. We impose
certain restrictions on our parameters: $\kappa$ is supposed to be not rational, while ${\bf s}$
has to be ordered: $s_\ell>s_1>\ldots>s_{\ell-1}$ (below the inequalities are different, but this
is because ${\bf s}$ will have a different meaning).

In Subsection \ref{SUBSECTION_parab} we recall some information about the parabolic category $\OCat$
and state a weak version of our main result, Theorem \ref{SSS_parab_main_easy}. In the last subsection
\ref{SUBSECTION_parab_main} of this section we will prove a stronger version and deduce it various consequences including
a combinatorial classification of modules  with given support in the Cherednik category $\OCat$
and the Koszul property for this category.

Subsection \ref{SUBSECTION_Hecke} recalls the definitions of Hecke algebras. We need them because the two categories we are interested in are highest weight covers
for certain Hecke algebras: the usual cyclotomic  Hecke algebras on the Cherednik side, see \cite{GGOR}
and \cite{rouqqsch}, and the degenerate cyclotomic Hecke algebras on the parabolic side,
see \cite{BKschurweyl}.

Subsection \ref{SUBSECTION_hw} deals with highest weight covers. We first recall some general
theory, mostly due to Rouquier in \cite{rouqqsch}, including his theorem on the equivalence
of highest weight covers. Then we treat the {\sf KZ} functor that realizes the Cherednik category
$\mathcal{O}$ as a highest weight cover for the Hecke algebra modules. We recall the definition
of the functor as well as some easy properties that are mostly well-known and standard.
Then we recall the usual and degenerate cyclotomic $q$-Schur algebras which provide highest weight covers of the usual and degenerate cyclotomic $q$-Schur algebra, and we recall the work of Brundan-Kleshchev showing that parabolic category $\mathcal{O}$ is a highest weight cover of a degenerate
cyclotomic Hecke algebra. This last cover is known to be equivalent to the degenerate cyclotomic Schur algebra.

In Subsection \ref{SUBSECTION_reduction} we study the Cherednik categories $\OCat$
a bit further. We provide a reduction result and also prove that under our ordering
of the parameters the Cherednik category $\OCat$ is equivalent to the category
of modules over the cyclotomic $q$-Schur algebra (the first result of this sort
was due to Rouquier, \cite{rouqqsch}).

So  Subsections \ref{SUBSECTION_hw} and \ref{SUBSECTION_reduction} suggest that
to prove Theorem \ref{SSS_parab_main_easy} one needs to establish a highest weight equivalence
between the categories of modules over the cyclotomic $q$- and degenerate Schur algebras.
In Subsection \ref{SUBSECTION_iso_Schur} we will establish an isomorphism between the
two algebras and show that it gives rise to a highest weight equivalence. Our proof
is based on an isomorphism between the usual  and degenerate cyclotomic
Hecke algebra due to Brundan and Kleshchev, \cite{BKKLR}.

This result is enough to complete the proof of Theorem \ref{SSS_parab_main_easy}.
However, for our combinatorial computation of the supports we need some more
preliminaries related to $\gl_\infty$-actions. In Subsection \ref{SUBSECTION_Fock}
we recall higher level Fock spaces and the corresponding crystals. In Subsection \ref{SUBSECTION_categor}
we  recall  categorical $\gl_\infty$-actions on the module categories of Hecke algebras,
the Cherednik and parabolic categories $\OCat$.

\subsection{Parabolic category $\OCat$}\label{SUBSECTION_parab}
 In this subsection we  assume that ${\bf s} = (s_1, \ldots , s_{\ell})\in \Z^{\ell}$ is chosen so that $s_1 > s_2 > \cdots > s_{\ell}.$ We set $\Lambda_{\bf s} = \sum_{i=1}^{\ell} \Lambda_{s_i}$, a weight for $\gl_{\infty}$.

\subsubsection{Category $\tOpar{{\bf s},m}$}\label{SSS_tOpar} Pick $m\geq s_1-s_{\ell}+n$ and let $d_i = m + s_i - s_1$ for $1\leq i \leq \ell$. Fix a partition $\tau$ whose transpose equals $(d_1 \geq \cdots \geq d_\ell)$.
In the following example $n=\ell =3$, ${\bf s} = (7,5,4)$ and we choose $m=6$:
$$\Yvcentermath1
\tau  \leftrightarrow \, \yng(1,1,2,3,3,3)$$
Let $\mathfrak{p}$ be the standard parabolic subalgebra of $\mathfrak{g} = \gl_{|\tau|}(\C)$ with Levi the block diagonal matrices $\mathfrak{l} = \mathfrak{gl}_{d_1}(\C)\times \cdots \times \mathfrak{gl}_{d_{\ell}}(\C).$ Let $\tOpar{{\bf s},m}$ denote the category consisting of all finitely generated $\mathfrak{g}$-modules that are locally finite dimensional over $\mathfrak{p}$, semisimple over $\mathfrak{l}$, and have integral generalized central character, see \cite[\S 4]{BKschurweyl}.

Let $\rho = (0, -1, \ldots ,1-|\tau|)$. Recall that a $\tau$-tableau $A$ is a filling of the Young diagram of $\tau$ with integers.  For example
$$\Yvcentermath1 A = \young(8,6,56,454,333,222)$$
By reading the set of entries of a $\tau$-tableau $A$ from top to bottom and left to right we produce ${\bf a} = (a_1, a_2, \ldots , a_{|\tau|})$ and hence a weight for $\g$, with respect to the standard diagonal torus, given by $\tau(A) = {\bf a} - \rho = (a_1, a_2+1, \ldots , a_{|\tau|}+|\tau|-1)$. In the example we get $\tau(A) = (8,7,7,7,7,7,12,12,11,11,14,14,14)$. If the entries of $A$ are strictly increasing from bottom to top then we call it column strict and we denote the set of such column strict tableaux by $\col{\tau}$. Given $A\in \col{\tau}$ we construct the parabolic Verma module $$N(A) = U(\g)\otimes_{U(\mathfrak{p})} V(A),$$ where $V(A)$ is the finite dimensional irreducible representation of $\mathfrak{l}$ with highest weight $\tau(A)$. It has an irreducible head $L(A)$ with highest weight $\tau(A)$. The set $\{ L(A) : A\in \col{\tau}\}$ is a complete set of irreducible representations in $\tOpar{{\bf s},m}$.

\subsubsection{Category $\Opar{{\bf s}}(n)$}\label{SSS_parab_indep_m}
To $\lambda = (\lambda^{(1)}, \ldots , \lambda^{(\ell)})\in \prt{\ell}{n}$ we associate $A_{\lambda}\in \col{\tau}$ whose $j^{\text{th}}$ column and $i^{\text{th}}$ (from top) row $A_i^{(j)}$ has value $$\lambda^{(j)}_{d_j-d_1+i}+ s_1-i+1$$ (if $d_1\geqslant d_j+i$, then
the corresponding box is absent). In the example above $A = A_{\lambda}$ for $\lambda = ((1), (1^2), \emptyset)$.
It is easy to describe the set $\{A_\lambda| \lambda\in \mathscr{P}_\ell(n)\}$ on the level of tableaux.
Namely, let $A_0$ be the ``ground-state'' tableau corresponding to the empty multi-partition: its
entries in the $i$th row are all equal to $s_1+1-i$.
The set in interest consists of all tableaux $A$ such that $A-A_0$ has non-negative integral entries summing to $n$.

We let
$\Opar{{\bf s},m}(n)$
be the Serre subcategory of $\tOpar{{\bf s},m}(n)$ generated by the irreducible representations $\{ L(A_{\lambda}) : \lambda \in \prt{\ell}{n}\}$. By \cite[Lemma 4.1]{BKschurweyl} this is a sum of blocks
of $\tOpar{{\bf s},m}(n)$ and hence a highest weight category with standard modules $\{ N(A_{\lambda}) : \lambda \in \prt{\ell}{n}\}$.

\begin{lemma} Up to equivalence sending $N(A_\lambda)$ to $N(A_\lambda)$ the category
of $\Opar{{\bf s},m}(n)$ is independent of the choice of $m$.
\end{lemma}
A proof will be given in \ref{SSS_BK_functor}.
We will therefore drop the $m$, henceforth denoting the subcategory by $\Opar{{\bf s}}(n)$.

\subsubsection{Main theorem: weaker version}\label{SSS_parab_main_easy}
Recall the $\star$-involution on $\mathscr{P}_\ell(n)$, \ref{multiprtns}. Also
 for an $\ell$-tuple ${\bf s}$ set ${\bf s}^{\star} = (-s_{\ell-1}, -s_{\ell -2}, \ldots , -s_{1}, -s_\ell)$.
\begin{theorem}
Let $\kappa \in \C\setminus \mathbb{Q}$ and ${\bf s} = (s_1, \ldots , s_{\ell})\in \Z^{\ell}$ with $s_1 > s_2 > \cdots > s_{\ell}$.
Then there is an equivalence $\Upsilon_n : \mathcal{O}_{\kappa, {\bf s}^{\star}}(n) \longrightarrow \Opar{{\bf s}}(n)$
mapping $\Delta_{\kappa, {\bf s}^{\star}}(\lambda^{\star})$ to $N(A_\lambda)$ for each $\lambda\in \mathscr{P}_\ell(n)$. \end{theorem}
Such a result has been suggested in \cite[Remark 8.10(b)]{VV} where it features as a degenerate analogue of their main conjecture which applies to all $\kappa$.

\subsection{Hecke algebras}\label{SUBSECTION_Hecke}
\subsubsection{The general case}
Recall the notation from \ref{SSS:init_notation}.

Let $\{ {\bf q}_u\}$ be a set of indeterminates with $u\in U$ and set ${\bf k} = \mathbb{C}[\{ {\bf q}_{u}^{\pm 1}\}]$. Let $\h^{\text{reg}} = \h \setminus \bigcup_{H\in \mathcal{A}} H$, choose $x_0\in \h^{\text{reg}}$, and let $B_W = \pi_1 (\h^{\text{reg}} /W , x_0) $, the braid group of $W$. Let $T_H\in B_W$ be a  generator for the monodromy around $H$. The elements $T_H$ with $H$ running over a minimal set of reflection
hyperplanes generate $B_W$. Here ``minimal'' means the minimal subset of the reflecting hyperplanes such that the corresponding reflections
generate $W$.
Let $\mathcal{H}(W)$ be the Hecke algebra of $W$ over ${\bf k}$, the quotient of ${\bf k}[B_W]$ by the relations $$\prod_{0\leq j < e_H} (T_H - \zeta_{e_H}^j{\bf q}_{H,j}) = 0, $$  see for instance \cite[\S 4]{BMR}.

\label{a defn} Given any $\C$-algebra homomorphism $\Theta: {\bf k} \rightarrow R$, we will let $\mathcal{H}_{\Theta}(W)$ denote the specialised algebra $\mathcal{H}(W)\otimes_{\bf k} R$.

\subsubsection{The cyclotomic case}
\label{subsubcyclo}
Let $R$ be a $\C$-algebra with distinguished elements $q\neq 0$ and $Q_1, \ldots , Q_{\ell}$. Following \cite[\S 3.2]{humathas} we define $\mathcal{H}^R_{q, {\bf Q}}(n)$ as the $R$-algebra with generators the Jucys-Murphy elements $L_1, \ldots , L_n$ and $T_1, \ldots , T_{n-1}$ and relations for $1\leq i\leq n-1$ and $1\leq j\leq n$
\begin{align*}
(L_1-Q_1)\cdots (L_1-Q_{\ell}) & =  0, & L_iL_j &= L_jL_i, \\
(T_i+1)(T_i-q) &= 0, & T_iT_{i+1}T_i &=  T_{i+1}T_iT_{i+1},\\
T_iL_i+\delta_{q,1} &= L_{i+1}(T_i-q+1), & T_iL_j & =  L_j T_i \quad \text{if }i\neq j, j+1, \end{align*} \vspace{-0.8cm}
$$ T_iT_j = T_jT_i \quad \text{if } |j-i|>1 \text{ and } j<n. $$

When $q =1$ the algebra $\mathcal{H}^R_{1,{\bf Q}}(n)$ is known as the degenerate cyclotomic Hecke algebra. On the other hand let $\Theta: {\bf k} \rightarrow R$ be the homomorphism defined by $$\Theta ({\bf q}_{H_{i,j}^k,0}) = q, \Theta({\bf q}_{H_{i,j}^k,1}) = 1, \Theta({\bf q}_{H_{i},j}) = \zeta_{\ell}^{-j}Q_j$$ with $q\neq 1$: then $H^R_{q, {\bf Q}}(n) = \mathcal{H}_{\Theta}(G_{\ell}(n))$. This is because $\mathcal{H}_{\Theta}(G_{\ell}(n))$ can be presented as a quotient of the affine Hecke algebra $\mathcal{H}^{aff}_{q}(n)$.
Recall that the last algebra is generated by $T_1,\ldots, T_{n-1},X_1,\ldots,X_n$ subject to the following relations
\begin{equation}\label{eq:affine_Hecke_relations}
\begin{split}
&X_iX_j=X_jX_i,\\
&T_iX_j=X_jT_i, j\neq i,i+1,\\
&T_iX_iT_i=qX_{i+1},\\
&(T_i-q)(T_i+1)=0,\\
&T_iT_j=T_jT_i, |i-j|>1,\\
&T_iT_{i+1}T_i=T_{i+1}T_iT_{i+1}.
\end{split}
\end{equation}
There is a unique epimorphism $\mathcal{H}^{aff}_q(n)\twoheadrightarrow \mathcal{H}_{\Theta}(G_{\ell}(n))$ such that $T_i\mapsto T_i, i=1,\ldots,n-1, X_1\mapsto L_1$, see e.g. \cite[\S 2.4]{mathassurv}. This induces an isomorphism $\mathcal{H}_{\Theta}(G_{\ell}(n)) \rightarrow \mathcal{H}^R_{q, {\bf Q}}(n).$

\subsubsection{Choice of $q,Q_i$} \label{paramsHec} Let ${\bf s} = (s_1, \ldots , s_{\ell})\in \C^\ell$ and $\kappa\in \C$. Let $R= \C, q = \exp(2\pi\sqrt{-1}\kappa)$ and define $$Q_i = \begin{cases} \exp(2\pi\sqrt{-1}\kappa s_i) \quad & \text{if }q \neq 1, \\ s_i & \text{if }q = 1.\end{cases}$$ then we will write $\mathcal{H}_{q, {\bf s}}(n)$ for $\mathcal{H}^{\C}_{q, {\bf Q}}(n).$

\subsubsection{Involution}\label{Hecinvol}
There are compatible isomorphisms of Hecke algebras $\sigma : \Hecke_q^{aff}(n)\rightarrow \Hecke_q^{aff}(n), \mathcal{H}_{q,{\bf s}}(n) \rightarrow  \mathcal{H}_{q,{\bf s}^{\ast}}(n)$ induced by sending $T_i$ to $-qT_i^{-1}$ for $1\leq i \leq n-1$ and $X_i$ to $X_i^{-1}$,  $L_i$ to $L_i^{-1}$ (if $q\neq 1$) or $X_i$ to $-X_i$, $L_i$ to $-L_i$ (if $q=1$) for $1\leq i \leq n$. Recall that ${\bf s}^\ast$ was defined in \ref{SSS_parab_main_easy}.

\subsection{Highest weight covers}\label{SUBSECTION_hw}
\subsubsection{Definition}\label{SSS_hw_defn}
Let $\Theta: {\bf k}\rightarrow k$ be an algebra homomorphism. Let $\mathcal{C}$ be a highest weight category defined over $k$, see \cite[4.1]{rouqqsch} for the definition, and $F: \mathcal{C} \rightarrow \mathcal{H}_\Theta (W)\md$ an exact functor. We define the pair $(\mathcal{C}, F)$ to be a highest weight cover of $\mathcal{H}_{\Theta}(W)$ if $F$ is essentially surjective and the restriction of $F$ to the projective objects in $\mathcal{C}$ is fully faithful. Given this, we then call the pair $(\mathcal{C}, F)$ an $i$-faithful cover of $\mathcal{H}_{\Theta}(W)$ if $F$ induces isomorphisms $$\ext^j_{\mathcal{C}}(M,N) \stackrel{\sim}{\rightarrow} \ext^j_{\mathcal{H}_{\Theta}(W)}(FM, FN)$$ for $0\leq j \leq i$ and for all $M,N \in \mathcal{C}^{\Delta}$. Here and below
$\mathcal{C}^{\Delta}$ denotes the full subcategory in $\mathcal{C}$ consisting of objects with a filtration whose successive quotients are standard modules.

\subsubsection{Passing to fraction fields}\label{SSS_fraction} Now let ${k}$ be an integral domain and set $K= \quot({k})$ so that we have a morphism $\theta: {\bf k} \xrightarrow{\Theta} k \hookrightarrow K$. Assume that $\mathcal{H}_{\theta}(W)$ is split semisimple. Then a highest weight cover $F$ becomes an isomorphism after tensoring with $K$. In particular, it induces a bijection between the set of simples
in $\mathcal{C}\otimes_k K$ and simple $\mathcal{H}_{\theta}(W)$-modules. Since $\mathcal{C}\otimes_{k} K$ is still a highest weight
category, we have an ordering on the set of simple modules. The notion of compatibility of $F$ with a partial ordering on
the simple $\mathcal{H}_{\theta}(W)$-modules can be defined now in a natural way.

\subsubsection{Equivalences}\label{hwequivdef} Let $\mathcal{C}_1$ and $\mathcal{C}_2$ be highest weight categories defined over $k$, with sets of standard objects $\Delta_1$ and $\Delta_2$ respectively. A functor $F: \mathcal{C}_1 \rightarrow \mathcal{C}_2$ is an equivalence of highest weight categories if it is an equivalence of categories and there is a bijection $\phi: \Delta_1 \xrightarrow{\sim} \Delta_2$ and invertible $k$-modules $U_D$ for each $D\in \Delta_1$ such that $F(D) \cong \phi(D)\otimes U_D$ for $D\in \Delta_1$.

If $(\mathcal{C}_i, F_i)$ are highest weight covers of $\mathcal{H}_{\Theta}(W)$ for $i=1,2$, then $F: \mathcal{C}_1 \to \mathcal{C}_2$ is an equivalence of highest weight covers if it is an equivalence of highest weight categories such that $F_1 = F_2\circ F$.

\subsubsection{Rouquier's equivalence theorem}
The importance of the above definitions comes from the following theorem of Rouquier, \cite[Theorem 4.49]{rouqqsch}.

\begin{theorem}\label{Thm:Rouquier_main}
Let ${k},K,\theta$ be as in \ref{SSS_fraction} and assume  that $\mathcal{H}_{\theta}(W)$ is split semisimple. Fix two orders, $\leq_1$ and $\leq_2$, on $\irr{\mathcal{H}_{\theta}(W)}$. Suppose for $i=1,2$ we have $1$-faithful highest weight covers $F_i : \mathcal{C}_i \rightarrow  \mathcal{H}_{\Theta}(W)\md$ compatible with $\leq_i$. If the ordering $\leq_1$ is a refinement of the ordering $\leq_2$, then there is an equivalence of highest weight covers $\mathcal{C}_1 \stackrel{\sim}{\rightarrow} \mathcal{C}_2$ which induces the bijections $$\irr{\mathcal{C}_1\otimes_{k} K} \stackrel{\sim}{\rightarrow} \irr{\mathcal{H}_{\theta}(W)} \stackrel{\sim}{\leftarrow} \irr{\mathcal{C}_2\otimes_{k} K }$$
specified in \ref{SSS_fraction}.
\end{theorem}

In the remaining part of the subsection we will provide examples of highest weight covers.

\subsubsection{KZ functor}\label{KZappears}
Let $\hat{R}$ be the completion of ${\bf R} = \C[\{{\bf h}_u\}]$ at a maximal ideal corresponding to the point $\{ h_u \}\in \C^U$.
Consider $\hat{R}$ as a ${\bf k}$-algebra via the homomorphism that sends ${\bf q}_{H,j}$ to $\exp(2\pi \sqrt{-1} {\bf h}_{H,j})$. Thus for any homomorphism $\Psi : {\bf R} \rightarrow R$ which factors through $\hat{R}$ there is a corresponding homomorphism $\Theta : {\bf k}\rightarrow R$. Then, see \cite[\S 5]{GGOR}, there is an exact functor $${\sf KZ}_{\Psi}: \mathcal{O}_{\Psi} (\h, W)\rightarrow \mathcal{H}_{\Theta}(W)\md.$$
Namely, given a module in $\mathcal{O}_\Psi(\h, W)$ one can localize it to $\h^{reg}$ to get a $W$-equivariant $R\otimes D(\h^{reg})$-module
with regular singularities.

In the rest of the paper we impose the following assumption on $W$.

\begin{hypothesis}
The Hecke algebra $\Hecke(W)$ is flat over ${\bf k}$ of rank $|W|$.
\end{hypothesis}

This is conjectured to be true in general, \cite[p.178]{BMR}, and holds for $W= G_\ell(n)$, the case we will consider, \cite[Theorem 4.24]{BMR}.

Under this hypothesis the  corresponding local system happens to be not
only $R[B_W]$-module but even an $\mathcal{H}_{\Theta}$-module.
\subsubsection{Faithfulness of $KZ$}
\label{whencanweuserouquier}
Let us recall that \begin{enumerate}
\item if $\Psi: {\bf k} \rightarrow k$ is such that \begin{equation} \label{Rapplies}\Psi({\bf h}_{H,m})  - \Psi({\bf h}_{H,m'}) -
\frac{m - m'}{e_H}  \notin \Z   \text{ for all $H\in \Refl$ and $m\neq m'$}\end{equation} then $(\mathcal{O}_{\Psi}, {\sf KZ}_{\Psi})$ is a $0$-faithful highest weight cover of $\mathcal{H}_{\Theta}(W)$, \cite[Proposition 5.9]{GGOR};
\item if $k$ is a discrete valuation ring and $\Psi: {\bf k} \rightarrow k$ is such that \eqref{Rapplies} holds then $(\mathcal{O}_{\Psi}, {\sf KZ}_{\Psi})$ is a $1$-faithful highest weight cover of $\mathcal{H}_{\Theta}(W)$, \cite[Theorem 5.3]{rouqqsch}.
\end{enumerate}

We will say that a parameter  $(h_{H,j})\in \C^U$ that satisfies \eqref{Rapplies} is {\it faithful}.
A parameter $p\in \param_1$ is said to be faithful, if it can be represented by a faithful $(h_{H,j})$.
In the cyclotomic case of \eqref{parameterchoice}, the parameters are faithful if and only if \begin{equation}\label{eq:ss_Hecke1} \kappa\not\in \frac{1}{2}+\Z \text{ and }\kappa(s_i-s_j)\not\in \Z.
\end{equation}
In particular, if $\kappa\not\in \mathbb{Q}$ and ${\bf s}\in \Z^{\ell}$ then this is equivalent to all $s_i$ being distinct.

Let us now explain some properties of the ${\sf KZ}$ functor.

\subsubsection{KZ twist}\label{SSSection_KZ_twist}
One of the corollaries of Theorem \ref{Thm:Rouquier_main} is the existence of equivalences between distinct specialised category $\mathcal{O}(W)$'s. Let $(h_{H,j}), (h'_{H,j})\in \C^U$ be faithful parameters with corresponding specializations $\Psi_h: {\bf R} \rightarrow \C$ and $\Psi_{h'}: {\bf R} \rightarrow \C$. The corresponding specialised Hecke algebras $\mathcal{H}_{\Theta_h}(W)$ and $\mathcal{H}_{\Theta_{h'}}(W)$ are equal if $h_{H,j} - h'_{H,j}\in \Z$ for each $(H,j)\in U$. This allows us to compare $\mathcal{O}_{\Psi_h}(W)$ and $\mathcal{O}_{\Psi_{h'}}(W)$ via the corresponding ${\sf KZ}$ functors.

\begin{theorem}[\cite{opdam}, \cite{rouqqsch}, \cite{berestchalykh}, \cite{gorgri}]\label{Prop:lab_pres_equiv}
There is a group homomorphism $\gamma: \Z^U \rightarrow \text{Sym}(\irr{W})$ with the following property. If $(h_{H,j}), (h'_{H,j})\in \C^U$ are faithful parameters whose difference $(z_{H,j})$ belongs to $\Z^U$ and such that the partial ordering $<_{\Psi_h}$ on $\irr{W}$ is a refinement of $<_{\Psi_{h'}}$, then there is an equivalence of highest weight categories $\Sigma_{h,h'}: \mathcal{O}_{\Psi_h}(W)\stackrel{\sim}{\rightarrow}\mathcal{O}_{\Psi_{h'}}(W)$ such that $\Sigma_{h,h'}(\Delta_{\Psi_h}(\lambda)) \cong \Delta_{\Psi_{h'}}(\gamma_{(z_{H,j})}(\la)).$
\end{theorem}

\begin{proof}
This is an application of Rouquier's Theorem \ref{Thm:Rouquier_main}. The only thing that does not follow
from the Rouquier result is the construction of $\gamma$. It was first studied by Opdam in \cite[Corollary 3.8]{opdam} and was shown in general to be a group homomorphism by Berest and Chalykh in \cite[Corollary 7.12]{berestchalykh}; its relation with Rouquier's theorem is explained in \cite[\S 2]{gorgri}.
\end{proof}
The homomorphism $\gamma$ is called the ${\sf KZ}$-twist. As explained in \cite[\S 6]{opdam} it can be induced from an action of the Galois group $Gal(L/K)$ on $\irr{W}$ where $K = \C(\{ {\bf q}_u\})$, the quotient field of the ring over which $\mathcal{H}(W)$ is defined, and $L$ is a splitting field for $\mathcal{H}(W)$. In particular, since $K$ is already a splitting field for $\mathcal{H}(G(\ell,1 ,n))$ by \cite[Theorems 3.7 and 3.10]{ArikiKoike} we see that $\gamma$ is trivial in the case $W = G(\ell , 1, n)$. In general we only know that for a large enough $N$ the subgroup $N\cdot \Z^U$ in the kernel of $\gamma$.

\subsubsection{KZ vs $\Res$}\label{proposition:KZ_commut}
Let $\Psi:{\bf R}\rightarrow \C$ be a homomorphism, and
let $\Theta$ be constructed from $\Psi$ as in \ref{KZappears}. Fix a parabolic
subgroup $\underline{W}\subset W$.

Let $\underline{\mathcal{H}}_{\Theta}$ denote the Hecke algebra of $\underline{W}$
corresponding to the parameter restricted from $\Theta$. Then there is a natural
embedding $\underline{\mathcal{H}}_{\Theta}\rightarrow \mathcal{H}_{\Theta}$ such that $\mathcal{H}_\Theta$
becomes a free right $\underline{\mathcal{H}}_{\Theta}$-module. So we can define the usual restriction $^{\mathcal{H}}\Res^{W}_{\underline{W}}:
\Lmod{\mathcal{H}_{\Theta}}\rightarrow \Lmod{\underline{\mathcal{H}}_\Theta}$ and induction $^{\mathcal{H}}\Ind_{\underline{W}}^W:
\Lmod{\underline{\mathcal{H}}_\Theta}\rightarrow \Lmod{\mathcal{H}_\Theta}$ functors. It turns out that the {\sf KZ} functors
intertwine both inductions and restrictions. More precisely, let $\underline{{\sf KZ}}: \OCat(\underline{W},\h^+)\rightarrow \Lmod{\underline{\Hecke}_\Theta}$
be the {\sf KZ} functor for $(\h_+,\underline{W})$.

\begin{proposition}[Theorem 2.1 and Corollary 2.3, \cite{Shan}]
We have isomorphisms of functors $^{\Hecke}\Res_{\underline{W}}^W\circ {\sf KZ}=\underline{{\sf KZ}}\circ \Res^{W}_{\underline{W}}$,
$^{\Hecke}\Ind_{\underline{W}}^W\circ \underline{{\sf KZ}}={\sf KZ}\circ \Ind^{W}_{\underline{W}}$.
\end{proposition}

A corollary of this proposition is, in particular, that the functors $\Res^{W}_{\underline{W}},\Ind_{\underline{W}}^W$
are biadjoint.

\subsubsection{KZ versus support filtrations}\label{Prop:supp_pres}
\begin{proposition}
Consider an arbitrary complex reflection group $W$ satisfying Hypothesis \ref{KZappears}.
Let $p,p'\in \param_1$ be such that the corresponding Hecke parameters coincide.
Let $\iota:\OCat_p\rightarrow \OCat_{p'}$ be an equivalence intertwining the {\sf KZ} functors.
Then $\iota$ preserves the support filtrations.
\end{proposition}
\begin{proof}
Recall the relationship between the supports of objects in $\OCat$ and the restriction
functors explained in \ref{SSS_res_support}. This relationship implies that we need
to prove the following claim: given a parabolic subgroup $\underline{W}$,
for an object $M\in \OCat_p$ we have $\Res_{\underline{W}}^W(M)=0$ if and only if
$\Res_{\underline{W}}^W(\iota(M))=0$. However, it is not convenient to work with the restriction
functors themselves because they associate different categories. Instead we are going to replace
$\Res_{\underline{W}}^W$ with $\mathcal{F}_?:=\Ind_{\underline{W}}^W\circ \Res_{\underline{W}}^W:\OCat_?\rightarrow \OCat_?$,
where $?=p$ or $p'$.
The claim that $\mathcal{F}_?(M)=0$ implies $\Res_{\underline{W}}^W(M)=0$
is a consequence of the fact that $\Ind_{\underline{W}}^W$ is right adjoint of $\Res_{\underline{W}}^W$.
We also remark that $\mathcal{F}_?$ maps projectives to projectives.

Now thanks to Proposition \ref{proposition:KZ_commut},
\begin{equation}\label{eq:KZ_comm2}\operatorname{\sf KZ}\circ \mathcal{F}_?=\,^\Hecke\! \mathcal{F}\circ \operatorname{\sf KZ},\end{equation}
where $\,^\Hecke\! \mathcal{F}$ is the analog of $\mathcal{F}_?$ for the Hecke algebra. Then  (\ref{eq:KZ_comm2})
 and \cite[Lemma 2.4]{Shan} imply $\iota\circ \mathcal{F}_p=\mathcal{F}_{p'}\circ \iota$. So $\mathcal{F}_p(M)=0$
 is equivalent to $\mathcal{F}_{p'}\circ \iota(M)=0$, and we are done.
\end{proof}

\subsubsection{Cyclotomic $q$-Schur algebras}\label{SSS_Schur_cyclo}
Now let us recall one more highest weight cover of the Hecke algebra modules.

 Let $R$ be a $\C$-algebra. Form $\mathcal{H}^R_{q, {\bf Q}}(n)$ as in \ref{subsubcyclo} where $q\neq 0$ and $Q_1, \ldots , Q_{\ell}$ are appropriate elements of $R$.
Given $\lambda \in \prt{\ell}{n}$, define $m_{\lambda}\in \mathcal{H}^R_{q,{\bf Q}}(n)$ as $$m_{\lambda} = \prod_{t=2}^{\ell} \prod_{k=1}^{|\lambda^{(1)}| + \cdots + |\lambda^{(t-1)}|} (L_k - Q_t) \cdot \sum_{w\in \mathfrak{S}_{\lambda}} T_w.$$
We call the $R$-algebra $$\mathcal{S}^R_{q, {\bf Q}}(n) = \edo_{\mathcal{H}^R_{q,{\bf Q}}(n)^{op}} \left(\bigoplus_{\lambda\in \prt{\ell}{n}} m_{\lambda}\mathcal{H}^R_{q,{\bf Q}}(n)\right)^{op}$$ the cyclotomic $q$-Schur algebra (degenerate cyclotomic $q$-Schur algebra if $q =1$). It is Morita equivalent rather than isomorphic
 to the cyclotomic $q$-Schur algebra of \cite{DJM} as they take a sum over multicompositions rather than multipartitions.

The category $\mathcal{S}^R_{q, {\bf Q}}(n)\md$ is a highest weight category with standard objects the Weyl modules $W^R_{q, {\bf Q}}(\lambda)$ for $\lambda\in \mathscr{P}_{\ell}(n)$, partially ordered by the dominance ordering, see for instance \cite[\S 4]{mathassurv}. There is a double centralizer property for $\mathcal{S}^R_{q, {\bf Q}}(n)$ and $\mathcal{H}^R_{q, {\bf Q}}(n)$ which produces an exact functor $$F_{q,{\bf Q}}^R : \mathcal{S}^R_{q, {\bf Q}}(n)\md \rightarrow \mathcal{H}^R_{q,{\bf Q}}(n)\md$$ called the cyclotomic $q$-Schur functor, \cite[\S 5]{mathassurv}. This functor is a highest weight cover. (Relevant details in the degenerate case can be found in \cite[\S 6]{AMR}.)

In the case $R=\C$ with ${\bf s} = (s_1, \ldots , s_{\ell})\in \C^{\ell}$,  we take $q,Q_1,\ldots,Q_\ell$
as in \ref{paramsHec}.
 We will write $\mathcal{S}_{q,{\bf s}}(n)$ for $\mathcal{S}^{\C}_{q,{\bf Q}}(n).$

Thanks to \cite[Proposition 4.40]{rouqqsch} and \cite[Theorem 6.18]{mattilt} and its degenerate analogue, $(\mathcal{S}_{q, {\bf s}}(n)\md, F_{q, {\bf s}})$ is a $0$-highest weight cover of $\mathcal{H}_{q , {\bf s}}(n)$ and $(\mathcal{S}_{1, {\bf s}}(n)\md, F_{1, {\bf s}})$ is a $0$-highest weight cover of $\mathcal{H}_{1,{\bf s}}(n)$.

Under certain conditions on the parameters the highest weight covers provided by the {\sf KZ} functor and by the Schur functors
are equivalent. The first result of this form is due to Rouquier, \cite[Theorem 6.8]{rouqqsch}. We will give
a version of this result in Proposition \ref{reductioncase},(1).

\subsubsection{BK functor}\label{SSS_BK_functor}
Let ${\bf s}$ be such that $s_1>s_2>\ldots>s_\ell$ and $n,m$ be such that $m\geqslant n+s_1-s_{\ell}$.

Brundan and Kleshchev, \cite{BKschurweyl}, construct a functor ${\sf BK} : \Opar{{\bf s},m}(n) \rightarrow \mathcal{H}_{1,{\bf s}}(n)\md$ and study its properties in detail. In particular,  $\mathcal{S}_{1, {\bf s}}(n)\md$ and $\Opar{{\bf s},m}(n)$ are equivalent categories by \cite[Theorem C]{BKschurweyl}. Call this equivalence $E$. It intertwines the Schur functor $F_{1,{\bf s}}: \mathcal{S}_{1, {\bf s}}(n)\md \rightarrow \mathcal{H}_{1,{\bf s}}(n)\md$ and the functor ${\sf BK}:  \Opar{{\bf s},m}(n)\rightarrow \mathcal{H}_{1,{\bf s}}(n)\md$ since both are given as $\Hom(M, -)$ where $M$ is a direct sum of all the projective-injective modules in either $\mathcal{S}_{1, {\bf s}}(n)\md$ or $\Opar{{\bf s},m}(n)$.

Both $(\mathcal{S}_{1, \bf{s}}(n)\md, F_{1, {\bf s}})$ and $(\Opar{{\bf s},m}(n), {\sf BK})$ are $0$-highest weight covers of $\mathcal{H}_{1,{\bf s}}(n)$. We've already observed this for the degenerate cyclotomic $q$-Schur algebra in \ref{SSS_Schur_cyclo}.
For $\Opar{{\bf s},m}(n)$ we use \cite[Proposition 4.40]{rouqqsch}, combined this time with \cite[Theorem 10.1]{stropquiver} and \cite[Lemmas 4.10 and 4.14]{BKdeganal}. Let $F_{1,{\bf s}}^!$ and ${\sf BK}^!$ denote the right adjoint functors to $F_{1, {\bf s}}$ and ${\sf BK}$ respectively. Then since $F_{1, {\bf s}}(W_1(\lambda)) \cong Sp_1(\lambda) \cong {\sf BK}(N(A_{\lambda}))$ by \cite[Theorem 6.12]{BKschurweyl}, we find using \cite[Proposition 4.40]{rouqqsch} that $$E(W_1(\lambda)) \cong EF_{1,{\bf s}}^{!}F_{1,{\bf s}}(W_1(\lambda)) \cong EF_{1, {\bf s}}^{!}(Sp_1(\lambda)) \cong {\sf BK}^{!}(Sp_1(\lambda)) = N(A_{\lambda}).$$ Thus $E$ is an equivalence of highest weight categories, sending $W_1(\lambda)$ to $N(A_{\lambda})$.

In particular, Lemma \ref{SSS_parab_indep_m} is now proved.

\subsection{Reduction for the Cherednik category $\OCat$}\label{SUBSECTION_reduction}
The general case of parameters $\tilde{\bf s}\in \C^{\ell}$ satisfying the faithfulness condition \eqref{eq:ss_Hecke1} may be reduced to the ``dominant integral" case.
\subsubsection{Equivalence on the parameters}
Put an equivalence relation on $\{1, \ldots , \ell\}$: $r$ and $r'$ are equivalent if there exists $a\in \Z$ such that $\kappa (\tilde{s}_r - \tilde{s}_{r'} - a)\in \Z$. By \cite[Remark 6.16]{rouqqsch} $\mathcal{O}_{\kappa, \tilde{\bf s}}(n)$ is equivalent to $$\bigoplus_{\stackrel{f: \{1,\ldots , \ell \}/\sim \rightarrow \Z_{\geq 0}}{\sum_I f(I) = n}} \bigotimes_{I\in \{1,\ldots , \ell \}/\sim} \mathcal{O}_{\kappa, (\tilde{s}_i)_{i\in I}}(f(I)).$$
 To study   $\mathcal{O}_{\kappa, \tilde{\bf s}}(n)$ we may therefore assume that for each $1\leq r, r' \leq \ell$ there exists integers $s_{rr'}, m_{rr'}$ such that $\kappa(\tilde{s}_r-\tilde{s}_{r'} -s_{rr'}) = m_{rr'}$. Since $\kappa\notin \mathbb{Q}$ the integers $s_{rr'}$ and $m_{rr'}$ are uniquely determined for all $r,r'$, so we have $m_{rr'} = m_{r1} - m_{r'1}$ and $s_{rr'} = s_{r1} - s_{r'1}$. It follows that we may reduce to studying category $\mathcal{O}_{\kappa, \tilde{\bf s}}(n)$ with $\tilde{s}_r = s_r+ \kappa^{-1}m_r$ where $s_r = s_{r1}$ and $m_r = m_{r1}$. We will refer to this as $\mathcal{O}_{\kappa, ({\bf s}, {\bf m})}(n)$.

\subsubsection{Reduction result} We will call ${\bf m}\in \Z^{\ell}$ dominant if $m_0 \geq m_1 \geq \cdots \geq m_{\ell-1}$ where $m_0 = m_{\ell}$. Combining the following proposition with above comments gives our reduction to the dominant integral case for $\kappa \notin \mathbb{Q}$. Recall the involutions ${}^\star$ defined on $\mathscr{P}_{\ell}(n)$, \ref{multiprtns}, and on parameters $\C^{\ell}$, \ref{SSS_parab_main_easy}.

\begin{proposition} \label{reductioncase} Assume $\kappa \in \C \setminus \mathbb{Q}$, ${\bf s}\in \Z^{\ell}$ has distinct entries and ${\bf m}\in \Z^{\ell}$ is dominant.
\begin{enumerate}
\item There is an equivalence of highest weight categories between $\mathcal{O}_{\kappa, ({\bf s}, {\bf m})}(n)$ and the cyclotomic $q$-Schur algebra $S_{q,{\bf s}^{\star}}(n)$ which sends $\Delta_{\kappa, ({\bf s}, {\bf m})}(\lambda)$ to $W_{q, {\bf s}^{\star}}(\lambda^{\star})$ for all $\lambda \in \mathscr{P}_{\ell}(n)$.
\item Let $w$ be a minimal length coset representative for the stabiliser of ${\bf m}$ in $\mathfrak{S}_{\ell}$. Then there is an equivalence of highest weight categories between $\mathcal{O}_{\kappa , ({\bf s}, {\bf m})}(n)$ and $\mathcal{O}_{\kappa, (w({\bf s}), w({\bf m}))}(n)$ which sends $\Delta_{\kappa , ({\bf s}, {\bf m})}(\lambda)$ to $\Delta_{\kappa, (w({\bf s}),w({\bf m}))}(w(\lambda))$ for all $\lambda \in \mathscr{P}_{\ell}(n)$. In addition, this equivalence preserves the support filtrations.
\end{enumerate}
\end{proposition}
Part (a) of the proposition is similar to \cite[Theorem 6.6]{rouqqsch}; the difference is that since we assume $\kappa$ is not rational we are able to give a slightly more generous set of parameters producing an equivalence with the cyclotomic $q$-Schur algebra. The proof will occupy the remainder of the subsection.

\subsubsection{An ordering}

The assumption that the entries of ${\bf s}$ are distinct ensures that the faithfulness conditions of \eqref{eq:ss_Hecke1} are satisfied, so we may apply Theorem \ref{Thm:Rouquier_main} provided there is  a compatibility on the orderings.

We begin by refining the ordering we wish to use on $\mathcal{O}_{\kappa, ({\bf s}, {\bf m})}(n)$. For this, we do not need to assume that ${\bf m}$ is dominant. We apply \cite[(4.2) and Theorem 4.1]{DG} to see that if $L_{\kappa, ({\bf s}, {\bf m})}(\lambda)$ and $L_{\kappa , ({\bf s}, {\bf m})}(\mu)$ belong to the same block of $\mathcal{O}_{\kappa, ({\bf s}, {\bf m})}(n)$ then there is an equality of multisets $$\{ \res^{\bf s}(b) : b\in \lambda\} = \{ \res^{\bf s}(b): b\in \mu\}.$$ Thanks to \eqref{cfunvalue}, we may therefore take an ordering on $\mathcal{O}_{\kappa, ({\bf s}, {\bf m})}(n)$ defined as follows: $$\lambda \succ_{{\bf s}, {\bf m}} \mu \text{ if } \{ \res^{\bf s}(b) : b\in \lambda\} = \{ \res^{\bf s}(b): b\in \mu\} \text{ and } \sum_{r=0}^{\ell-1}|\mu^{(r)}|(\ell m_r- r)  > \sum_{r=0}^{\ell-1}|\lambda^{(r)}|(\ell m_r- r).$$

\subsubsection{Proof of (1)} Let $R = \C[[t]]$ and let $\hat{\Psi}: {\bf R}\rightarrow R$ be defined by $\hat{\Psi} (h_{H_{i,j}^k,0}) = \kappa , \hat{\Psi}(h_{H_{i,j}^k,1}) = 0 ,$ and $\hat{\Psi}(h_{H_i, j}) = (t+\kappa)s_{j} - j/\ell + m_j$ for $0\leq j \leq \ell-1$, allowing us to define $\mathcal{O}_{\hat{\Psi}}$. By \ref{whencanweuserouquier}(2), this produces a $1$-highest weight cover of $\mathcal{H}^R_{q, {\bf Q}}(n)$ with $q = \exp(2\pi \sqrt{-1}\kappa)$ and $Q_j = \exp(2\pi\sqrt{-1}(t+\kappa)s_j)$. An integral version of the isomorphism $\sigma$ of \ref{Hecinvol} gives an isomorphism $\mathcal{H}^R_{q, {\bf Q}}(n)\cong \mathcal{H}^R_{q, {\bf Q}^{\star}}(n)$ where $Q^{\star}_j = \exp(2\pi\sqrt{-1}(-t-\kappa)s_{\ell-j})$ and so we may consider $(\mathcal{O}_{\hat{\Psi}}, {\sf KZ})$ as a $1$-highest weight cover of $\mathcal{H}^R_{q, {\bf Q}^{\star}}(n)$. Similarly \cite[Theorem 6.6]{rouqqsch} shows that $\mathcal{S}^R_{q,{\bf Q}^{\star}}(n)\md$ produces a $1$-highest weight cover of $\mathcal{H}^R_{q, {\bf Q}^{\star}}(n)$.  If $K = \C((t))$ and $\hat{\psi}: {\bf R} \stackrel{\hat{\Psi}}{\rightarrow} R \hookrightarrow K$ then we have $\irr{\mathcal{O}_{\hat{\psi}}} \stackrel{\sim}{\rightarrow} \irr{G_{\ell}(n)} \stackrel{\sim}{\leftarrow} \irr{\mathcal{S}^K_{q, {\bf Q}^{\star}}(n)}$ as both $\mathcal{O}_{\hat{\psi}}$ and $\mathcal{S}^K_{q, {\bf Q}^{\star}}(n)\md$ are Morita equivalent to $\mathcal{H}^K_{q, {\bf Q}^{\star}}(n)\md$ and the irreducibles of $H^K_{q, {\bf Q}^{\star}}(n)$ can be identified canonically with $\irr{G_{\ell}(n)}$ via Tits' deformation theorem. Under these identifications $W_{q, {\bf Q}^{\star}}^K(\lambda) = W_{q,{\bf Q}^{\star}}^R(\lambda)\otimes_R K$ is identified with the Specht module $Sp_{q, {\bf Q}^{\star}}^K(\lambda)$, which in turn is identified with $\lambda \in \irr{G_{\ell}(n)}$ since the Specht modules belong to a flat family over ${\bf k} = \C[\{{\bf q}_u^{\pm 1}\}]$ whose fibre at the special point ${\bf 1}$ corresponding to $q = 1$ and $Q_i = \zeta_{\ell}^i$ gives the $\mathcal{H}_{\bf 1}(G_{\ell}(n)) = \C [G_{\ell}(n)]$-representation $\lambda$. Similarly $\Delta_{\hat{\Psi}}(\lambda)\otimes_R K$ is identified with ${\sf KZ}_{\hat{\psi}}(\Delta_{\hat{\psi}}(\lambda))$ and this too belongs to a flat family whose fibre at $q=1$ and $Q_i = \zeta_\ell^i$ gives the $\mathcal{H}_{\bf 1}(G_{\ell}(n)) = \C [G_{\ell}(n)]$ representation $\lambda$. Under the isomorphism $\C[G_{\ell}(n)] \cong \C[G_{\ell}(n)]$ induced by $\mathcal{H}^R_{q, {\bf Q}}(n)\cong \mathcal{H}^R_{q, {\bf Q}^{\star}}(n)$ we have the generators $s_{H^0_{i,j}}$ sent to $-s_{H^0_{i,j}}$ and $s_{H_1}$ to $s_{H_1}^{-1}$, so we see that $\lambda$ is translated to $\lambda^{\star}$. Hence the identifications above match $\Delta_{\hat{\Psi}}(\lambda)\otimes_R K$ with $W_{q, {\bf Q}^{\star}}^R(\lambda^{\star})\otimes_R K$.

The ordering on $\mathcal{S}_{q, {\bf Q}^{\star}}^R(n)\md$ can be chosen to be the dominance ordering of $\prt{\ell}{n}$ intersected with block decomposition of $\mathcal{S}_{q,{\bf s}^{\star}}(n)$. By \cite[Theorem 2.11]{lylemat} this ordering is explicitly $$\lambda \succ_{{\bf s}^{\star}} \mu \text{ if and only if } \lambda \rhd \mu \text{ and } \{\res^{{\bf s}^{\star}}(b) : b\in\lambda\} = \{\res^{{\bf s}^{\star}}(b) : b\in \mu\} \text{ as multisets}.$$

Now we will show that $\lambda^{\star} \succ_{{\bf s}^{\star}} \mu^{\star}$ implies $\lambda\succ_{{\bf s}, {\bf m}} \mu$ for any $\lambda,\mu\in\prt{\ell}{n}$. As multisets $\{ \res^{\bf s}(b): b\in \lambda\} = \{ -\res^{{\bf s}^{\star}} (b): b\in \lambda^{\star}\}$ so the conditions involving these sets are clear. Now observe that \begin{eqnarray*}\sum_{r=0}^{\ell-1}(|\mu^{(r)}|-|\lambda^{(r)}|)(\ell m_r- r) &=& \sum_{i=1}^{\ell - 1}\left(\ell (m_{i} - m_{i-1}) -1\right)\left(\sum_{r=i}^{\ell -1} |\mu^{(r)}| - \sum_{r=i}^{\ell-1}|\lambda^{(r)}|\right).\end{eqnarray*} Since $m_{i-1}\geq m_i$ for $1\leq i \leq \ell -1$ by hypothesis, we see that $\lambda^{\star} \rhd \mu^{\star}$ results in $\sum_{r=0}^{\ell-1}|\mu^{(r)}|(\ell m_r- r)  \geq \sum_{r=0}^{\ell-1}|\lambda^{(r)}|(\ell m_r- r)$. If we have equality, then $|\lambda^{(r)}| = |\mu^{(r)}|$ for all $r$, so that for each $r$ we have $\lambda^{(r)}\unrhd \mu^{(r)}$ in the usual dominance ordering for partitions and $\lambda^{(r)} \rhd \mu^{(r)}$ for at least one value of $r$. Now consider the multisets $\{ \res^{{\bf s}}(b): b\in \lambda^{(r)} \text{ where } 1\leq r\leq \ell \text{ with }\lambda^{(r)}\neq \mu^{(r)}  \}$ and $\{ \res^{{\bf s}}(b): b\in \mu^{(r)} \text{ where } 1\leq r\leq \ell \text{ with }\lambda^{(r)}\neq \mu^{(r)}\}$. By hypothesis they are equal. However \cite[Lemma 4.2]{DG} shows that for each $r$ appearing in the definition of these multisets there exists a box $b\in \mu^{(r)}$ such that $\res^{\bf s}(b)$ is less than $\res^{\bf s}(b')$ for all $b'\in \lambda^{(r)}$. This contradicts equality. Hence we deduce that $\lambda^{\star} \succ_{{\bf s}^{\star}} \mu^{\star}$ implies that $\lambda \succ_{({\bf s}, {\bf m})} \mu$.

Now we may apply Theorem \ref{Thm:Rouquier_main} to produce the equivalence of highest weight categories between  $\mathcal{O}_{\kappa, ({\bf s}, {\bf m})}(n)$ and $S_{q, {\bf Q}^{\star}}(n)\md$ sending $\Delta_{\kappa, {\bf s}}(\lambda)$ to $W_{q, {\bf Q}^{\star}}(\lambda^{\star})$.

\subsubsection{Proof of (2)} We now know that we may take the ordering $(\;)^{\star} \succ_{{\bf s}}(\;)^{\star}$ on $\mathcal{O}_{\kappa, ({\bf s}, {\bf m})}(n)$. Adding arbitrary dominant ${\bf m}'$ to ${\bf m}$ does not change this ordering and so, by Theorem \ref{SSSection_KZ_twist}, we have an equivalence of highest weight categories between $\mathcal{O}_{\kappa, ({\bf s}, {\bf m})}(n)$ and $\mathcal{O}_{\kappa, ({\bf s}, {\bf m}+{\bf m}')}(n)$ which sends $\Delta_{\kappa, ({\bf s}, {\bf m})}(\lambda)$ to $\Delta_{\kappa, ({\bf s}, {\bf m}+{\bf m}')}(\lambda)$ for all $\lambda \in \mathscr{P}_{\ell}(n)$.  Choosing an appropriate $m$', we may (and will)
 assume that the hypothesis of Theorem \ref{SSS_sph_sym_main} is satisfied
 for the parameter corresponding to $(\kappa, {\bf s},{\bf m}+{\bf m'})$ and so by that result we find these categories are in turn equivalent to $\mathcal{O}_{\kappa, (w({\bf s}), w({\bf m}+{\bf m}'))}(n)$. Moreover, we may take the ordering on this category to be $$\lambda \succ'  \mu \text{ if and only if } w^{-1}(\lambda)^{\star} \rhd w^{-1}(\mu)^{\star} \text{ and } \{\res^{{\bf s}^{\star}}(b) : b\in w^{-1}(\lambda)\} = \{\res^{{\bf s}^{\star}}(b) : b\in w^{-1}(\mu)\} \text{ as multisets}.$$

We claim that $\lambda \succ' \mu$ implies $\lambda \succ_{(w({\bf s}), w({\bf m}))} \mu$. It is immediate that $ \{\res^{{\bf s}^{\star}}(b) : b\in w^{-1}(\lambda)\} =  \{-\res^{w({\bf s})}(b) : b\in\lambda\}$ so the conditions involving these sets are clear. Now \begin{eqnarray*}\sum_{r=0}^{\ell-1}(|\mu^{(r)}|-|\lambda^{(r)}|)(\ell m_{w(r)}- r) &=& \sum_{r=0}^{\ell - 1}(|w^{-1}(\mu)^{(r)}| - |w^{-1}(\lambda)^{(r)}|)(\ell m_r - w^{-1}(r)) \\ & = & \sum_{i=1}^{\ell - 1}\left(\ell (m_i - m_{i-1}) + (w^{-1}(i-1) -w^{-1}(i))\right)\sum_{r=i}^{\ell -1} (|w^{-1}(\mu)^{(r)}| - |w^{-1}(\lambda)^{(r)}|).\end{eqnarray*}
By hypothesis $m_{i-1}\geq m_i$ for $1\leq i \leq \ell -1$, and since $w$ was chosen to have minimal length, $m_{i-1} = m_i$ implies that $w^{-1}(i-1)< w^{-1} (i)$. Thus $\ell (m_i - m_{i-1}) + (w^{-1}(i-1) -w^{-1}(i))<0$ for $1\leq i\leq \ell -1$. It follows that if $w^{-1}(\lambda) \lhd w^{-1}(\mu)$ then $\sum_{r=0}^{\ell-1}|\mu^{(r)}|(\ell m_{w(r)}- r)  \geq \sum_{r=0}^{\ell-1}|\lambda^{(r)}|(\ell m_{w(r)}- r)$. Moreover, arguing as in the proof of (i) this inequality must be strict. Thus the claim follows.

By Theorem \ref{Thm:Rouquier_main} $\mathcal{O}_{\kappa, (w({\bf s}), w({\bf m}+{\bf m}'))}(n)$ is equivalent to $\mathcal{O}_{\kappa, (w({\bf s}), w({\bf m}))}(n)$. Moreover, Theorem \ref{SSSection_KZ_twist} implies that we can assume, in addition, that this equivalence
preserves the labeling of the standards. So we get an equivalence $\mathcal{O}_{\kappa, ({\bf s}, {\bf m})}(n)\xrightarrow{\sim}\mathcal{O}_{\kappa, (w({\bf s}), w({\bf m}))}(n)$ with the required behavior on standards.
It remains to check that this equivalence preserves the support filtrations. The equivalence
$\mathcal{O}_{\kappa, ({\bf s}, {\bf m}+{\bf m}')}(n)\xrightarrow{\sim} \mathcal{O}_{\kappa, (w({\bf s}), w({\bf m}+{\bf m}'))}(n)$
preserves the supports by Theorem \ref{SSS_sph_sym_main}. The remaining two intermediate equivalences commute
with ${\sf KZ}$ and so preserve the supports by Proposition \ref{Prop:supp_pres}.

 This completes the proof of (2).

\subsection{Isomorphisms of cyclotomic Schur algebras}\label{SUBSECTION_iso_Schur}

\subsubsection{Result} If $q\in \C$ is not a root of unity, there is an explicit $\C$-algebra isomorphism between $\mathcal{H}_{q, {\bf s}}(n)$ and $\mathcal{H}_{1, {\bf s}}(n)$, \cite[Corollary 2]{BKKLR}. Similarly there is an isomorphism between the degenerate and non-degenerate cyclotomic $q$-Schur algebras, although we couldn't find stated explicitly in the literature.

\begin{proposition}
\label{degenerateSchurIM}
Suppose that $q\in \C$ is not a root of unity. There is a $\C$-algebra isomorphism between $\mathcal{S}_{q, {\bf s}}(n)$ and $\mathcal{S}_{1, {\bf s}}(n)$. This isomorphism induces an equivalence between $\mathcal{S}_{q, {\bf s}}(n)\md$ and $\mathcal{S}_{1, {\bf s}}(n)\md$ which, for each $\lambda \in \prt{\ell}{n}$, sends the Weyl module $W_{q, {\bf s}}(\lambda)$ to the degenerate Weyl module $W_{1, {\bf s}}(\lambda)$.\end{proposition}

The proof of this proposition occupies the rest of the subsection.

\subsubsection{Isomorphism via KLR algebras}
Let $\Theta: \mathcal{H}_{1, {\bf s}}(n)\rightarrow \mathcal{H}_{q,{\bf s}}(n)$ denote the isomorphism of \cite[Corollary 2]{BKKLR}: it is the composition of the isomorphisms $\sigma: \mathcal{H}_{1, {\bf s}}(n)
\rightarrow R_n^{\Lambda_{\bf s}}$ of \cite[(3.49)]{BKKLR} and $\rho: R_n^{\Lambda_{\bf s}}\rightarrow \mathcal{H}_{q,{\bf s}}(n)$ of \cite[(4.47)]{BKKLR} which pass through the Khovanov-Lauda-Rouquier algebra $R_n^{\Lambda_{\bf s}}$. We wish to show first that under this isomorphism $m_{\lambda}\mathcal{H}_{1,{\bf s}}(n)$ is transferred to a module that is isomorphic to $m_{\lambda}\mathcal{H}_{q,{\bf s}}(n)$, in other words that $\Theta(m_{\lambda}\mathcal{H}_{1,{\bf s}}(n)) \cong m_{\lambda}\mathcal{H}_{q,{\bf s}}(n)$ as right $\mathcal{H}_{q, {\bf s}}(n)$-modules. The isomorphism between the cyclotomic $q$-Schur algebra and its degenerate version follows from this.

We will decompose $m_{\lambda }$ as $u_{\lambda} p_{\lambda}$ where $u_{\lambda}= \prod_{t=2}^{\ell} \prod_{k=1}^{|\lambda^{(1)}| + \cdots + |\lambda^{(t-1)}|} (L_k - Q_t)$ and $p_{\lambda} = \sum_{w\in \mathfrak{S}_{\lambda}} T_w.$ There are sets of pairwise orthogonal idempotents $\{ e({\bf j}): {\bf j}\in \Z^n \}$ in $\mathcal{H}_{1,{\bf s}}(n)$ and in $\mathcal{H}_{q, {\bf s}}(n)$, \cite[\S 3.1 and \S4.1]{BKKLR}: almost all $e({\bf j})$ are zero, and they sum to $1$. Thus $$\Theta(m_{\lambda} \mathcal{H}_{1, {\bf s}}(n))= \Theta\left(\bigoplus_{{\bf j}\in \Z^n} e({\bf j})u_{\lambda}p_{\lambda}\mathcal{H}_{1, {\bf s}}(n)\right) = \bigoplus_{{\bf j}\in \Z^n} {\Theta}(e({\bf j})u_{\lambda})\cdot {\Theta}(p_{\lambda}\mathcal{H}_{1, {\bf s}}(n)).$$  By definition we have $$\sigma(e({\bf j})(L_k - {s_t})) = e({\bf j})(y_k + j_k - s_t) \quad \text{and} \quad \rho^{-1}(e({\bf j})(L_k-q^{s_t})) = e({\bf j})q^{j_k}(1-q^{s_t-j_k} - y_k)$$ where $y_k$ is a nilpotent element in the generating set of $R_n^{\Lambda_s}$, see \cite[Lemma 2.1]{BKKLR}. Since the $e({\bf j})$'s commute with the $L_k$'s by construction, this shows that for all ${\bf j}\in \Z^n$ and all $\lambda \in \prt{\ell}{n}$, $\Theta(e({\bf j})u_{\lambda}) =  \alpha_{{\bf j}, \lambda} e({\bf j})u_{\lambda} $ where $\alpha_{{\bf j},\lambda}$ is an invertible element of $\mathcal{H}_{q, {\bf s}}(n)$ (the invertibility follows from the fact
 that all $y_k$ are nilpotent). Thus $${\Theta}(e({\bf j})u_{\lambda})\cdot {\Theta}(p_{\lambda}\mathcal{H}_{1, {\bf s}}(n)) = \alpha_{{\bf j}, \lambda}e({\bf j})u_{\lambda}\cdot {\Theta}(p_{\lambda}\mathcal{H}_{1, {\bf s}}(n)) \cong e({\bf j})u_{\lambda}\cdot {\Theta}(p_{\lambda}\mathcal{H}_{1, {\bf s}}(n)), $$ and we will now show that ${\Theta}(p_{\lambda}\mathcal{H}_{1, {\bf s}}(n)) = p_{\lambda}\mathcal{H}_{q, {\bf s}}(n)$ for each $\lambda \in \prt{\ell}{n}$.

Let $I_{\lambda}\subseteq \{ 1, \ldots , n-1\}$ be the indexing set for the simple reflections in $\mathfrak{S}_{\lambda}$. Observe first that \begin{equation}\label{intersection}p_{\lambda} \mathcal{H}_{q, {\bf s}}(n) = \bigcap_{i\in I_{\lambda}} (T_i+1)\mathcal{H}_{q, {\bf s}}(n).\end{equation} Indeed the equality $(T_i+1)(\sum_{w\in \mathfrak{S}_\lambda, s_iw>w} T_w) = p_{\lambda}$ shows the left hand side is contained in the right hand side. Conversely if $x = (T_i+1)y$ for some $y\in \mathcal{H}_{q, {\bf s}}(n)$ then $T_ix = qx$, and so $x\in \bigcap_{i\in I_{\lambda}} (T_i+1)\mathcal{H}_{q, {\bf s}}(n)$ implies that $(\sum_{w\in \mathfrak{S}_{\lambda}} T_w)x = \sum_{w\in\mathfrak{S}_{\lambda}} q^{\ell(w)} x$. Since $q$ is not a non-trivial root of unity, we see that $\sum_{w\in\mathfrak{S}_{\lambda}} q^{\ell(w)}$ is non-zero, and so the right hand side is contained in the left hand side.

It follows from \eqref{intersection} that it is enough to show that $$\Theta ((T_i+1)e({\bf j})  \mathcal{H}_{1, {\bf s}}(n)) =  (T_i+1)e({\bf j})  \mathcal{H}_{q, {\bf s}}(n)$$ for each $i\in I_{\lambda}$ and ${\bf j}\in \Z^n$. By \cite[(3.42) and (4.43)]{BKKLR} we have $$\sigma((T_i+1)e({\bf j})) = (1 + \psi_iq_i({\bf j}) - p_i({\bf i}))e({\bf j}), \quad \rho^{-1}((T_i+1)e({\bf j})) = (1 + \psi_iQ_i({\bf j}) - P_i({\bf i}))e({\bf j}),$$ where $\psi_i$ is a particular generator of $R_n^{\Lambda}$ and the elements $p_i({\bf j}), q_i({\bf j}),P_i({\bf j})$ and $Q_i({\bf j})$ are elements defined in \cite[(3.3) and (4.3)]{BKKLR}. There is some freedom of choice in the $q_i({\bf j})$ and $Q_i({\bf j})$. We will choose $q_i({\bf j})$ exactly as in \cite[(3.30)]{BKKLR}, but we will take $$Q_i({\bf j}) = \begin{cases} 1 - q + qy_{i+1}-y_i & \quad \text{if } j_i = j_{i+1} \\
\frac{q_i({\bf j})(1-P_i({\bf j}))}{1-p_i({\bf j})} & \quad \text{if } j_i \neq j_{i+1}. \end{cases}$$ (In this definition if $j_i = j_{i+1}+1$ then $1-p_i({\bf j}) = (y_i - y_{i+1})/(1+ y_i - y_{i+1})$ which is not invertible. In this case, however, $1-P_i({\bf j}) = q^{j_i}(y_i-y_{i+1})/(y_{i+1}({\bf j}) - y_i({\bf j}))$ and the expression for $Q_i({\bf j})$ simplifies to $q_i({\bf j})q^{j_i}(1+ y_i - y_{i+1})/(y_{i+1}({\bf j}) - y_i({\bf j}))$, where the denominator is now invertible since it equals $q^{j_{i+1}}(1-q) + (q^{j_i}y_i-q^{j_i+1}y_{i+1})$ where $y_i$ and $y_{i+1}$ commute and are nilpotent. Using \cite[(3.23), (3.28), (3.29) and (4.28)]{BKKLR}  one checks that these $Q_i({\bf j})$ satisfy the conditions \cite[(4.33) - (4.35)]{BKKLR} which they are required to.

With this in hand we now see that $$\Theta ((T_i+1)e({\bf j}))= \rho \sigma ((T_i+1)e({\bf j})) = (T_i+1)\frac{Q_i({\bf j})}{q_i({\bf j})} e({\bf j}).$$ Since both $Q_i({\bf j})$ and $q_i({\bf j})$ are invertible and commute with $e({\bf j})$ it follows that $\Theta ((T_i+1)e({\bf j})  \mathcal{H}_{1, {\bf s}}(n)) =  (T_i+1)e({\bf j})  \mathcal{H}_{q, {\bf s}}(n)$, as required.

\subsubsection{Isomorphism vs Weyl modules}

To see that the isomorphism we have chosen sends Weyl modules to degenerate Weyl modules, let $F_{q,{\bf s}}: \mathcal{S}_{q, {\bf s}}(n)\md \rightarrow \mathcal{H}_{q, {\bf s}}(n)\md$ and ${F_{1,{\bf s}}}: \mathcal{S}_{1,{\bf s}}(n)\md \rightarrow \mathcal{H}_{1 , {\bf s}}(n)\md$ be the cyclotomic $q$-Schur functor and its degenerate analogue. As we have mentioned in \ref{SSS_Schur_cyclo}, $(\mathcal{S}_{q, {\bf s}}(n)\md, F_{q, {\bf s}})$ is a $0$-highest weight cover of $\mathcal{H}_{q , {\bf s}}(n)$ and $(\mathcal{S}_{1, {\bf s}}(n)\md, F_{1, {\bf s}})$ is a $0$-highest weight cover of $\mathcal{H}_{1,{\bf s}}(n)$. As a result we have that
\begin{equation}\label{eq:Weyl_q} W_{q, {\bf s}}(\mu) \cong Hom_{H_{q, {\bf s}}(n)}( \bigoplus_{\lambda\in\prt{\ell}{n}} m_{\lambda}\mathcal{H}_{q,{\bf s}}(n), F_{q,{\bf s}}(W_{q, {\bf s}}(\mu)))\end{equation} and \begin{equation}\label{eq:Weyl_1}W_{1, {\bf s}}(\mu) \cong Hom_{H_{1, {\bf s}}(n)}( \bigoplus_{\lambda\in\prt{\ell}{n}} m_{\lambda}\mathcal{H}_{1, {\bf s}}(n), F_{1,{\bf s}}(W_{1, {\bf s}}(\mu))).\end{equation} However, by \cite[Proposition 2.17]{jamesmatjantzensum}, $F_{q, {\bf s}}(W_{q, {\bf s}}(\mu))$and $F_{1, {\bf s}}(W_{1, {\bf s}}(\mu))$ are isomorphic to the corresponding Specht modules labelled by $\mu$, $Sp_q(\mu)$ and $Sp_1(\mu)$. The result now follows from (\ref{eq:Weyl_q}) and (\ref{eq:Weyl_1}), the above paragraphs and the fact that $\Theta$ preserves Specht modules, \cite[Theorem 6.23]{kmr}.

\subsection{Fock spaces and crystals}\label{SUBSECTION_Fock}
We fix ${\bf s} = (s_1,\ldots , s_{\ell})\in \Z^{\ell}$ and set $\Lambda_{\bf s} = \Lambda_{s_1}+ \cdots + \Lambda_{s_{\ell}}$, a level $\ell$ weight for $\mathfrak{gl}_{\infty}$. We assume  $s_1>s_2>\ldots>s_\ell$.\subsubsection{Space and quantum group action}
 \label{combaction}
Given boxes $A$ and $B$ from a multipartition,
we say that $A$ is above $B$, respectively below,
if $A$ lies in a row that is strictly above or to the left of, respectively strictly below or to the right of, the row containing $B$ in the Young diagram when visualized as in (\ref{youngdiag}). A box $A\in\la$ is called removable (for $\la$) if $\la\setminus \{A\}$ has a shape of a multipartition. A box $B\not\in\la$ is called addable (for $\la$) if $\la\cup \{B\}$ has a shape of a multipartition. We use the notation
$
\la_A=\la\setminus \{A\}, \la^B=\la\cup\{B\}.
$
Given $\la \in \mathscr{P}_{\ell}$, a removable $i$-box $A$
and an addable $i$-box $B$,
we set
\begin{equation*}
\:\:\:\:\:\: d_i(\la)=\#\{\text{addable $i$-boxes of $\la$}\}
-\#\{\text{removable $i$-boxes of $\la$}\};
\end{equation*}
\begin{equation*}
\begin{split}
d_A(\la)= \#\{\text{addable $i$-boxes of $\la$ below $A$}\}\hspace{34.5mm}\\
-\#\{\text{removable $i$-boxes of $\la$ below $A$}\};\hspace{10mm}
\end{split}
\end{equation*}
\begin{equation*}
\begin{split}
d^B(\la)=\#\{\text{addable $i$-boxes of $\la$ above $B$}\}\hspace{34.5mm}\\
-\#\{\text{removable $i$-boxes of $\la$ above $B$}\}.\hspace{10mm}
\end{split}
\end{equation*}

Define $F(\La_{\bf s})$, the level $\ell$ Fock space attached to ${\bf s}$, to be the $\mathbb{Q}(v)$-vector space on basis
$\{M_\la\:|\:\la \in \mathscr{P}_{\ell}\}$
with quantum group $U_v(\mathfrak{gl}_{\infty})$-action defined by
\begin{equation}
\label{actionqgp}
E_i {M}_\la=\sum_A v^{d_A(\la)}{M}_{\la_A},\qquad
F_i {M}_\la=\sum_B v^{-d^B(\la)}{M}_{\la^B}, \qquad
K_i {M}_\la=q^{d_i(\la)}{M}_\la,
\end{equation}
where the first sum is over all removable $i$-boxes $A$ for $\la$, and the
second sum is over all addable $i$-boxes $B$ for $\la$.

\subsubsection{Dual canonical basis}\label{ssdcb}
The bar involution on $U_v(\mathfrak{gl}_{\infty})$ is the
automorphism $-:U_v(\mathfrak{gl}_{\infty}) \rightarrow U_v(\mathfrak{gl}_{\infty})$ that is
semilinear with respect to the field automorphism
$\mathbb{Q}(v) \rightarrow \mathbb{Q}(v), f(v) \mapsto f(v^{-1})$
and satisfies
\begin{equation*}
\overline{E_i} = E_i,\qquad
\overline{F_i} = F_i,\qquad
\overline{K_i} = K_i^{-1}.
\end{equation*}
There is a compatible bar involution on $F(\La_{\bf s})$: this is an anti-linear involution
$-:F(\La_{\bf s}) \rightarrow F(\La_{\bf s})$
such that $\overline{uf} = \overline{u}\,\overline{f}$
for each $u \in U_v(\mathfrak{gl}_{\infty}), f \in F(\La_{\bf s})$. It arises from the realization $F(\La_{\bf s}) \cong F(\La_{s_1})\otimes \cdots \otimes F(\La_{s_{\ell}})$ and a tensor product construction of Lusztig, see \cite[\S 3.9]{BKgraded}. It has the property that
\begin{equation*}
\overline{M_\la} = M_\la + \text{(a $\Z[v,v^{-1}]$-linear combination of
$M_\mu$'s for $\mu \lhd \la$)}.
\end{equation*}
The dual-canonical basis
$
\left\{L_\lambda\:\big|\:\lambda \in \mathscr{P}_{\ell}\right\}
$
of $F(\La_{\bf s})$ is defined by letting $L_\la$ be the unique bar-invariant vector
such that
\begin{equation}
\label{dualcan}
L_\lambda = M_\la + \text{(a $v \Z[v]$-linear combination of
$M_\mu$'s for various $\mu\lhd\la$)}.
\end{equation}
The polynomials $d_{\mu,\lambda}(v) \in \Z[v]$ defined from
\begin{equation}\label{KLpoly}
M_\la = \sum_{\mu \in \mathscr{P}_{\ell}} d_{\mu,\la}(v) L_\mu\quad
\end{equation}
satisfy $d_{\la,\la}(v) = 1$ and $d_{\mu,\la}(v) = 0$
unless $\mu \unlhd \la$. In \cite[Remark 14]{Bdual} there are explicit formulae expressing these polynomials in terms of finite type $A$ Kazhdan-Lusztig polynomials.

\subsubsection{Crystal} \label{sscg}
The actions of the Chevalley generators
$E_i$ and $F_i$ on the dual-canonical basis
of $F(\La_{\bf s})$ is reflected by an underlying crystal graph, $(\mathscr{P}_{\ell}, \tilde e_i, \tilde f_i, \eps_i, \phi_i)$, which has the following explicit combinatorial description.

Given an integer $i$, let $A_1,\dots,A_k$ denote the addable and removable $i$-boxes of $\la\in\mathscr{P}$ ordered so that $A_m$ is above $A_{m+1}$ for each $m=1,\dots,k-1$.
Consider the sequence $(\sigma_1,\dots,\sigma_k)$
where $\sigma_r = +$ if $A_r$ is addable or $-$ if $A_r$ is removable.
Step by step, we remove  pairs $(-,+)$ of subsequent elements
(in this order, i.e., a $-$ should precede a $+$).
At the end we are left with a sequence in which no $-$ appears to
the left of a $+$.
This
is called the reduced $i$-signature of $\la$. For the reduced signature $(\sigma_1,\ldots,\sigma_n)$ of $\lambda$ we have that
\begin{align*}
\eps_i(\la) = \#\{r=1,\dots,n\:|\:\sigma_r = -\},\qquad
&\phi_i(\la) = \#\{r=1,\dots,n\:|\:\sigma_r = +\}.
\end{align*}
Finally, if $\eps_i(\la) > 0$, we have that
$\tilde e_i \la = \la_{A_r}$
where $r$ indexes the leftmost $-$ in the reduced $i$-signature;
similarly, if $\phi_i(\la) >0 $ we have that
$\tilde f_i \la = \la^{A_r}$
where $r$ indexes the rightmost $+$ in the reduced $i$-signature.

The singular vectors $\lambda\in\mathscr{P}_{\ell}$ in the crystal can be therefore characterized as follows:
 for each $i\in\Z$, the sequence $(\sigma_1,\ldots,\sigma_k)$ of $i$-addable and $i$-removable boxes has the property that for each $j$ the number of $+$'s among $\sigma_j,\ldots,\sigma_k$ is not less than the number of $-$'s.

\iffalse
We note as a consequence of (\ref{act1})--(\ref{act3})
that the symmetric bilinear form $(.|.)$ on $F(\La)$
defined by $(M_\la | M_\mu) := \delta_{\la,\mu}$
for all $\la,\mu \in \Par$ has the property
that
\begin{equation}\label{anotherform}
(u v | w) = (v | \overline{\tau(u)} w)
\end{equation}
for all $u \in U_q(\g)$ and $v, w \in F(\La)$.
\fi

\subsection{$\gl_\infty$-categorifications}\label{SUBSECTION_categor}
Here we consider the case when $W=G_{\ell}(n)$,  $\kappa\in \C\setminus\mathbb{Q}$,
while the parameters $s_1,\ldots, s_{\ell}$ are integers.
\subsubsection{Categorification for Hecke algebras}
The subalgebra $\Hecke_{\kappa,{\bf s}}(n-1)
\subset \Hecke_{\kappa,{\bf s}}(n)$ is centralized by the image of  $X_n\in \Hecke^{aff}_q(n)$ in $\Hecke_{\kappa,{\bf s}}(n)$.
So the element $X_n$ defines an endomorphism $X$ of the restriction functor
$^{\Hecke}\!\Res_{W(n)}^{W(n-1)}:\Lmod{\mathcal{H}_{\kappa,{\bf s}}(n)}\rightarrow \Lmod{\mathcal{H}_{\kappa,{\bf s}}(n-1)}$
and of the induction functor $^{\Hecke}\!\Ind_{W(n-1)}^{W(n)}:\Lmod{\mathcal{H}_{\kappa,{\bf s}}(n-1)}
\rightarrow \Lmod{\mathcal{H}_{\kappa,{\bf s}}(n)}$.
By the work of Ariki, \cite{Ariki}, one knows the following:
\begin{itemize}
\item All eigenvalues of $X$ are of the form $q^i$. So it makes sense to consider the
generalized endofunctors $^{\Hecke}\!E_i(n),^{\Hecke}\! F_i(n)$ corresponding to the eigenvalue
$q^i$ so that $^{\Hecke}\!\Res_{W(n)}^{W(n-1)}=\bigoplus_{i\in \BZ} {^\Hecke\! E_i(n)}, ^{\Hecke}\!\Ind_{W(n-1)}^{W(n)}=
\bigoplus_{i} {^\Hecke\! F_i(n)}$.
\item The functors $^{\Hecke}\!E_i:=\bigoplus_{n\geqslant 0} ^{\Hecke}\!E_i(n), ^{\Hecke}\!F_i:=\bigoplus_{n\geqslant 0} ^{\Hecke}\!F_i(n)$
define a $\glinfty$-categorification in the sense of Rouquier of the category $\Lmod{\Hecke_{\kappa,{\bf s}}}:=\bigoplus_{n\geqslant 0}\Lmod{\Hecke_{\kappa,{\bf s}}(n)}$.
\item With respect to the operators $e_i:=[^{\Hecke}\!E_i], f_i:=[^{\Hecke}\!F_i]$ the $K$-group $[\Lmod{\Hecke_{\kappa,{\bf s}}}]$
is isomorphic to the irreducible module $L(\Lambda_{\bf s})$, where the highest weight $\Lambda_{\bf s}$ is constructed
from ${\bf s}$ as explained in the beginning of Subsection \ref{SUBSECTION_Fock}.
\end{itemize}

\subsubsection{Categorification for Cherednik algebras}
\label{indresCher}
According to \cite{Shan}, one has exact functors
$E_i(n):\OCat_{\kappa,{\bf s}}(n)\rightarrow \OCat_{\kappa,{\bf s}}(n-1), F_i(n):\OCat_{\kappa,{\bf s}}(n-1)\rightarrow
\OCat_{\kappa,{\bf s}}(n)$. The following holds:
\begin{itemize}
\item $\Res_{G(n)}^{G(n-1)}=\bigoplus_i E_i(n), \Ind_{G(n-1)}^{G(n)} = \bigoplus_i F_i(n)$.
\item We have ${\sf KZ}\circ E_i(n)=\,^{\Hecke}\!E_i(n)\circ {\sf KZ}, {\sf KZ}\circ F_i(n)=\,^{\Hecke}\! F_i(n)\circ {\sf KZ}$.
\item The functors $E_i,F_i$ define a $\glinfty$-categorification of $\OCat_{\kappa,{\bf s}}:=
\bigoplus_{n\geqslant 0} \OCat_{\kappa,{\bf s}}(n)$.
\item With respect to the induced operators $e_i,f_i$ the $K$-group $[\OCat_{\kappa,{\bf s}}]$ is isomorphic
to the Fock space $F(\Lambda_{\bf s})$ with multi-charge ${\bf s}$. Moreover, the {\sf KZ} functor induces a natural projection from
the Fock space to $L(\Lambda_{\bf s})$.
\end{itemize}

\subsubsection{Categorification for parabolic category $\OCat$}\label{SSS_parab_cat}
There are $i$-induction and $i$-restriction endofunctors of  $\Opar{\bf s}:=\bigoplus_{n\geqslant 0}\Opar{\bf s}(n)$
essentially defined in \cite[7.4]{ChuangRouquier} and studied further  in \cite[\S 4.4]{BKshift}. Let us recall their construction.  We use the notation of Subsection \ref{SUBSECTION_parab} and assume, in particular, that $s_1>\ldots>s_\ell$.

First we define certain functors on $\tOpar{{\bf s},m}$. Namely,  we have the functors $F:=\C^{|\tau|}\otimes\bullet,
E:=(\C^{|\tau|})^*\otimes\bullet$, where $\C^{|\tau|}$ stands for the tautological $\gl_{|\tau|}$-module. The functors
are biadjoint. We have endomorphisms $X$ of $E,F$ given by applying the tensor Casimir $\sum_{i,j}e_{ij}\otimes e_{ji}$,
where $e_{ij}$ denotes the matrix unit in position $(i,j)$. For $E_i,F_i$ we take the generalized eigenfunctors of $E,F$
corresponding to eigenvalues $-|\tau|-i$ and $i$, respectively. This defines a $\gl_\infty$-categorification on
$\tOpar{{\bf s},m}$, in particular, $E_i$ and $F_i$ are biadjoint.

According to \cite[Lemma 4.3]{BKschurweyl}, $F_iN(A)$ (resp., $E_i N(A)$) has a Verma flag whose
successive quotients are precisely $N(B)$ with $B$ being a column strict table  obtained from $A$ by replacing an $i$ with
an $i+1$
(resp., an $i+1$ with an $i$) for each occurrence of $i$ (resp., $i+1$) in $A$, every $N(B)$ occurs once.

The previous paragraph implies that $f_i$ maps $\Opar{{\bf s},m}(n)$ to $\Opar{{\bf s},m}(n+1)$.
Also the construction of ${\sf BK}$  in \cite[Section 5.3]{BKschurweyl} implies that ${\sf BK}$
intertwines $F$ with the induction functor for $\Hecke_{1,{\bf s}}$. Recall, see \ref{SSS_BK_functor},
that ${\sf BK}$ is fully faithful on projectives and so induces an isomorphism of $\End(F)$
with the endomorphisms of the induction functor.  This isomorphism preserves the endomorphisms
$X$ and so ${\sf BK}$ intertwines also the functors $f_i$. From here and an adaptation of
\cite[Lemma 2.4]{Shan} it follows that the functors $f_i$ do not depend on the choice of $m$
and so give rise to functors $\Opar{{\bf s}}(n)\rightarrow \Opar{{\bf s}}(n+1)$
to be also denoted by $f_i$.

On the contrary, $e_i$ does not need to map $\Opar{{\bf s},m}(n)$ to $\Opar{{\bf s},m}(n-1)$.
However,  this always holds for $m\gg 0$ (for each $i$ we need its own $m$).
The adjointness property of $e_i$ and $f_i$ gives rise to functors $e_i:\Opar{{\bf s}}(n)\rightarrow \Opar{{\bf s}}(n-1)$.
The functor ${\sf BK}$ intertwines the functors $e_i$ as well.
So we get a $\gl_\infty$-categorification on $\Opar{{\bf s}}:=\bigoplus_{n\geqslant 0}\Opar{{\bf s}}(n)$.

\subsection{Main theorem and applications}\label{SUBSECTION_parab_main}
\subsubsection{Main result: stronger version} Recall the notation $\mathcal{O}_{\kappa, ({\bf s}, {\bf m})}(n)$ from \ref{SUBSECTION_reduction} where ${\bf s}, {\bf m}\in \mathbb{Z}^{\ell}$.
\begin{theorem}
\label{parabolic}
Let $\kappa \in \C\setminus \mathbb{Q}$, ${\bf s} = (s_1, \ldots , s_{\ell})\in \Z^{\ell}$ with $s_1 > s_2 > \cdots > s_{\ell}$ and ${\bf m} = (m_1, \ldots , m_{\ell})$ with $m_{\ell}\geq m_1 \geq \cdots \geq m_{\ell-1}$.
\begin{enumerate}
\item There is an equivalence of highest weight categories $\Upsilon_n : \mathcal{O}_{\kappa, ({\bf s}^{\star}, {\bf m})}(n) \longrightarrow \Opar{{\bf s}}(n)$ and a commutative diagram $$\begin{CD}
\mathcal{O}_{\kappa, ({\bf s}^{\star}, {\bf m})}(n) @> \Upsilon_n > \sim >  \Opar{{\bf s}}(n) \\ @V {\sf KZ} VV @VV {\sf BK} V \\
\mathcal{H}_{q, {\bf s}^{\star
}}(n)\md @> \sim > {\tilde{\Theta}^*} > \mathcal{H}_{1, {\bf s}}(n)\md,
\end{CD}$$
where $\tilde{\Theta} =  \sigma\circ \Theta: \mathcal{H}_{1, {\bf s}}(n)\rightarrow \mathcal{H}_{q, {\bf s}^{\star
}}(n)$ with $\sigma$ the involution defined in \ref{Hecinvol}.
\item We have $\Upsilon_n(\Delta_{\kappa, ({\bf s}^{\star}, {\bf m})}(\lambda^{\star}))\cong N(A_{\lambda})$  for each $\lambda\in \prt{\ell}{n}$.
\item The equivalence $\Upsilon:=\bigoplus_{n=0}^\infty \Upsilon_n$ intertwines the functors of $i$-restriction and $i$-induction on $\mathcal{O}_{\kappa, ({\bf s}^{\star}, {\bf m})}$ with those of $(-i)$-restriction and $(-i)$-induction on $\Opar{{\bf s}}$ for all $i\in \Z$.
\end{enumerate}
\end{theorem}
We expect it is possible to obtain an analogous result on weakening the condition on ${\bf s}$ to that the $s_i$'s are pairwise distinct,
see Section \ref{SS_open_prob}. %This would probably require a generalization of the work of \cite{BKschurweyl} to find an appropriate analogue of \ref{SSS_BK_functor}, but we have not pursued this.
Taken with the results of \ref{SUBSECTION_reduction} this would deal with all faithful parameter choices with $\kappa \in \C\setminus \mathbb{Q}$.
\begin{proof}
(1) and (2).
Combine the equivalences from  \ref{SSS_BK_functor},\ref{reductioncase}(1) and \ref{degenerateSchurIM}.

(3) By Shan's results recalled in \ref{indresCher} ${\sf KZ}$ commutes with $i$-restriction and $i$-induction, while by Brundan and Kleshschev's results, see \ref{SSS_parab_cat}, so too does ${\sf BK}$.

Let $I: \mathcal{O}_{\kappa, {\bf s}^{\star}}(n)\prj \rightarrow  \mathcal{O}_{\kappa, {\bf s}^{\star}}(n)$ be the inclusion of the full subcategory of projective modules. Abusing notation, let $E_i$ and $F_i$ denote the $i$-restriction and $i$-induction functors on any relevant category. Since $\sigma(X_n) = X_n^{-1}$ we see that $\tilde{\Theta}^{\ast} E_i = E_{-i} \tilde{\Theta}^{\ast}$. We have ${\sf BK} E_i \Upsilon_n \cong E_i {\sf BK} \Upsilon_n \cong E_i \tilde{\Theta}^*{\sf KZ} \cong \tilde{\Theta}^* {\sf KZ} E_{-i} \cong {\sf BK} \Upsilon_{n-1} E_{-i}.$ Restricting to $\mathcal{O}_{\kappa, {\bf s}'}(n)\prj$ and noting that $\Upsilon_{n-1}, \Upsilon_n$ and $E_i$ preserve projectives and that ${\sf BK}^!{\sf BK} \cong \id$ on projectives, we see that $E_i \Upsilon_n I \cong \Upsilon_{n-1} E_{-i} I$. By \cite[Lemma 1.2]{Shan} this implies that $E_i \Upsilon_n \cong \Upsilon_{n-1} E_{-i}$. The argument for the $F_i$'s is similar.\end{proof}

\subsubsection{Decomposition numbers} \label{firstconsequence}
According to \cite{backelin}, the category $\tOpar{{\bf s},m}$ admits a $\Z$-grading that makes it Koszul.
Being the sum of blocks of $\tOpar{{\bf s},m}$, so does $\Opar{\bf s}(n)$.
It now follows from Theorem \ref{parabolic} that the category $\mathcal{O}_{\kappa , ({\bf s}^{\star}, {\bf m})}(n)$ admits
such a grading. Let ${\mathcal{O}}^{\sf gr}_{\kappa, ({\bf s}^{\star}, {\bf m})}(n)$ denote this category, with grading shift $\lan j \ran$ defined by $(M\lan j \ran)_i = M_{i-j}$ for $M\in {\mathcal{O}}^{\sf gr}_{\kappa, ({\bf s}^{\star}, {\bf m})}(n)$ and $i,j\in \Z$.

\begin{corollary} Let $\kappa\in \C\setminus \mathbb{Q}$ and ${\bf s} = (s_1, \ldots , s_{\ell})\in \Z^{\ell}$ with $s_1 > s_2 > \cdots > s_{\ell}$. Set $\mathcal{O}^{\sf gr}_{\kappa,({\bf s}^{\star}, {\bf m})} = \bigoplus_{n\geq 0}\mathcal{O}^{\sf gr}_{\kappa,({\bf s}^{\star}, {\bf m})}(n)$ and let $F_{\kappa , ({\bf s}^{\star}, {\bf m})}$ be the $\Z[v,v^{-1}]$-module $[ \mathcal{O}^{\sf gr}_{\kappa,({\bf s}^{\star}, {\bf m})} ] =  \bigoplus_{n\geq 0}[\mathcal{O}^{\sf gr}_{\kappa,({\bf s}^{\star}, {\bf m})}(n)]$, the direct sum of the corresponding Grothendieck groups.
\begin{enumerate}
\item Under the isomorphism $\Phi: F_{\kappa , ({\bf s}^{\star}, {\bf m})}\otimes_{\Z[v,v^{-1}]}\mathbb{Q}(v) \rightarrow F(\Lambda_{\bf s})$ induced by sending $[\Delta^{{\sf gr}}_{\kappa, ({\bf s}^{\star}, {\bf m})}(\lambda^{\star})]$ to $M_{\lambda}$ for all $\la \in \mathscr{P}_{\ell}$, we have $\Phi([L^{\sf gr}_{\kappa , ({\bf s}^{\star}, {\bf m})}(\la^{\star})]) = L_{\la}$, the dual-canonical basis element defined in \eqref{dualcan}.
\item In particular, $$\sum_j [{\sf rad}^j\Delta_{\kappa , ({\bf s}^{\star}, {\bf m})}(\la^{\star})/{\sf rad}^{j+1}\Delta_{\kappa , ({\bf s}^{\star}, {\bf m})}(\la^{\star}) : L_{\kappa, ({\bf s}^{\star}, {\bf m})}(\mu^{\star})]v^j = d_{\mu, \lambda}(v),$$ where $d_{\mu,\lambda}(v)$ is defined in (\ref{KLpoly}) and $$\Delta_{\kappa , ({\bf s}^{\star}, {\bf m})}(\la^{\star}) = {\sf rad}^0\Delta_{\kappa ,({\bf s}^{\star}, {\bf m})}(\lambda^{\star})\supseteq {\sf rad}^1\Delta_{\kappa , ({\bf s}^{\star}, {\bf m})}(\la^{\star})\supseteq {\sf rad}^2\Delta_{\kappa , ({\bf s}^{\star}, {\bf m})}(\la^{\star})\supseteq \cdots$$ is the radical filtration of $\Delta_{\kappa, ({\bf s}^{\star}, {\bf m})}(\la^{\star})$.
\end{enumerate}
\end{corollary}

\begin{proof}
After using Theorem \ref{parabolic} to transfer to the setting of $\Opar{\bf s}$, (1) is proved exactly as in \cite[Theorem 4.5]{BKshift}, where, of course, the variable $v$ keeps track of the grading. Part (2) is a standard Koszul fact, see for instance \cite[Theorem 3.11.4]{BGS}.
\end{proof}
Recently Stroppel and Webster have determined the decomposition numbers of the cyclotomic $q$-Schur algebras at roots of unity in terms of Fock space combinatorics, \cite{SW}.
\subsubsection{Support of irreducible modules} \label{secondconsequence}It is shown in \cite[Theorem 6.3]{Shan} that the set $\irr{\kappa, ({\bf s}^{\star}, {\bf m})} = \{ [L_{\kappa, ({\bf s}^{\star}, {\bf m})}(\la)]: \la\in\mathscr{P}_{\ell} \}$ and the operators $$\tilde{e}_i, \tilde{f}_i : \irr{\kappa, ({\bf s}^{\star}, {\bf m})}\rightarrow \irr{\kappa, ({\bf s}^{\star}, {\bf m})}\coprod \{ 0\}, \qquad \tilde{e}_i([L]) = [\text{head}(E_iL)]; \, \tilde{f}_i([L]) = [\text{soc}(F_iL)]$$ define a crystal that is isomorphic to the crystal of \ref{sscg} for $\Lambda_{{\bf s}^{\star}}$. It seems, however, to be complicated to write an explicit description of this isomorphism. But, thanks to Theorem \ref{parabolic} there is also an isomorphism from the representation theoretic crystal of \cite{Shan} to the crystal of $\Lambda_{\bf s}$ in which the roles of $\tilde{e}_i$ and $\tilde{f}_i$ are switched to that of $\tilde{e}_{-i}$ and $\tilde{f}_{-i}$ for all $i\in \Z$ and which just sends $[L_{\kappa, ({\bf s}^{\star}, {\bf m})}(\la^{\star})]$ to $\la$ for all $\la\in \mathscr{P}_{\ell}$.

For each $\lambda\in\mathscr{P}_{\ell}$ let $$N(\lambda) = \max \{ j\in \Z_{\geq 0}: \text{there exist } i_1, \ldots , i_j \in \Z \text{ such that }\tilde{e}_{i_1}\cdots \tilde{e}_{i_j} (\lambda) \neq 0\}.$$ Given a non-negative integer $m\leq n$ set $$X^n_{m} = G_n\cdot \{ (0, \ldots ,0, x_1, \ldots , x_{m})\}\subseteq \C^n.$$
It is known (and can be proved similarly to \cite[Remark 3.7]{shanvasserot}) that the support of any irreducible module in $\OCat_{\kappa,{\bf s}}$ is one of $X^n_m$.

\begin{corollary}  Let $\kappa\in \C\setminus \mathbb{Q}$ and ${\bf s} = (s_1, \ldots , s_{\ell})\in \Z^{\ell}$ with $s_1 > s_2 > \cdots > s_{\ell}$.
\begin{enumerate}
\item If $\lambda\in\prt{\ell}{n}$ then the support of $L_{\kappa, ({\bf s}^{\star}, {\bf m})}(\la^{\star})$ is $X^n_{N(\lambda)}$.
\item The module $L_{\kappa, ({\bf s}^{\star}, {\bf m})}(\lambda^{\star})$ is finite dimensional if and only if $N(\lambda) = 0$.
\end{enumerate}
\end{corollary}
\begin{proof}
(2) is a consequence of (1) which itself is a consequence of  \cite[Corollary 3.17]{shanvasserot}
 and the explicit identification above of the crystal $(\irr{\kappa, ({\bf s}^{\star}, {\bf m})}, \tilde{e}_i, \tilde{f}_i)$.
\end{proof}
This gives the labeling of all finite dimensional representations of ${\bf H}_{\kappa, ({\bf s}^{\star}, {\bf m})}(G_\ell(n))$ under our current restrictions on $\kappa$, ${\bf s}$ and ${\bf m}$. They are $L_{\kappa, ({\bf s}^{\star}, {\bf m})}(\la^{\star})$ where $\la\in\mathscr{P}_{\ell}$ is such that $\tilde{e}_i (\lambda) = 0$ for all $i\in \Z$. This set can be calculated then by the process of \ref{sscg}.

\subsubsection{Inner product on finite dimensional representations}
There is  a contravariant duality $d$ on ${\mathcal{O}}^{\sf gr}_{\kappa, ({\bf s}^{\star}, {\bf m})}(n)$ induced from the usual duality on parabolic category $\mathcal{O}$. It fixes the simple modules, and it has the property that $d(M\lan j\ran) = (dM)\lan -j \ran.$

For $M,N\in {\mathcal{O}}^{\sf gr}_{\kappa, ({\bf s}^{\star}, {\bf m})}(n)$ we set $$\lan M, N \ran_v = \sum_{i, j \in \Z} (-1)^i \dim \ext^i_{{\mathcal{O}}^{\sf gr}_{\kappa, ({\bf s}^{\star}, {\bf m})}(n)}(M, (dN)\lan j \ran)v^j.$$ This descends to a $\Z[v,v^{-1}]$-bilinear form on $[{\mathcal{O}}^{\sf gr}_{\kappa, ({\bf s}^{\star}, {\bf m})}(n)]$. For any $\lambda, \mu\in \prt{\ell}{n}$ we have \begin{eqnarray*} \sum_{j}(-v)^j &\dim &\ext^j_{{\bf H}_{\kappa, ({\bf s}^{\star}, {\bf m})}(G_n)}(L_{\kappa, ({\bf s}^{\star}, {\bf m})}(\la), L_{\kappa, ({\bf s}^{\star}, {\bf m})}(\mu)) \\ & = &\sum_{j}(-v)^j \dim \ext^j_{\mathcal{O}_{\kappa, ({\bf s}^{\star}, {\bf m})}(n)}(L_{\kappa, ({\bf s}^{\star}, {\bf m})}(\la), L_{\kappa, ({\bf s}^{\star}, {\bf m})}(\mu)) \\ &=& \sum_{i, j}(-v)^j \dim \ext^j_{\mathcal{O}^{\sf gr}_{\kappa, ({\bf s}^{\star}, {\bf m})}(n)}(L^{\sf gr}_{\kappa, ({\bf s}^{\star}, {\bf m})}(\la), L^{\sf gr}_{\kappa, ({\bf s}^{\star}, {\bf m})}(\mu)\lan i\ran) \\ & = & \sum_{j}(-v)^j \dim \ext^j_{\mathcal{O}^{\sf gr}_{\kappa, ({\bf s}^{\star}, {\bf m})}(n)}(L^{\sf gr}_{\kappa, ({\bf s}^{\star}, {\bf m})}(\la), L^{\sf gr}_{\kappa, ({\bf s}^{\star}, {\bf m})}(\mu)\lan j\ran) \\ & = & \lan L^{\sf gr}_{\kappa, ({\bf s}^{\star}, {\bf m})}(\la), L^{\sf gr}_{\kappa, ({\bf s}^{\star}, {\bf m})}(\mu) \ran_v.\end{eqnarray*} Here the first equality is proved in \cite[Proposition 4.4]{etingofaffine}, the third equality  is a consequence of $\mathcal{O}_{\kappa, ({\bf s}^{\star}, {\bf m})}(n)$ being Koszul and the fourth a consequence of Koszulity and $d(L^{\sf gr}_{\kappa, ({\bf s}^{\star}, {\bf m})}(\mu)) = L^{\sf gr}_{\kappa, ({\bf s}^{\star}, {\bf m})}(\mu)$. The above sequence of equalities shows that the following result confirms \cite[Conjecture 4.7]{etingofaffine}.

\begin{corollary} If $q\in \C^*$ is not a non-trivial root of unity then the restriction of the inner product $\lan - , - \ran_{v\mapsto q}$ to the finite dimensional representations in ${\mathcal{O}}^{\sf gr}_{\kappa, ({\bf s}^{\star}, {\bf m})}(n)$ is non-degenerate.
\end{corollary}

\begin{proof}
We abuse notation by letting $\lan -,-\ran_v$ also denote the sum of the forms on all $[{\mathcal{O}}^{\sf gr}_{\kappa, ({\bf s}^{\star}, {\bf m})}(n)]$. We have $$\lan \Delta^{\sf gr}_{\kappa, ({\bf s}^{\star}, {\bf m})}(\lambda), \Delta^{\sf gr}_{\kappa, ({\bf s}^{\star}, {\bf m})}(\mu)\ran_v =  \sum_{i, j \in \Z} (-1)^i \dim \ext^i_{{\mathcal{O}}^{\sf gr}_{\kappa, ({\bf s}^{\star}, {\bf m})}(n)}( \Delta^{\sf gr}_{\kappa, ({\bf s}^{\star}, {\bf m})}(\lambda), \nabla^{\sf gr}_{\kappa, ({\bf s}^{\star}, {\bf m})}(\mu)\lan j \ran)v^j = \delta_{\lambda, \mu}.$$ So thanks to Corollary \ref{firstconsequence} we can identify with $\lan -,-\ran_v$ with the bilinear form on an integral form of $F(\Lambda_{\bf s})$ which is defined by $$\lan M_{\lambda}, M_{\mu} \ran_v = \delta_{\lambda, \mu}.$$
Using \eqref{actionqgp} one checks that $\lan u x, y \ran_v = \lan x, \tau(u) y\ran_v$ for $x,y\in F(\Lambda_{\bf s}), u \in U_{v}(\mathfrak{gl}_{\infty})$ and where $\tau$ is the antiautomorphism of $U_v(\mathfrak{gl}_{\infty})$ defined by $$\tau (E_i) = vF_iK_i  , \quad \tau(F_i) = v^{-1}K_i^{-1}E_i, \quad  \tau(K_i) = K_i.$$

Let $F(\Lambda_{\bf s})^{\sf sing}_n$ denote the subspace of $F(\Lambda_{\bf s})_{n, {\sf sing}}$ annihilated by the $E_i$'s. As we have that $$F(\Lambda_{\bf s})_n = F(\Lambda_{\bf s})_{n,{\sf sing}} \oplus \sum_{i\in \Z}F_i (F(\Lambda_{\bf s})_{n-1})$$ and moreover $$\lan F(\Lambda_{\bf s})_{n,{\sf sing}} , F_i (F(\Lambda_{\bf s})_{n-1})\ran_v = \lan v^{-1}K_i^{-1}E_i (F(\Lambda_{\bf s})_{n,{\sf sing}}) , F(\Lambda_{\bf s})_{n-1} \ran_v = 0,$$
it follows that $\lan -, -\ran_v$ restricted to $F(\Lambda_{\bf s})_{n,{\sf sing}}$ is non-degenerate.

 By Corollary \ref{secondconsequence} $F(\Lambda_{\bf s})_{n, {\sf sing}}$ corresponds to the subspace of $[{\mathcal{O}}^{\sf gr}_{\kappa, ({\bf s}^{\star}, {\bf m})}(n)]$ spanned by the finite dimensional representations. Thus the Gram determinant of the restriction of $\lan -,-\ran_v$ to the finite dimensional representations is non-zero. Moreover, since the integrable representation theory of $U_{v}(\mathfrak{gl}_{\infty})$ is stable under specializing $v$ to $q\in \C^*$ (because $q$ is not a non-trivial root of unity), we see that $F(\Lambda_{\bf s})_{n, {\sf sing}}\otimes_{\Z [v,v^{-1}]} \C_q = (F(\Lambda_{\bf s})\otimes_{\Z [v,v^{-1}]} \C_q)_{n, {\sf sing}}$. Since $\lan - , -\ran_{v\mapsto q}$ is non-degenerate on $F(\Lambda_{\bf s})_n$, the same analysis as above shows that the form is non-degenerate on the finite dimensionals.\end{proof}
We end by noting that for $k\in \mathbb{Q}$ \cite[Remark 5.14]{shanvasserot} confirms this conjecture in the special case $q=1$, see also \cite[Conjecture 4.1]{etingofaffine}.
\section{Problems}\label{SS_open_prob}
Here we give a few problems related to the topics discussed in this paper. 
\medskip

\noindent
{\bf Problem 1: Conjecture \ref{sph_gen_main}}. Confirm it!
\medskip

\noindent
{\bf Problem 2: Cherednik $\mathcal{O}$ vs parabolic $\mathcal{O}$}.
Recall that for irrational $\kappa\notin \mathbb{Q}$ and distinct integral ${\bf s}=(s_1,\ldots,s_{\ell})$ and ${\bf m}=(m_1,\ldots,m_\ell)$ we have a reduction to the case ${\bf m}$ is dominant, Proposition \ref{reductioncase} . If $s_1>s_2>\ldots >s_{\ell}$ then we showed that the cyclotomic $\mathcal{O}_{\kappa, ({\bf s}^{\ast}, {\bf m})}$ is equivalent to an appropriate parabolic category $\mathcal{O}$ of type $A$. A non-trivial further generalization would be to remove this restriction on strict dominance for ${\bf s}$, showing that the cyclotomic category $\mathcal{O}_{\kappa, ({\bf s}^{\ast}, {\bf m})}$
is still isomorphic to an appropriate parabolic category $\mathcal{O}$
of type $A$. To prove this one  needs to generalize results from \cite{BKschurweyl} to more general
parabolic categories $\mathcal{O}$ then considered there.
\medskip

\noindent
{\bf Problem 3: Spherical categories $\mathcal{O}$ at aspherical parameters}.
Consider the spherical subalgebra $e{\bf H}_pe$ inside the Cherednik algebra
${\bf H}_p$. In Section \ref{SECTION_sphericalO} we have introduced the spherical category $\mathcal{O}_p^{sph}$ of
$e{\bf H}_pe$-modules. This category comes equipped with an exact functor $e:\mathcal{O}_p\rightarrow \mathcal{O}^{sph}_p, N\mapsto eN,$ that is a quotient functor. By considering supports, it is easy to show that the {\sf KZ} functor  factors through $e$.
Thus in the cyclotomic case the  category $\mathcal{O}_p^{sph}$ inherits a categorification from $\mathcal{O}_p$. It would be very interesting to understand the kernel of $e$, even at the level of Grothendieck groups.
\medskip

\noindent
{\bf Problem 4: Subcategories associated to singular irreducibles}.
Let $N$ be an irreducible module in $\mathcal{O}_p(n_0)$ that is annihilated by
all $i$-restriction functors. Consider the
Serre subcategory $\OCat_p^N$ in $\OCat_p$  that is spanned by the simples appearing
in the modules obtained from $N$ by all possible compositions of the $i$-induction functors.
It would be very interesting to understand the structure of this subcategory: for instance one could hope that $\OCat^N(n)$ is a highest weight category. One expects that the category is related to the finer structure of the categorification of the higher level Fock space. Using techniques from \cite{shanvasserot} we can show that it admits
a quotient functor (an analogue of the {\sf KZ} functor) to the category of modules over an appropriate
cyclotomic Hecke algebra.
\medskip

\noindent{\bf Problem 5: Definition of $\mathcal{O}$ at the microlocal level}. Let us use the notation of
Section \ref{SECTION_derived}. Consider the microlocal sheaf of algebras $\Dcal^\alpha$ whose
space of sections over a $\C^\times$-stable open subset $U$ is $\Dcal^\alpha(U):=\Dcal_\hbar^\alpha(U)_{T_2-fin}/(\hbar-1)$,
where the subscript ``$T_2$-fin'' means that one takes $T_2$-finite elements. Then the full subcategory of $\Dcal^\alpha$-$\mods^{T_1}$ consisting of objects supported on $\pi^{-1}(\h/W)$ is an analogue of category $\OCat$. We do not, however, know how to define an analogue of $\OCat$
inside the non-equivariant category $\Dcal^\alpha$-$\mods$. The support condition is insufficient:
we get too many objects in $\Dcal^\alpha$-$\mods$ supported on $\pi^{-1}(\h/W)$. We conjecture
that the correct category is defined by the support condition together with monodromy conditions
on  the local behaviour of an object near each irreducible component of $\pi^{-1}(\h/W)$.
One motivation for this problem comes from the study of the symplectic reflection algebras
associated to the wreath products $\Gamma_n$ of $S_n$ with a {\it non-cyclic} Kleinian group. In those cases
no category behaving like the category $\mathcal{O}$ is known (at least in characteristic $0$).
%One of the properties required from the analog is that the number of simples there should equal
%to the number of irreducibles in $\Gamma_n$.
%So one can try to introduce a suitable category over an analogue of $\Dcal^\alpha$ as the category
%of all modules with some support condition and some monodromy conditions (one needs to be careful:
%the number of irreducible components of any natural lagrangian subvariety of $X$ is either smaller or larger
%than the number of $\Gamma_n$-irreducibles).

\end{document}